\documentclass[11pt]{article}

\usepackage{hyperref}
\usepackage{fullpage}

\usepackage{amsmath,amsfonts,amsthm,amssymb,xspace,bm}
\usepackage{verbatim,dsfont,mathtools,algorithm,algorithmic,url}
\usepackage{booktabs}
\newcommand{\ra}[1]{\renewcommand{\arraystretch}{#1}}
\usepackage{multirow}
\usepackage{xfrac}
\usepackage[numbers]{natbib}
\usepackage{wrapfig}
\usepackage{varwidth}
\usepackage{colortbl}
\usepackage{epstopdf}
\usepackage{appendix}
\usepackage{enumitem}
\usepackage{pifont}
\usepackage[table]{xcolor}

\theoremstyle{plain}
\newtheorem{theorem}{Theorem}[section]
\newtheorem{lemma}[theorem]{Lemma}
\newtheorem*{lemma*}{Lemma}

\newtheorem{proposition}[theorem]{Proposition}
\newtheorem{corollary}[theorem]{Corollary}

\newtheorem{definition}[theorem]{Definition}

\theoremstyle{definition}

\newtheorem{remark}{Remark}

\newcommand{\R}{\mathbb{R}}

\DeclareMathOperator{\trace}{\textsc{Tr}}

\newcommand{\ip}[2]{\left\langle #1, #2 \right\rangle}

\newcommand{\norm}[1]{\left \Vert #1\right \Vert}
\newcommand{\dist}{{\rm{\textsc{Dist}}}}
\newcommand{\rank}{\mathrm{rank}}
\newcommand{\cmark}{\ding{51}}%

\newcommand{\obs}{y}
\newcommand{\linmap}{\mathcal{A}}

\newcommand{\plus}[1]{{#1}^{\scalebox{0.7}{+}}}

\newcommand{\tagterm}[2]{\underbrace{#2}_{#1}}
\newcommand{\algo}{\texttt{BFGD}\xspace}
\newcommand{\fgd}{\texttt{FGD}\xspace}

\newcommand{\C}{\mathcal{C}}
\newcommand{\X}{X}
\newcommand{\Xp}{\plus{X}}
\newcommand{\Xo}{X^{\star}}

\newcommand{\U}{U}
\newcommand{\W}{W}

\newcommand{\f}{f}

\newcommand{\gradf}{\nabla \f}

\newcommand{\Uo}{U^{\star}}
\newcommand{\Vo}{V^{\star}}

\newcommand{\weta}{\widehat{\eta}}

\def\Up{\plus{U}}
\def\Vp{\plus{V}}

\def\U{U}
\def\V{V}

\def\Y{Y}

\def\obs{y}
\def\linmap{\mathcal{A}}

\newcommand{\Ut}{\tilde{U}}
\newcommand{\Vt}{\tilde{V}}

\newcommand{\Uinit}{\U_0}
\newcommand{\Vinit}{\V_0}
\newcommand{\Xinit}{\X_0}

\newcommand{\Xh}{\hat{\X}}

\newcommand{\Us}{\U^\star}
\newcommand{\Vs}{\V^\star}
\newcommand{\Xs}{\X^\star}

\newcommand{\innerp}[2]{\left \langle #1, #2 \right\rangle}
\newcommand{\factor}{\U \V^\top}
\newcommand{\factoropt}{\Uo {\Vo}^\top}

\newcommand{\factorinit}{\U_0 \V_0^\top}

\newcommand{\normt}[1]{\left\|{#1}\right\|_2}
\newcommand{\normf}[1]{\left\|{#1}\right\|_F}

\newcommand{\gradU}{\nabla_U f(\factor)}
\newcommand{\gradV}{\nabla_V f(\factor)}

\newcommand{\gradW}{\nabla_W f(\factor)}

\newcommand{\proj}[1]{\mathcal{P}_{#1}}
\newcommand{\set}[1]{\mathcal{#1}}

\newcommand{\Xopt}{X^{\star}}
\newcommand{\Xoptr}{X^{\star}_r}

\def\Wo{\W^\star}
\def\Wp{\plus{\W}}

\newcommand{\Yo}{Y^\star}

\newcommand{\Wfac}{\begin{bmatrix} \U \\ \V \end{bmatrix}}

\newcommand{\Wofac}{\begin{bmatrix} \Uo \\ \Vo \end{bmatrix}}

\def\mumin{\mu_{\min}}

\def\fh{\hat{f}}

\newcommand{\xmark}{\ding{55}}%

\begin{document}

\title{Finding Low-Rank Solutions via Non-Convex Matrix Factorization, Efficiently and Provably}
\author{
 Dohyung Park, Anastasios Kyrillidis, Constantine Caramanis, and Sujay Sanghavi \\
 \vspace{0.2cm}
 The University of Texas at Austin \\
 \vspace{0.1cm}
 \{dhpark, anastasios, constantine\}@utexas.edu,
 sanghavi@mail.utexas.edu
}
\maketitle

\begin{abstract}%   <- trailing '%' for backward compatibility of .sty file
A rank-$r$ matrix $\X \in \R^{m \times n}$ can be written as a product $\U \V^\top$, where $\U \in \R^{m \times r}$ and $\V \in \R^{n \times r}$. 
One could exploit this observation in optimization: \textit{e.g.}, consider the minimization of a convex function $f(\X)$ over rank-$r$ matrices, where the set of rank-$r$ matrices is modeled via the factorization $\U\V^\top$. 
%Such a heuristic has been widely used for specific problem instances, where the solution sought is (approximately) low-rank. 
Though such parameterization reduces the number of variables, and is more computationally efficient (of particular interest is the case $r \ll \min\{m, n\}$), it comes at a cost: $f(\U \V^\top)$ becomes a non-convex function w.r.t. $\U$ and $\V$. 

We study such parameterization for optimization of generic convex objectives $f$, and focus on first-order, gradient descent algorithmic solutions. We propose the \emph{Bi-Factored Gradient Descent} (\algo) algorithm, an efficient first-order method that operates on the $U, V$ factors. We show that when $f$ is (restricted) smooth, \algo has local sublinear convergence, and linear convergence when $f$ is both (restricted) smooth and (restricted) strongly convex. For several key applications, we provide simple and efficient initialization schemes that provide approximate solutions good enough for the above convergence results to hold.
\end{abstract}

\section{Introduction}{\label{sec:intro}}
We study matrix problems of the form: %that can be described by the following optimization form:
\begin{equation}{\label{intro:eq_00}}
\begin{aligned}
	& \underset{X \in \R^{m \times n}}{\text{minimize}}
	& & f(X),
	%\\
	%& \text{subject to}
	%& & \rank(X) \leq r.
\end{aligned}
\end{equation} 
where the minimizer $\Xopt \in \R^{m \times n}$ is rank-$r^\star$ ($r^\star \leq \min\left\{m,n\right\}$), or \emph{approximately} low rank, \emph{i.e.}, 
$\|\Xopt - \Xopt_{r^\star}\|_F$ is sufficiently small, for $\Xopt_{r^\star}$ being the best rank-$r^\star$ approximation of $\Xopt$.
In our discussions, $f$ is a differentiable convex function; further assumptions on $f$ will be described later in the text. 
Note, in particular, that in the absence of further assumptions, $\Xopt$ may not be unique.

Specific instances of \eqref{intro:eq_00}, where the solution is assumed low-rank, appear in several applications in diverse research fields; 
a non-exhaustive list includes factorization-based recommender systems
\citep{srebro2004maximum, rennie2005fast, decoste2006collaborative, bennett2007netflix, koren2009matrix, jaggi2010simple},
multi-label classification tasks \citep{agrawal2013multi, bhatia2015sparse, carneiro2007supervised, makadia2008new, wang2009multi, weston2011wsabie},
dimensionality reduction techniques \citep{schein2003logisticPCA, chiang2014prediction, johnson2014logistic, verstrepen2015collaborative, gupta2015collectively, liu2016neighborhood},
density matrix estimation of quantum systems
\citep{aaronson2007learnability, gross2010quantum, kalev2015quantum}, 
phase retrieval applications \citep{candes2015phase, waldspurger2015phase}, 
sensor localization \citep{biswas2006semidefinite, weinberger2007graph} and
protein clustering \citep{lu2005framework} tasks,
image processing problems \citep{andrews1976singular}, as well as applications in system theory \citep{fazel2004rank},
just to name a few.
Thus, it is critical to devise user-friendly, efficient and provable algorithms that solve \eqref{intro:eq_00}, taking into consideration such near low-rank structure of $\Xopt$.

While, in general, imposing a low-rankness may result in an NP-hard problem, \eqref{intro:eq_00} with a rank-constraint can be solved in polynomial-time for numerous applications, where $f$ has specific structure.
A prime---and by now well-known---example of this is the matrix sensing/matrix completion problem \citep{candes2009exact, recht2010guaranteed, jaggi2010simple} (we discuss this further in Section \ref{sec:examples}). 
There, $f$ is a least-squares objective function and the measurements satisfy the appropriate restricted isometry/incoherence assumptions. 
In such a scenario, the optimal low-rank $\Xopt$ can be recovered in polynomial time, by solving \eqref{intro:eq_00} with a rank-constraint \citep{jain2010guaranteed, becker2013randomized, balzano2010online, lee2010admira, kyrillidis2014matrix, tanner2013normalized}, or by solving its convex nuclear-norm relaxation, as in \citep{lin2010augmented, becker2011templates, cai2010singular, becker2011nesta, chen2014coherent, yurtsever2015universal}. 

In view of the above and although such algorithms have attractive convergence rates, they directly manipulate the $n \times n$ variable matrix $X$, which in itself is computationally expensive in the high-dimensional regime. 
Specifically, each iteration in these schemes typically requires computing the top-$r$ singular value/vectors of the matrix. As $n$ scales, these computational demands at each iteration can be prohibitive.

\medskip
\noindent \textit{Optimizing over factors.}
In this paper, we follow a different path: a rank-$r$ matrix $\X \in \R^{m \times n}$ can be written as a product of two matrices $\U \V^\top$, where $\U \in \R^{m \times r}$ and $\V \in \R^{n \times r}$. 
Based on this premise, we consider optimizing $f$ over the $\U$ and $\V$ space. 
Particularly, we are interested in solving (\ref{intro:eq_00}) via the parametrization:
\begin{equation}
\begin{aligned}
 \underset{\U \in \R^{m \times r}, ~\V \in \R^{n \times r}}{\text{minimize}}
& & f(\U\V^\top) \quad \quad \quad \text{where $r \leq \rank(\Xopt) \leq \left\{m, ~n\right\}$}.
\end{aligned} \label{intro:eq_01}
\end{equation} 
Note that characterizations \eqref{intro:eq_01} and \eqref{intro:eq_00} are equivalent in the case $\text{rank}(\X^\star) = r$.\footnote{Here, by equivalent, we mean that the set of global minima in \eqref{intro:eq_01} contains that of \eqref{intro:eq_00}. It remains an open question though whether the reformulation in \eqref{intro:eq_01} introduces spurious local minima in the factored space for the majority of $f$ cases.}
Observe that such parameterization leads to a very specific kind of non-convexity in $f$.
Even more importantly, proving convergence for these settings becomes a harder task, due to the bi-linearity of the variable space.

\medskip
\noindent \textit{Motivation.} Our motivation for studying \eqref{intro:eq_01} originates from large-scale problem instances:  when $r$ is much smaller than $\min\{m, n\}$, factors $\U \in \R^{m \times r}$ and $\V \in \R^{n \times r}$ contain far fewer variables to maintain and optimize, than $\X = \U\V^\top$. 
Thus, by construction, such parametrization also makes it easier to update and store the iterates $\U, \V$. 

Even more importantly, observe that $\U\V^\top$ reformulation automatically encodes the rank constraint.
Standard approaches, that operate in the original variable space, either enforce the $\rank(\X) \leq r$ constraint at every iteration or involve a nuclear-norm projection. 
Doing so requires computing a truncated SVD\footnote{This holds in the best scenario; in the convex case, where the rank constraint is ``relaxed" by the nuclear norm, the projection onto the nuclear-norm ball often requires a full SVD calculation.} per iteration, which can get cumbersome in large-scale settings.
In stark contrast, working with $f(\U\V^\top)$ replaces singular value computation per iteration with matrix-matrix multiplication operations. 
%, but only matrix-matrix calculations. 
Thus, such non-conventional approach turns out to be a more practical and realistic option, when the dimension of the problem is large. 
We defer this discussion to Section \ref{sec:svd_mm} for some empirical evidence of the above.

\medskip
\noindent \textit{Our contributions.} 
While the computational gains are apparent, such bi-linear reformulations $X = \U\V^\top$often lack theoretical guarantees. Only recently, there have been attempts in providing answers to when and why such non-convex approaches perform well in practice, in the hope that they might provide a new algorithmic paradigm for designing faster and better algorithms; see \citep{jain2015computing, anandkumar2016efficient, tu2016low, zheng2015convergent, chen2015fast, bhojanapalli2016dropping, zhao2015nonconvex, sun2015guaranteed, zheng2016convergent, jin2016provable}.

As we describe below and in greater detail in Section \ref{sec:related}, our work is more general, addressing important settings that could not (as far as we know) be treated by the previous literature. Our contributions can be summarized as follows:

\begin{itemize}

\item[$(i)$] We study a gradient descent algorithm on the non-convex formulation given in \eqref{intro:eq_01} for \emph{non-square} matrices. We call this \emph{Bi-Factored Gradient Descent} (\algo). Recent developments (cited above, and see Section \ref{sec:related} for further details) rely on properties of $f$ for special cases \citep{sun2015guaranteed, tu2016low, zheng2016convergent, zhao2015nonconvex}, and their convergence results seem to rely on this special structure. In this work, we take a more generic view of such factorization techniques, closer to results in convex optimization. We provide local convergence guarantees for general smooth (and strongly convex) $f$ objectives.

\item[$(ii)$] In particular, when $f$ is only (restricted) smooth, we show that a simple lifting technique leads to a local sublinear rate convergence guarantee, using results from that of the square and PSD case \citep{bhojanapalli2016dropping}. Moreover, we provide a simpler and improved proof than \citep{bhojanapalli2016dropping}, which requires a weaker initial condition.

\item[$(iii)$] When $f$ is both (restricted) strongly convex and smooth, results from the PSD case do not readily extend. 
In such cases, of significant importance is the use of a regularizer in the objective, that restricts the geometry of the problem at hand. 
Here, we improve upon \citep{tu2016low, zheng2016convergent, yi2016rpca}---where such a regularizer was used only for the cases of matrix sensing/completion and robust PCA---and solve a different formulation that lead to local linear rate convergence guarantees. 
Our proof technique
proves a significant generalization: using any smooth and strongly convex regularizer on the term $(U^\top U - V^\top V)$, with optimum at zero, one can guarantee linear convergence. 

\item[$(iv)$] Our theory is backed up with extensive experiments, including affine rank minimization (Section \ref{sec:ARM}), compressed noisy image reconstruction from a subset of image pixels (Section \ref{sec:image_exp}), and 1-bit matrix completion tasks (Section \ref{sec:1-bit}). 
Overall, our proposed scheme shows superior recovery performance, as compared to state-of-the-art approaches, while being $(i)$ simple to implement, $(ii)$ scalable in practice and, $(iii)$ versatile to various applications.
\end{itemize}

\subsection{When Such Optimization Criteria Appear in Practice?}{\label{sec:examples}}
In this section, we describe some applications that can be modeled as in \eqref{intro:eq_01}.
The list includes criteria with $(i)$ smooth and strongly convex objective $f$ (\emph{e.g.}, quantum state tomography from a limited set of observations and compressed image de-noising),
and $(ii)$ just smooth objective $f$ (\emph{e.g.}, 1-bit matrix completion and logistic PCA). 
For all cases, we succinctly describe the problem and provide useful references on state-of-the-art approaches; we restrict our discussion on first-order, gradient schemes. 
Some discussion regarding recent developments on factorized approaches is deferred to Section \ref{sec:related}.
Section \ref{sec:exp} provides specific configuration of our algorithm, for representative tasks, and make a comparison with state of the art.

\subsubsection{Matrix Sensing Applications}
Matrix sensing (MS) problems have gained a lot of attention the past two decades, mostly as an extension of Compressed Sensing \citep{donoho2006compressed, baraniuk2007compressive} to matrices; see \citep{fazel2008compressed, recht2010guaranteed}.
The task involves the reconstruction of an \textit{unknown and low-rank ground truth matrix $\X^\star$, from a limited set of measurements}. 
The assumption on low-rankness depends on the application at hand and often is natural: 
\emph{e.g.}, in background subtraction applications, $\X^\star$ is a collection of video frames, stacked as columns, where the ``action" from frame to frame is assumed negligible \citep{candes2011robust, waters2011sparcs}; in (robust) principal component analysis \citep{chandrasekaran2009sparse, candes2011robust}, we intentionally search for a low-rank representation of the data;
in linear system identification, the low rank $\X^\star $ corresponds to a low-order linear, time-invariant system \citep{liu2009interior};
in sensor localization, $\X^\star$ denotes the matrix of pairwise distances with rank-dependence on the, usually, low-dimensional space of the data \citep{javanmard2013localization};
in quantum state tomography, $\X^\star$ denotes the density state matrix of the quantum system and $\X^\star$ is designed to be rank-1 (pure state) or rank-$r$ (almost pure state), for $r$ relatively small \citep{aaronson2007learnability, flammia2012quantum, kalev2015quantum}.

In a non-factored form, MS is expressed via the following criterion:
\begin{equation}{\label{MS:eq_00}}
\begin{aligned}
	& \underset{\X \in \R^{m \times n}}{\text{minimize}}
	& & f(\X) := \tfrac{1}{2} \cdot \|\obs - \linmap(\X)\|_2^2\\
	& \text{subject to}
	& & \text{rank}(\X) \leq r,
\end{aligned}
\end{equation} 
where usually $m \neq n$ and $r \ll \min\{m, ~n\}$. 
Here, 
$y = \linmap\left(\X^\star\right) + \varepsilon \in \R^p$ contains the (possibly noisy) samples, where $p \ll m \cdot n$.  
Key role in recovering $\X^\star$ plays the sampling operator $\linmap$: 
it can be a Gaussian-based linear map \citep{fazel2008compressed, recht2010guaranteed}, 
a Pauli-based measurement operator, used in quantum state tomography applications \citep{liu2011universal}, 
a Fourier-based measurement operator, which leads to computational gains in practice due to their structure \citep{krahmer2011new, recht2010guaranteed}, 
or even a permuted and sub-sampled noiselet linear operator, used in image and video compressive sensing applications \citep{waters2011sparcs}.

Critical assumption for $\linmap$ that renders \eqref{MS:eq_00} a polynomially solvable problem, is the \emph{restricted isometry property} (RIP) for low-rank matrices \citep{candes2011tight}:
\begin{definition}[Restricted Isometry Property (RIP)]\label{def:RIP}
A linear map $\linmap$ satisfies the $r$-RIP with constant $\delta_r$, if
\begin{align*}
(1 - \delta_r)\|\X\|_F^2 \leq \|\linmap(\X)\|_2^2 \leq (1 + \delta_r)\|\X\|_F^2,
\end{align*} is satisfied for all matrices $\X \in \R^{n \times n}$ such that $\text{rank}(\X) \leq r$. 
\end{definition} 
It turns out linear maps that satisfy Definition \ref{def:RIP} also satisfy the (restricted) smoothness and strong convexity assumptions \citep{negahban2012restricted}; see Theorem 2 in \citep{chen2010general} and Section \ref{sec:prelim} for their definition.

\medskip
\noindent \emph{State-of-the-art approaches.} 
The most popularized approach to solve this problem is through \emph{convexification}:
\citep{fazel2002matrix, recht2010guaranteed, candes2009exact} show that the nuclear norm $\|\cdot\|_*$ is the tightest convex relaxation of the non-convex $\text{rank}(\cdot)$ constraint and algorithms involving nuclear norm have been shown to be effective in recovering low-rank matrices. 
This leads to:
\begin{equation}{\label{MS:eq_01}}
\begin{aligned}
	& \underset{\X \in \R^{n \times p}}{\text{minimize}}
	& & f(\X) 
	& \text{subject to}
	& & \|\X\|_{*} \leq t,
\end{aligned}
\end{equation} 
and its regularized variant:
\begin{equation}{\label{MS:eq_02}}
\begin{aligned}
	& \underset{\X \in \R^{n \times p}}{\text{minimize}}
	& & f(\X) + \lambda \cdot \|\X\|_{*}.
\end{aligned}
\end{equation} 
Efficient implementations include Augmented Lagrange Multiplier (\texttt{ALM}) methods \citep{lin2010augmented}, convex conic solvers like the \texttt{TFOCS} software package \citep{becker2011templates} and, convex proximal and projected first-order methods \citep{cai2010singular, becker2011nesta}. 
However, due to the nuclear norm, in most cases these methods are binded with full SVD computations per iteration, which 
constitutes them impractical in large-scale settings.

From a non-convex perspective, algorithms that solve \eqref{MS:eq_00} in a non-factored form include \texttt{SVP} and Randomized \texttt{SVP} algorithms \citep{jain2010guaranteed, becker2013randomized}, 
Riemannian Trust Region Matrix Completion algorithm (\texttt{RTRMC}) \citep{boumal2011rtrmc}, 
\texttt{ADMiRA} \citep{lee2010admira} and the \texttt{Matrix ALPS} framework \citep{kyrillidis2014matrix, tanner2013normalized}.%\footnote{Based on the results of \citep{kyrillidis2014matrix}, we use \texttt{Matrix ALPS II} in our experiments, as that scheme seems to be more effective and faster than the aforementioned algorithms.}

In all cases, algorithms admit fast linear convergence rates. % towards $\X^\star$. 
Moreover, the majority of approaches assumes a \emph{first-order} oracle: 
information of $f$ is provided through its gradient $\nabla f(\X)$. 
For MS, $\nabla f(\X) = -2 \linmap^*\left(\obs - \linmap(\X)\right)$, which requires $O({\rm T}_{\text{map}})$ complexity, where ${\rm T}_{\text{map}}$ denotes the time required to apply linear map (or its adjoint $\linmap^*$) $\linmap$. 
Formulations \eqref{MS:eq_00}-\eqref{MS:eq_02} require at least one top-$r$ SVD calculation per iteration; this translates into additional $O(mnr)$ complexity.

\medskip
\noindent \emph{Motivation for factorizing \eqref{MS:eq_00}. }
Problem \eqref{MS:eq_00} can be factorized as follows:
\begin{equation}{\label{MS:trans_eq_00}}
\begin{aligned}
	& \underset{\U \in \R^{n \times r}, \V \in \R^{p \times r}}{\text{minimize}}
	& &  f(\U\V^\top) := \tfrac{1}{2} \cdot \|\obs - \linmap(\U\V^\top)\|_2^2.
\end{aligned}
\end{equation} 
For this case and assuming a first-order oracle, the gradient of $f$ with respect to $\U$ and $\V$ can be computed respectively as $\nabla_{\U} f(\U \V^\top) := \gradf(\X) \V$ and $\nabla_{\V} f(\U \V^\top) := \gradf(\X)^\top \U$, respectively. 
This translates into $2\cdot O({\rm T}_{\text{map}} + mnr)$ time complexity. 
However, one avoids performing any SVD calculations per iteration, which in practice is considered a great computational bottleneck, even for moderate $r$ values.
Thus, if there exist linearly convergent algorithms for \eqref{MS:trans_eq_00}, intuition indicates that one could obtain computational gains.

\subsubsection{Logistic PCA and Low-Rank Estimation on Binary Data}
Finding a low-rank approximation of binary matrices has gain a lot of interest recently, due to the wide appearance of categorical responses in real world applications \citep{schein2003logisticPCA, chiang2014prediction, johnson2014logistic, verstrepen2015collaborative, gupta2015collectively, liu2016neighborhood}.
While regular linear principal component analysis (PCA) is still applicable for binary or categorical data, $(i)$ the way data are pre-processed (\emph{e.g.}, centering data before applying PCA), and/or $(ii)$ the least-squares objective in PCA, constitute it a natural choice for \emph{real-valued} data, where observations are assumed to follow a Gaussian distribution.
\citep{tipping1999probabilistic, collins2001generalization} propose generalized versions of PCA for other type of data sets:
In the case of binary data, this leads to Logistic Principal Component Analysis (Logistic PCA), where each binary data vector is assumed to follow the multivariate Bernoulli distribution, parametrized by the principal components that live in a $r$-dimensional subspace. 
Moreover, collaborative filtering on binary data and network sign prediction tasks  have shown that standard least-squares loss functions perform poorly, while generic logistic loss optimization shows more interpretable and promising results.

To rigorously formulate the problem, let $\Y \in \left\{0, 1\right\}^{m \times n}$ be the observed binary matrix, where each of the $m$ rows stores a $n$-dimensional binary feature vector. 
Further, assume that each entry $\Y_{ij}$ is drawn from a Bernoulli distribution with mean $q_{ij}$, according to: $\mathbb{P}\left[\Y_{ij} ~|~ q_{ij} \right] = q_{ij}^{Y_{ij}} \cdot (1 - q_{ij})^{1- \Y_{ij}}$. 
Define the log-odds parameter $\X_{ij} = \log \left( \tfrac{q_{ij}}{1 - q_{ij}}\right)$ and the logistic function $\sigma(\X_{ij}) = \left(1 + e^{-\X_{ij}}\right)^{-1}$. 
Then, we equivalently have $\mathbb{P}\left[\Y_{ij} ~|~ X_{ij} \right] = \sigma(\X_{ij})^{Y_{ij}} \cdot \sigma(-X_{ij})^{1-Y_{ij}}$, or in matrix form,
\begin{align*}
\mathbb{P}\left[\Y ~|~ X \right] = \prod_{ij} \sigma(\X_{ij})^{Y_{ij}} \cdot \sigma(-X_{ij})^{1-Y_{ij}},
\end{align*} where we assume independence among entries of $\Y$. 
The negative log-likelihood for log-odds parameter $\X$ is given by:
\begin{align*}
f(\X) := -\sum_{ij} \left(Y_{ij} \cdot \log \sigma(\X_{ij}) + \left(1 - \Y_{ij}\right) \cdot \log \sigma(-\X_{ij})\right).
\end{align*}
Assuming a compact, \emph{i.e.}, low-rank, representation for the latent variable $\X$, we end up with the following optimization problem:
\begin{equation}{\label{logPCA:eq_00}}
\begin{aligned}
	& \underset{\X \in \R^{m \times n}}{\text{minimize}}
	& & f(\X) := -\sum_{ij} \left(Y_{ij} \cdot \log \sigma(\X_{ij}) + \left(1 - \Y_{ij}\right) \cdot \log \sigma(-\X_{ij})\right)\\
	& \text{subject to}
	& & \text{rank}(\X) \leq r;
\end{aligned}
\end{equation} observe that the objective criterion is just a smooth convex loss function. 

\medskip
\noindent \emph{State-of-the-art approaches.}\footnote{Here, we note that \citep{landgraf2015dimensionality} proposes a slightly different way to generalize PCA than \citep{collins2001generalization}, based on a different interpretation of Pearson's PCA formulation \citep{pearson1901principal}. 
The resulting formulation looks for a \emph{weighted} projection matrix $\U\U^\top$ (instead of $\U\V^\top)$, where the number of parameters does not increase with the number of samples and the application of principal components to new data requires only one matrix multiplication. For this case, the authors in \citep{landgraf2015dimensionality} propose, among others, an alternating minimization technique where convergence to a local minimum is guaranteed. Even for this case though, our framework applies.} 
In \citep{chiang2014prediction}, the authors consider the problem of \emph{sign prediction} of edges in a signed network and cast it as a low-rank matrix completion problem: In order to model sign inconsistencies between the entries of binary matrices, the authors consider more appropriate loss functions to minimize, among which is the logistic loss. The proposed algorithmic solution follows (stochastic) gradient descent motions; however, no guarantees are provided.
\citep{johnson2014logistic} utilizes logistic PCA for collaborative filtering on implicit feedback data (page clicks and views, purchases, etc.):  to find a local minimum, an alternating gradient descent procedure is used---further, the authors use AdaGrad \citep{duchi2011adaptive} to adaptively update the gradient descent step size, in order to reduce the number of iterations for convergence. 
A similar alternating gradient descent approach is followed in \citep{schein2003logisticPCA}, with no known theoretical guarantees.

\medskip
\noindent \emph{Motivation for factorizing \eqref{logPCA:eq_00}.} Following same arguments as before, in logistic PCA and logistic matrix factorization problems, we often assume that the observation binary matrix is generated as the \texttt{sign} operation on a linear factored model: $\texttt{sign}(\U\V^T)$.
%Assuming the probability of $\{\pm 1\}$ values is distributed according to a logistic function, 
Parameterized by the latent factors $\U, \V$, we obtain the following optimization criterion:
\begin{equation}{\label{logPCA:eq_01}}
\begin{aligned}
	& \underset{\U \in \R^{m \times r}, ~\V \in \R^{n \times r}}{\text{minimize}}
	& & f(\U\V^\top) := -\sum_{ij} \left(Y_{ij} \cdot \log \sigma(\U_{i} \V_j^\top) + \left(1 - \Y_{ij}\right) \cdot \log \sigma(-\U_{i} \V_j^\top)\right),
\end{aligned}
\end{equation} where $\U_i$, $V_j$ represent the $i$-th and $j$-th row of $\U$ and $\V$, respectively.

\subsection{Related Work}{\label{sec:related}}

As it is apparent from the discussion above, this is not the first time such transformations have been considered in practice. 
Early works on principal component analysis \citep{christoffersson1970one, ruhe1974numerical} and non-linear estimation procedures \citep{wold1969nonlinear}, use this technique as a heuristic; empirical evaluations show that such heuristics work well in practice \citep{rennie2005fast, haeffele2014structured, aravkin2014fast}.
\citep{burer2003nonlinear, burer2005local} further popularized these ideas for solving SDPs: their approach embeds the PSD and linear constraints into the objective and applies low-rank variable re-parameterization. While the constraint considered here is of different nature---\emph{i.e}, rank constraint vs. PSD constraint---the motivation is similar: in SDPs, by representing the solution as a product of two factor matrices, one can remove the positive semi-definite constraint and thus, avoid computationally expensive projections onto the PSD cone.

We provide an overview of algorithms that solve instances of \eqref{intro:eq_01}; for discussions on methods that operate on $\X$ directly, we defer the reader to \citep{agarwal2010fast, jaggi2010simple, kyrillidis2014matrix} for more details; see also Table \ref{table:algo_comp_summary} for an overview of the discussion below. 
We divide our discussion into two problem settings: $(i)$ $\X^\star$ is square and PSD and, $(ii)$ $\X^\star$ is non-square. 

\begin{table*}[!ht]
\centering
\begin{small}
\rowcolors{2}{white}{black!05!white}
\begin{tabular}{c c c c c c c c c c c}
  \toprule
  Algo. & & $X$ & & $UU^\top$ & & $UV^\top$ & & Rate & & Function $f$ \\ 
  \cmidrule{1-1} \cmidrule{3-3} \cmidrule{5-5} \cmidrule{7-7} \cmidrule{9-9} \cmidrule{11-11}
      \cite{jain2013low} & & \cmark & & \xmark & & \xmark & & Linear & & MS \\
      \cite{tanner2013normalized} & & \cmark & & \xmark & & \xmark & & Linear & & MS \\
      \cite{kyrillidis2014matrix} & & \cmark & & \xmark & & \xmark & & Linear & & MS \\
      \cite{hazan2008sparse} & & \xmark & & \cmark & & \xmark & & Sublinear & & MS \\
      \cite{laue2012hybrid} & & \xmark & & \cmark & & \xmark & & Sublinear & & Convex $f$ \\
      \cite{chen2015fast} & & \xmark & & \cmark & & \xmark & & (Sub)linear & & Generic $f^{\dagger}$ \\
      \cite{bhojanapalli2016dropping} & & \xmark & & \cmark & & \xmark & & (Sub)linear & & Convex $f$ \\
      \cite{jaggi2010simple} & & \xmark & & \cmark & & \cmark & & Sublinear & & Convex $f$ \\
      \cite{tu2016low} & & \xmark & & \cmark & & \cmark & & Linear & & MS \\
      \cite{zheng2015convergent} & & \xmark & & \cmark & & \cmark & & Linear & & MS \\
      \cite{hardt2014fast} & & \xmark & & \cmark & & \cmark & & Linear & & MC \\
      \cite{jain2013low} & & \xmark & & \cmark & & \cmark & & Linear & & MC, MS \\
      \cite{sun2015guaranteed} & & \xmark & & \cmark & & \cmark & & Linear & & MC \\
      \cite{zhao2015nonconvex} & & \xmark & & \cmark & & \cmark & & Linear & & Strongly convex $f$ \\
      \midrule
      This work & & \xmark & & \cmark & & \cmark & & (Sub)linear & & Convex $f$ \\
  \bottomrule
\end{tabular}
\end{small}
\caption{Summary of selected \textbf{non-convex} solvers for low-rank inference problems. $X$, $UU^\top$, and $UV^\top$ denote the setting the algorithm works at; $X$ means that no factorization is considered.  
``\textit{Rate}" describes the convergence rate; ``\textit{(Sub)linear}" denotes that both sublinear and linear rates are proved, depending on the nature of $f$. ``\textit{Function $f$}" denotes the problems cases covered by each algorithm: ``\textit{MS}" and ``\textit{MC}" stand for matrix sensing and completion, respectively; \textit{``Convex $f$"} corresponds to standard smooth and strongly convex functions $f$. For the case of \cite{chen2015fast}, \textit{``Generic $f^{\dagger}$"} corresponds to a specific class of functions that can be even concave but should satisfy specific conditions; see \cite{chen2015fast}.} \label{table:algo_comp_summary} 
\end{table*}

\medskip
\noindent \textit{Square and PSD $\X^\star$.}
A rank-$r$ matrix $\X \in \R^{n \times n}$ is PSD if and only if it can be factored as $\X = \U\U^\top$ for some $\U \in \R^{n \times r}$. 
This is a special case of the problem discussed above, where $m = n$ and \eqref{intro:eq_00} includes a PSD constraint. Thus, after the re-parameterization, \eqref{intro:eq_01} takes the form:
\begin{equation}
\begin{aligned}
 \underset{\U \in \R^{n \times r}}{\text{minimize}}
& & f(\U\U^\top) \quad \quad \quad \text{where $r = \text{rank}(\X^\star) \leq n$}.
\end{aligned} \label{intro:eq_02}
\end{equation} 

Several recent works have studied \eqref{intro:eq_02}. 
For the special case where $f$ is a least-squares objective for an underlying linear system, %\footnote{This includes affine rank minimization problems, as well as matrix completion instances.}, 
\citep{tu2016low} and \citep{zheng2015convergent} propose gradient descent schemes that function on the factor $\U$. 
Both studies employ careful initialization (performing few iterations of \texttt{SVP} \citep{jain2010guaranteed} for the former and, using a spectral initialization procedure---as in \citep{candes2015phase2}---for the latter) and step size selection, in order to prove convergence.\footnote{Recently, \citep{ge2016matrix} and \citep{bhojanapalli2016global} proved that $\U\U^\top$ factorization introduces no spurious local minima for the cases of matrix completion and sensing, respectively: random initialization eventually leads to convergence to the optimal $\X^\star$ (or close to $\X^\star$ in the approximately low rank case). The extension of these results to the non-square matrix sensing setting can be found in \cite{park2016non}. }
However, their analysis is designed only for least-squares instances of $f$. 
Some results and discussion on their step size selection/initialization and how it compares with this work are provided in Section \ref{sec:exp}.

The work of \citep{chen2015fast} proposes a first-order algorithm for \eqref{intro:eq_02}, where $f$ is more generic.
The algorithmic solution proposed can handle additional constraints on the factors $\U$; the nature of these constraints depends on the problem at hand.\footnote{Any additional constraints should satisfy the 
\textit{faithfulness} property: a constraint set $\C$ is faithful if for each $\U \in \C$, within some bounded radius from 
optimal point, we are guaranteed that the closest (in the Euclidean sense) rotation of optimal $\Uo$ lies within 
$\mathcal{U}$.} 
The authors present a broad set of exemplars for $f$---matrix completion and sensing, as well as sparse PCA, among others.
For each problem, a set of assumptions need to be satisfied; \textit{i.e.}, faithfulness, local descent, local Lipschitz 
and local smoothness conditions; see \citep{chen2015fast} for more details. 
Under such assumptions and with proper initialization, one can prove convergence with $O(\sfrac{1}{\varepsilon})$ 
or $O(\log(\sfrac{1}{\varepsilon}))$ rate, depending on the nature of $f$, and for problems that even fail to be locally convex.

\citep{bhojanapalli2016dropping} proposes \emph{Factored Gradient Descent} (\fgd) algorithm for \eqref{intro:eq_02}. 
\fgd is also a first-order scheme; key ingredient for convergence is a novel step size selection that can be used for any $f$, as long as it is (restricted) gradient Lipschitz continuous; when $f$ is further (restricted) strongly convex, their analysis lead to faster convergence rates.
Using proper initialization, this is the first paper that \emph{provably} solves \eqref{intro:eq_02} for general convex functions $f$ and under common convex assumptions. 
An extension of these ideas to some constrained cases can be found in \citep{park2016provable}.

To summarize, most of these results guarantee convergence---up to linear rate---on the factored space, starting from a ``good'' initialization point and employing a carefully selected step size.

\medskip
\noindent \textit{Non-square $\X^\star$.}
\citep{jain2013low} propose \texttt{AltMinSense}, an alternating minimization algorithm for matrix sensing and matrix completion problems. 
This is one of the first works to prove linear convergence in solving \eqref{intro:eq_01} for the MS model.
\citep{hardt2014fast} improves upon \citep{jain2013low} for the case of reasonably well-conditioned matrices. 
Their algorithm handles problem cases with bad condition number and gaps in their spectrum \citep{mary2016personal}.
Recently, \citep{tu2016low} extended the \texttt{Procrustes Flow} algorithm to the non-square case, where gradient descent, instead of exact alternating minimization, is utilized. \citep{zheng2016convergent} extended the first-order method of \citep{chen2015fast} for matrix completion to the rectangular case. All the studies above focus on the case of least-squares objective $f$.

\citep{sun2015guaranteed} generalize the results in \citep{jain2013low, hardt2014fast}: 
the authors show that, under common incoherence conditions and sampling assumptions, most first-order variants (\textit{e.g.}, gradient descent, alternating minimization schemes and stochastic gradient descent, among others) indeed converge to the low-rank ground truth $\X^\star$. 
%Specifically, for the alternating gradient descent variant, the authors propose several step size selection procedures\footnote{However, the restricted Armijo rule, as well as the line search procedure, can be applied to any the aforementioned algorithms too.}. 
Both the theory and the algorithm proposed are restricted to the matrix completion objective.

Recently, \citep{zhao2015nonconvex}---based on the inexact first-order oracle, previously used in \citep{balakrishnan2014statistical}---proved that linear convergence is guaranteed if $f(UV^\top)$ is strongly convex over either $U$ and $V$, when the other is fixed. 
While the technique applies for generic $f$ and for non-square $\X$, the authors provide algorithmic solutions only for matrix completion / matrix sensing settings.\footnote{\textit{E.g.}, in the gradient descent case, the step size proposed depends on RIP \citep{recht2010guaranteed} constants and it is not clear what a good step size would be in other problem settings.} 
Furthermore, their algorithm requires QR-decompositions after each update of $\U$ and $\V$; this is required in order to control the notion of inexact first order oracle.

\section{Preliminaries}{\label{sec:prelim}}

\textit{Notation.} 
For matrices $\X, \Y \in \R^{m \times n}$, $\ip{\X}{\Y} = \trace\left(\X^\top \Y \right)$ represents their inner product. 
We use $\norm{\X}_F$ and $\sigma_1(\X)$ for the Frobenius and spectral norms of a matrix, respectively; we also use $\|\X\|_2$ to denote spectral norm.
Moreover, we denote as $\sigma_i(\X)$ the $i$-th singular value of $\X$.
For a rank-$r$ matrix $\X = \U\V^\top$, the gradient of $f$ w.r.t. $\U$ and $\V$ is $\gradf(\U \V^\top) \V$ and $\gradf(\U \V^\top)^\top \U$, respectively.  
With a slight abuse of notation, we will also use the terms $\nabla_{\U} f(\U \V^\top) := \gradf(\U \V^\top) \V$ and $\nabla_{\V} f(\U \V^\top) := \gradf(\U \V^\top)^\top \U$.

Given a matrix $\X$, we denote its best rank-$r$ approximation with $\X_r$; $\X_r$ can be computed in polynomial time via the SVD.
For our discussions from now on and in an attempt to simplify our notation, we denote the optimum point we search for as $\Xopt_r$, both $(i)$ in the case where we intentionally restrict our search to obtain a rank-$r$ approximation of $\Xopt$---while $\text{rank}(\Xopt) > r$---and $(i)$ in the case where $\Xopt \equiv \Xopt_r$, \emph{i.e.}, by default, the optimum point is of rank $r$.

An important issue in optimizing $f$ over the factored space is the existence of non-unique possible factorizations for a given $\X$. 
Since we are interested in obtaining a low-rank solution in the original space, we need a notion of distance to the low-rank solution $\Xoptr$ over the factors. 
Among infinitely many possible decompositions of $\Xoptr$, we focus on the set of ``equally-footed'' factorizations \citep{tu2016low}:
\begin{align}
\mathcal{X}^\star_r = \Big\{ &\left(\Uo,\Vo\right):~\Uo \in \R^{m \times r}, \Vo \in \R^{n \times r}, \nonumber \\ &\quad \quad \quad \quad \quad \Uo {\Vo}^\top = \Xoptr, \sigma_i(\Uo) = \sigma_i(\Vo) = \sigma_i(\Xoptr)^{1/2}, \forall i \in [r]\Big\}. \label{prelim:eq_footing}
\end{align}
Note that $(\Uo,\Vo) \in \mathcal{X}^\star_r$ if and only if the pair can be written as
$\Uo = A^\star \Sigma^{\star 1/2} R ,~ \Vo = B^\star \Sigma^{\star 1/2} R$,
where $A^\star \Sigma^\star B^\star$ is the singular value decomposition of $\Xoptr$, and $R \in \mathbb{R}^{r \times r}$ is an orthogonal matrix. 

Given a pair $(\U,\V)$, we define the distance to $\Xoptr$ as:
\begin{align*}
\dist\left(\U,\V;\Xoptr\right) = \min_{(\Uo,\Vo) \in \mathcal{X}^\star_r} \normf{\begin{bmatrix} \U \\ \V \end{bmatrix} - \begin{bmatrix} \Uo \\ \Vo \end{bmatrix}}.
\end{align*}

\medskip
\noindent \textit{Assumptions.} We consider applications that can be described $(i)$ either by restricted \emph{strongly} convex functions $f$ with \emph{gradient Lipschitz continuity}, or $(ii)$ by convex functions $f$ that have only Lipschitz continuous gradients.\footnote{Our ideas can be extended in a similar fashion to the case of restricted strong convexity \citep{negahban2012restricted, agarwal2010fast}.} We state these standard definitions below.

\begin{definition}{\label{prelim:def_00}}
Let $f: \R^{m \times n} \rightarrow \R$ be a convex differentiable function. Then, $f$ is gradient Lipschitz continuous with parameter $L$ (or $L$-smooth) if: 
\begin{equation}
\norm{\gradf\left(\X\right) - \gradf\left(\Y\right)}_F \leq L \cdot \norm{\X - \Y}_F, 
\end{equation} $\forall \X, \Y \in  \R^{m \times n}.$
\end{definition} 

\begin{definition}{\label{prelim:def_01}}
Let $f: \R^{m \times n} \rightarrow \R$ be convex and differentiable. Then, $f$ is $\mu$-strongly convex if: 
\begin{equation}\label{eq:sc}
f(\Y) \geq f(\X) + \ip{\gradf\left(\X\right)}{\Y - \X} + \tfrac{\mu}{2} \norm{Y - \X}_F^2,
\end{equation} $\forall \X, \Y \in \set{X} \subseteq \R^{m \times n}$.
\end{definition}

\medskip
\noindent \textit{The Factored Gradient Descent (\fgd) algorithm.} Part of our contributions is inspired by the work of \citep{bhojanapalli2016dropping}, where the \fgd algorithm is proposed.
For completeness, we describe here the problem they consider and the proposed algorithm.
\citep{bhojanapalli2016dropping} considers the problem:
\begin{equation*}
\begin{aligned}
	& \underset{X \in \R^{n \times n}}{\text{minimize}}
	& & f(X),
	& \text{subject to}
	& & \X \succeq 0,
\end{aligned}
\end{equation*} 
and proposes the following first-order recursion for its solution:
\begin{align*}
\U_{t+1} = \U_t - \eta \cdot \nabla f(\U_t \U_t^\top) \cdot U_t.
\end{align*} 
Key property in their analysis is the positive semi-definiteness of the feasible space.
For a proper initialization and step size, \citep{bhojanapalli2016dropping} show sublinear and linear convergence rates towards optimum, depending on the nature of $f$.

\section{The Bi-Factored Gradient Descent (\algo) Algorithm}

In this section, we provide an overview of the \emph{Bi-Factored Gradient Descent} (\algo) algorithm for two problem settings in \eqref{intro:eq_00}: $(i)$ $f$ being a $L$-smooth convex function and, $(ii)$ $f$ being $L$-smooth and $\mu$-strongly convex. 
For both cases, we assume a good initialization point $\X_0 = \U_0 \V_0^\top$ is provided; for a discussion regarding initialization, see Section \ref{sec:init}. 
Given $\X_0$ and under proper assumptions, we further describe the theoretical guarantees that accompany \algo.

As introduced in Section \ref{sec:intro}, \algo is built upon non-convex gradient descent over each factor $\U$ and $\V$, written as
\begin{align} \label{eqn:nonconvex_gd}
\U_{t+1} = \U_t - \eta \cdot \nabla_\U f(\U_t \V_t^\top), &\quad  
\V_{t+1} = \V_t - \eta \cdot \nabla_\V f(\U_t \V_t^\top).
\end{align}
When $f$ is convex and smooth, \algo follows exactly the motions in \eqref{eqn:nonconvex_gd}; in the case where $f$ is also strongly convex, \algo is based on a different set of recursions, which we discuss in more detail in the rest of the section.

\subsection{Reduction to \fgd: When $f$ is Convex and $L$-Smooth} \label{sec:bfgd_algo_smooth}
In \citep{jaggi2010simple}, the authors describe a simple technique to transform problems similar to \eqref{intro:eq_00} into problems where we look for a square and PSD solution. 
The key idea is to \emph{lift} the problem and introduce a stacked matrix of the two factors, as follows:
\begin{align*}
W = \begin{bmatrix} U \\ V \end{bmatrix} \in \R^{(m+n) \times r},
\end{align*}
and optimize over a new function $\hat{f} : \R^{(m+n)\times(m+n)} \to \R$ defined as
\begin{align*}
\hat{f} \left( WW^\top \right) = \hat{f} \left( \begin{bmatrix} UU^\top & UV^\top \\ VU^\top & VV^\top \end{bmatrix} \right) = f(UV^\top).
\end{align*}
Following this idea, one can utilize algorithmic solutions designed only to work on square and PSD-based instances of \eqref{intro:eq_00}, where $f$ is just $L$-smooth. 
Here, we use the \emph{Factored Gradient Descent} (\fgd) algorithm of \citep{bhojanapalli2016dropping} on the $W$-space, as follows:
\begin{align} \label{eqn:nonconvex_fgd}
W_{t+1} = W_t - \eta \cdot \nabla_W \hat{f}(W_t W_t^\top).
\end{align}
Then, it is easy to verify the following remark: %into one can apply \fgd to our problem, as presented in the following.
\begin{remark} \label{prop:reduction}
Define
\begin{align*}
\hat{f} \left( \begin{bmatrix} A & B \\ C & D \end{bmatrix} \right) = \frac{1}{2}f(B) + \frac{1}{2} f(C^\top)
\end{align*}
for $A \in \R^{m \times m}$, $B \in \R^{m \times n}$, $C \in \R^{n \times m}$, $D \in \R^{n \times n}$. Then \fgd for minimizing $\fh(WW^\top)$ with the stacked matrix $W = [U^\top, V^\top]^\top \in \R^{(m+n) \times r}$ is equivalent to \eqref{eqn:nonconvex_gd}.
\end{remark}
A natural question is whether this reduction gives a desirable convergence behavior. 
Since \fgd solves for a different function $\hat{f}$ from the original $f$, the convergence analysis depends also on $\hat{f}$. 
When $f$ is convex and smooth, we can rely on the result from \citep{bhojanapalli2016dropping}.
\begin{proposition} \label{lem:reduction}	
If $f$ is convex and $L$-smooth, then $\hat{f}$ is convex and $\tfrac{L}{2}$-smooth.
\end{proposition}
\begin{proof}
For any $Z_1 = \begin{bmatrix} A_1 & B_1 \\ C_1 & D_1 \end{bmatrix}, Z_2 = \begin{bmatrix} A_2 & B_2 \\ C_2 & D_2 \end{bmatrix} \in \R^{(m+n) \times (m+n)}$, we have
\begin{align*}
\left\| \nabla \hat{f}(Z_1) - \nabla \hat{f}(Z_2) \right\|_F
&= \left\| \begin{bmatrix} 0 & \tfrac{1}{2} \cdot \nabla f(B_1) \\ \tfrac{1}{2} \cdot \nabla f(C_1^\top)^\top & 0 \end{bmatrix}
         - \begin{bmatrix} 0 & \tfrac{1}{2} \cdot \nabla f(B_2) \\ \tfrac{1}{2} \cdot \nabla f(C_2^\top)^\top & 0 \end{bmatrix}
         \right\|_F \\
&= \tfrac{1}{2} \cdot \sqrt{ \left\| \nabla f(B_1) - \nabla f(B_2) \right\|_F^2 + \left\| \nabla f(C_1^\top) - \nabla f(C_2^\top) \right\|_F^2} \\
&\le \tfrac{L}{2} \cdot \sqrt{ \left\| B_1 - B_2 \right\|_F^2 + \left\| C_1 - C_2 \right\|_F^2} \\
&\le \tfrac{L}{2} \cdot \|Z_1 - Z_2\|_F
\end{align*}
where the first inequality follows from the $L$-smoothness of $f$.
\end{proof}
Based on the above proposition, we use \fgd to solve \eqref{intro:eq_01} with $\hat{f}$: its procedure is exactly \eqref{eqn:nonconvex_gd}, with a different step size: 
\begin{align} \label{eq:new_step_size_smooth}
\eta &\le \frac{1}{20L \normt{\begin{bmatrix} \U_0 \\ \V_0 \end{bmatrix}}^2 + 3 \|\gradf(\U_0\V_0^\top)\|_2}.
\end{align}
While one can rely on the sublinear convergence analysis from \citep{bhojanapalli2016dropping}, we provide a new guarantee with a weaker initial condition; see Section \ref{sec:guarantees}.
%In Section \ref{sec:guarantees}, we discuss a convergence guarantee under this step size condition.

\subsection{Using \algo when $f$ is $L$-Smooth and Strongly Convex}{\label{sec:bfgd_strcvx}}

Assume $f$ function satisfies both properties in Definitions \ref{prelim:def_00} and \ref{prelim:def_01}. 
In this case, we cannot simply rely on the lifting technique as above since $\hat{f}$ is clearly not strongly convex. 
Instead, we consider a slight variation, where we appropriately regularize the objective and force the solution pair $(\widehat{\U}, \widehat{\V})$ to be ``balanced".
This regularization is based on the set of optimal pairs $(U^\star, V^\star)$ in $\mathcal{X}^\star_r$, as defined in \eqref{prelim:eq_footing}.
In particular, given $\mathcal{X}^\star_r$, the equivalent optimization problem that ``forces" convergence to balanced $(\U^\star, \V^\star)$ is as follows:
\begin{equation}{\label{eq:problem}}
\begin{aligned}
 \underset{\U \in \R^{m \times r}, ~\V \in \R^{n \times r}}{\text{minimize}}
& & f(\U\V^\top) + \lambda \cdot g(\U^\top \U - \V^\top \V),
\end{aligned}
\end{equation}
where $g : \R^{r \times r} \to \R$ is an additional convex \emph{regularizer}. 
We require:
\begin{itemize}[leftmargin=0.9cm]
\item[$(i)$] $g$ is convex and minimized at zero point; \emph{i.e.}, $\nabla g(0) = 0$. \vspace{-0.2cm}
\item[$(ii)$] The gradient, $\nabla g(U^\top U - V^\top V) \in \R^{r \times r}$ , is symmetric for any such pair. \vspace{-0.2cm}
\item[$(iii)$] $g$ is $\mu_g$-strongly convex and $L_g$-smooth.
\end{itemize}
Under such assumptions, the addition of $g$ in the objective just restricts the set of optimum points to be ``balanced", \emph{i.e., the minimizer of \eqref{eq:problem} minimizes also \eqref{intro:eq_01}}.\footnote{In particular, for any rank-$r$ solution $\factor$ in \eqref{intro:eq_01}, there is a factorization $(\Ut,\Vt)$ minimizing $g$ with the same function value $f(\Ut \Vt^\top) = f(\factor)$, which are 
$$
\Ut = A \Sigma^{\frac{1}{2}} ,\quad \Vt = B \Sigma^{\frac{1}{2}}
$$
where $\factor = A \Sigma B^\top$ is the singular value decomposition.}

\medskip
\noindent \textit{The necessity of the regularizer.}
As we show next, the theoretical guarantees of \algo heavily depend on the condition number of the pair $(\Uo, \Vo)$ the algorithm converges to.
In particular, one of the requirements of \algo is that every estimate $\U_t$ (resp. $\V_t$) be ``relatively close'' to the convergent point $\Uo$ (resp. $\Vo$), such that their distance $\|\U_t - \Uo\|_F$ is bounded by a function of $\sigma_r(\Uo)$, for all $t$. 
Then, it is easy to observe that, for arbitrarily ill-conditioned $(\Uo, \Vo) \in \set{X}_r^*$, such a condition might not be easily satisfied by \algo per iteration\footnote{Even if $\factor$ is close to $\factoropt$, the condition numbers of $U$ and $V$ can be much larger than the condition number of $\factor$.}, unless we ``force'' the sequence of estimates $(\U_t, \V_t), ~\forall t,$ to converge to a better conditioned pair $(\Uo, \Vo)$.
This is the key role of regularizer $g$: it guarantees putative estimates $U_t$ and $V_t$ are not too ill-conditioned, per iteration.  
%Note that adding $g$ does not change the optimum of $f$ in the original $X$ space.

An example of $g$ is the Frobenius norm (weighted by $\mu/2$), as proposed in \citep{tu2016low}. Other examples are sums of element-wise (at least) $\mu_g$-strongly convex and (at most) $L_g$-gradient Lipschitz functions (of the form $g(X) = \sum_{i,j} g_{ij}(X_{ij})$) with the optimum at zero. However, any other user-friendly $g$ can be selected, as long as it satisfies the above conditions. 
We show in this paper that any such regularizer results provably in convergence, with attractive convergence rate; see Section \ref{sec:regularizer} for a toy example where the addition of $g$ leads to faster convergence rate in practice.
%As we observe in practice though, in most cases, when the initial point $(U_0, V_0)$ is well-conditioned, one can remove $g$ from the objective and observe slightly different performance in practice.
%See Section \ref{} for a case where $g$ leads to faster convergence.

\medskip
\noindent \textit{The \algo algorithm.}
\algo is a first-order, gradient descent algorithm for \eqref{eq:problem}, that operates on the factored space $(\U, \V)$ in an alternating fashion. 
Principal components of \algo is a proper step size selection and a ``decent" initialization point.
\algo can be considered as the non-squared extension of \texttt{FGD}\xspace algorithm in \citep{bhojanapalli2016dropping}, which is specifically designed to solve problems as in \eqref{intro:eq_01}, for $\U = \V$ and $m = n$. 
The key differences with \texttt{FGD} though, other than the necessity of a regularizer $g$, are:
\begin{itemize}[leftmargin=0.9cm]
\item [$(i)$] Our analysis leads to provable convergence results in the non-square case. Such a result cannot be trivially obtained from \citep{bhojanapalli2016dropping}.
\item[$(ii)$] The main recursion followed is different in the two schemes: in the non-squared case, we update the left and right factors $(\U, ~\V)$ with a different rule, according to which:
\begin{align*}
\U_{t+1} &= \U_t - \eta \left( \nabla_\U f(\U_t \V_t^\top) + \lambda \cdot \nabla_\U g(\U_t^\top \U_t - \V_t^\top \V_t) \right),\\
\V_{t+1} &= \V_t - \eta \left( \nabla_\V f(\U_t \V_t^\top) + \lambda \cdot \nabla_\V g(\U_t^\top \U_t - \V_t^\top \V_t) \right).
\end{align*} The parameter $\lambda > 0$ is arbitrarily chosen.
\item[$(iii)$] Due to this new update rule, a slightly different and proper step size selection should be devised for \algo. Our step size is selected as follows: 
\begin{align}{\label{eq:new_step_size}}
\eta \leq \frac{1}{12 \cdot \max \left\{L,~L_g \right\} \cdot \normt{\begin{bmatrix} \Uinit \\ \Vinit \end{bmatrix}}^2}.
\end{align} Compared to the step size proposed in \citep{bhojanapalli2016dropping} (which is of the same form with \eqref{eq:new_step_size_smooth}), our analysis drops the dependence to $\|\nabla f(\cdot)\|_2$ at the denominator. This is leads to a faster computed step size and highlights the non-necessity of this term for proof of convergence, \emph{i.e.}, the presence of the $\|\nabla f(\cdot)\|_2$ term is sufficient but not necessary.
%\item[$(iv)$] We provide a simpler and improved proof for the case of (restricted) smooth convex $f$, which requires a weaker initial condition than \citep{bhojanapalli2016dropping}.
\end{itemize}

\begin{algorithm}[!t]
	\caption{\algo for smooth and strongly convex $f$}{\label{alg:BFGD}}
	\begin{algorithmic}[1]
		\STATE \textbf{Input:} Function $f$, target rank $r$, \# iterations $T$. 
		\STATE Set initial values for $U_0, V_0$
		\STATE Set step size $\eta$ as in \eqref{eq:new_step_size}.
		\FOR {$t=0$ to $T-1$}
			\STATE $\U_{t+1} = \U_t - \eta \left( \nabla_\U f(\U_t \V_t^\top) - \lambda \cdot \nabla_\U g(\U_t^\top \U_t - \V_t^\top \V_t) \right)$
			\STATE $\V_{t+1} = \V_t - \eta \left( \nabla_\V f(\U_t \V_t^\top) - \lambda \cdot \nabla_\V g(\U_t^\top \U_t - \V_t^\top \V_t) \right)$
		\ENDFOR
		\STATE \textbf{Output:} $\X = \U_T \V_T^\top$. 	
	\end{algorithmic}
\end{algorithm} 
The scheme is described in Algorithm \ref{alg:BFGD}.
%Observe that $\eta$ has similar formula with the step size in \citep{bhojanapalli2016dropping}.
%Though, as we show next, our analysis simplifies further the selection of the step size.\footnote{There is no dependence on the spectral norm of the gradient.}. 
As shown in the next section, constant step size \eqref{eq:new_step_size} is sufficient to lead to attractive convergence rates for \algo, for $f$ $L$-smooth and $\mu$-strongly convex.

\section{Local Convergence for \algo}{\label{sec:guarantees}}
This section includes the main theoretical guarantees of \algo, both for the cases of just smooth $f$, and $f$ being smooth and strongly convex.
% Our results follow and improve upon the results for the square and PSD case from \citep{bhojanapalli2016dropping}. 
 To provide such local convergence results, we assume that there is a known ``good'' initialization which ensures the following.
\begin{center}
\begin{minipage}{.85\textwidth}
\textit{Assumption A1.} \textit{Define $\kappa = \tfrac{\max\{L,L_g\}}{\min\{\mu,\mu_g\}}$ where $\mu_g$ and $L_g$ are the strong convexity and smoothness parameters of $g$, respectively. 
Then, we assume we are provided with a ``good" initialization point $\X_0 = \U_0 \V_0^\top$ such that:} 
\begin{align*}
\dist(\Uinit, \Vinit ; \Xoptr) &\leq \tfrac{\sqrt{2} \cdot \sigma_r(\Xoptr)^{1/2}}{10 \sqrt{\kappa}} \quad \quad \quad \quad \quad \text{(Strongly convex and smooth } f\text{)}.\nonumber \\ 
\text{\textit{For the case where }} f &\text{\textit{ is just smooth, we assume:}} \\
\dist(\Uinit, \Vinit ; \Xoptr) &\leq \tfrac{\sqrt{2} \cdot \sigma_r(\Xoptr)^{1/2}}{10} \quad \quad \quad \quad \quad \text{(Smooth $f$)}.
\end{align*}
\end{minipage}
\end{center}

\medskip
\medskip
For our analysis, we will use the following step size assumptions: 
\begin{align}
\weta &\le \frac{1}{8 \max\{L, L_g\} \cdot \normt{\begin{bmatrix} \U_t \\ \V_t \end{bmatrix}}^2} ~\quad \quad \quad \quad \quad \quad \text{(Strongly convex and smooth } f\text{)}, \label{eq:weta1} \\
\weta &\le \frac{1}{15L \normt{\begin{bmatrix} \U_t \\ \V_t \end{bmatrix}}^2 + 3 \|\gradf(\U_t\V_t^\top)\|_2} ~\quad \quad \quad \text{(Smooth $f$)}.  \label{eq:weta}
\end{align} 
While these step sizes are different than the ones we use in practice, there is a constant-fraction connection between $\weta$ and $\eta$.
\begin{lemma}{\label{lem:equiv_eta}}
Let $(\U_0, \V_0)$ be such that Assumption A1 is satisfied.  Then, \eqref{eq:weta1} holds if \eqref{eq:new_step_size} is satisfied, and \eqref{eq:weta} holds if \eqref{eq:new_step_size_smooth} is satisfied.
\end{lemma}
The proof is provided in the Appendix \ref{sec:proof_equiv_eta}. By this lemma, our analysis below is equivalent---up to constants---to that if we were using the original step size $\eta$ of the algorithm. 
However, for clarity reasons and ease of exposition, we use $\weta$ below.

For the case of strongly convex $f$, both Assumption A1 and the step size depends on the strong convexity and smoothness parameters of $g$. 
When $\mu$ and $L$ are known a priori, this dependency can be removed since one can choose $g$ such that at least $\mu$-restricted strongly convex and at most $L$-smooth. 
Then, $\kappa$ becomes the condition number of $f$, and the step size depends only on $L$. 

%Observe that step sizes in \eqref{eq:weta1} and \eqref{eq:weta} are computationally inefficient in practice: 
%they require at most two spectral norm computations of $\U_t, \V_t$ and $\gradf(\U_t \V_t^\top)$ per iteration. However, as the following lemma states, even in the case where we cannot afford such calculations per iteration, there is a 

\subsection{Linear Local Convergence Rate when $f$ is $L$-Smooth and $\mu$-Strongly Convex}
The following theorem proves that, under proper initialization, \algo admits linear convergence rate, when $f$ is both $L$-smooth and $\mu$-restricted strongly convex.
\begin{theorem} \label{thm:strongcvx}
Suppose that $f$ is $L$-smooth and $\mu$-restricted strongly convex and regularizer $g$ is $L_g$-smooth and $\mu_g$-restricted strongly convex.
Define $\mumin := \min \left\{\mu, \mu_g\right\}$ and $L_{\max} := \max \left\{L, L_g \right\}$. 
Denote the unique minimizer of $f$ as $\Xopt \in \R^{m \times n}$ and assume that $\Xopt$ is of arbitrary rank. 
Let $\weta$ be defined as in \eqref{eq:weta1}. 
If the initial point $(\Uinit, \Vinit)$ satisfies Assumption A1, then \algo algorithm in Algorithm \ref{alg:BFGD} converges linearly to $\Xoptr$, within error $O\left( \sqrt{\frac{\kappa}{\sigma_r(\Xoptr)}} \normf{\Xs - \Xoptr} \right)$, according to the following recursion:
\begin{align} \label{eqn:conv_linear}
\dist(\U_{t+1}, \V_{t+1}; \Xoptr)^2 \le \gamma_t \cdot \dist(\U_t, \V_t; \Xoptr)^2 + \weta L \normf{\Xs - \Xoptr}^2,
\end{align}
for every $t \ge 0$, where the contraction parameter $\gamma_t$ satisfies:
\begin{align*}
\gamma_t = 1 - \weta \cdot \tfrac{\mumin \cdot \sigma_r(\Xoptr)}{8}
&\ge1 - \tfrac{\mumin}{17 \cdot L_{\max}} \cdot \tfrac{\sigma_r(\Xoptr)}{\sigma_1(\Xoptr)} > 0.
%&\ge 1 - \frac{1}{128 \kappa} \cdot \frac{\sigma_r(\Xoptr)}{\sigma_1(\Xoptr)} > 0.
\end{align*}
\end{theorem}
The proof is provided in Section \ref{sec:proof_linear}. 
The theorem states that if $\Xo$ is (approximately) low-rank, the iterates converge to a close neighborhood of $\Xoptr$.

The above result can also be expressed w.r.t. the function value $f(\factor)$, as follows:
\begin{corollary} \label{cor:strongcvx}
Under the same initial condition with Theorem \ref{thm:strongcvx}, Algorithm \eqref{alg:BFGD} satisfies the following recursion w.r.t. the distance of function values from the optimal point:
$$
f(\U_t \V_t^\top) - f(\Xopt) \le \gamma^t \cdot \sigma_1(\Xopt) \cdot \left( f(\Uinit \Vinit^\top) - f(\Xopt) \right) + \frac{\sqrt{\mu L}}{\sigma_r(\Xopt)} \normf{\Xopt - \Xoptr}^2.
$$
\end{corollary}

\subsection{Local Sublinear Convergence}
In Section \ref{sec:bfgd_algo_smooth}, we showed that a lifting technique can reduce our problem \eqref{intro:eq_01} to a rank-constrained semidefinite program, and applying \fgd from \citep{bhojanapalli2016dropping} is exactly \algo \eqref{eqn:nonconvex_gd}. 
While the sublinear convergence guarantee of \fgd can also be applied to our problem, we provide an improved result.
\begin{theorem} \label{thm:smooth}
Suppose that $\f$ is $L$-smooth with a minimizer $\Xs \in \R^{m \times n}$. 
Let $\Xh_r$ be any target rank-$r$ matrix, and let $\weta$ be defined as in \eqref{eq:weta}. 
If the initial point $\X_0 = \U_0\V_0^\top$, $\U_0 \in \R^{m \times r}$ and $\V_0 \in \R^{n \times r}$, satisfies Assumption A1, then \fgd converges with rate $O(1/T)$ to a tolerance value according to:
\begin{align*}
f(U_T V_T^\top) - f(\Uo {\Vo}^\top) = \hat{f}(W_T W_T^\top) - \hat{f}(\Wo {\Wo}^\top) \le \frac{10 \cdot \dist(\Uinit,\Vinit;\Xoptr)^2}{\eta T}
\end{align*}
\end{theorem}
Theorem \ref{thm:smooth} guarantees a local sublinear convergence with a looser initial condition. 
While \citep{bhojanapalli2016dropping} requires $\min_{R \in O(r)} \normf{W - \Wo R} \le \frac{\sigma_r^2(\Wo)}{100\sigma_1^2(\Wo)} \cdot \sigma_r(\Wo)$, our result requires that the initial distance to the $\Wo$ is merely a constant factor of $\sigma_r(\Wo)$.

\section{Initialization}{\label{sec:init}}

In this section, we present initialization procedures for the case where $f$ is strongly convex and smooth.
Our main theorem guarantees linear convergence in the factored space given that the initial point $(\Uinit,\Vinit)$ is within a ball around the closest target factors, with radius $O(\kappa^{-1/2} \sigma_r(\Xoptr)^{1/2})$. To find such a solution, we propose an extension of the initialization in \citep{bhojanapalli2016dropping}.

\begin{lemma} \label{lem:initialize}
Consider an initial solution $\factorinit$ which is the best rank-$r$ approximation of 
\begin{align} \label{eqn:initial_sol}
\Xinit = -\tfrac{1}{L}\gradf(0)
\end{align}
Then we have
\begin{align*}
\normf{\factorinit - \Xoptr} &\le 2\sqrt{2\left(1 - \tfrac{1}{\kappa}\right)} \normf{\Xopt} + 2\normf{\Xopt - \Xoptr}
\end{align*}
\end{lemma}

Combined with Lemma 5.14 in \citep{tu2016low}, which transforms a good initial solution from the original space to the factored space, the following corollary gives one sufficient condition for global convergence of \algo with the SVD of \eqref{eqn:initial_sol} as initialization.
\begin{corollary}\label{cor:init}
If
\begin{align*}
\normf{\Xopt - \Xoptr} \le \frac{\sigma_r(\Xopt)}{100\sqrt{\kappa}}, \quad
\kappa \le 1 + \frac{\sigma_r(\Xoptr)^2}{4608 \cdot \normf{\Xoptr}^2}
\end{align*}
then the initial solution 
$$
\Uinit = A_0 \Sigma_0^{1/2} ,~ \Vinit = B_0 \Sigma_0^{1/2}
$$
where $A_0 \Sigma_0 B_0$ is the SVD of $-\frac{1}{L} \gradf(0)$ satisfies the initial condition of Theorem \ref{thm:strongcvx}.
\end{corollary}

Corollary \ref{cor:init} requires weaker conditions than \citep{bhojanapalli2016dropping} in order Theorem \ref{thm:strongcvx} be transformed to \emph{global guarantees}.
While our theoretical results can only guarantee global convergence for a well-conditioned problem ($\kappa$ close to one), we show in the experiments that the algorithm performs well in practice where the sufficient conditions are yet to be satisfied.

\section{Experiments}{\label{sec:exp}}

In this section, we first provide comparison results regarding the actual computational complexity of SVD and matrix-matrix multiplication procedures; 
while such comparison is not thoroughly complete, it provides some evidence about the gains of optimizing over $U, V$ factors, in lieu of SVD-based rank-$r$ approximations.
We also describe a toy example that highlights the effect of the regularizer $g$ in convergence rates, for strongly convex and smooth $f$.
Next, we provide extensive results on the performance of \algo, as compared with state of the art, for the following problem settings:
$(i)$ affine rank minimization, where the objective is smooth and (restricted) strongly convex,
$(ii)$ image denoising/recovery from a limited set of observed pixels, where the problem can be cast as a matrix completion problem, and
$(iii)$ 1-bit matrix completion, where the objective is just smooth convex. 
In all cases, the task is to recover a low rank matrix from a set of observations, where our machinery naturally applies.

\subsection{The Complexity of SVD and Matrix-Matrix Multiplication in Practice}{\label{sec:svd_mm}}
To provide an idea of how matrix-matrix multiplication scales, in comparison with truncated SVD,\footnote{Here, we consider algorithmic solutions where both SVD and matrix-matrix multiplication computations are performed with high-accuracy. One might consider \emph{approximate} SVD---see the excellent monograph \citep{halko2011finding}---and matrix-matrix multiplication approximations---see \citep{drineas2001fast, drineas2006fast, kyrillidis2014approximate, cohen2015optimal}; we believe that studying such alternatives is an interesting direction to follow for future work.}
we compare it with some state-of-the-art SVD subroutines: 
$(i)$ the Matlab's \texttt{svds} subroutine, based on \texttt{ARPACK} software package \citep{lehoucq1998arpack}, 
$(ii)$ a collection of implicitly-restarted Lanczos methods for fast truncated SVD and symmetric eigenvalue decompositions  (\texttt{irlba}, \texttt{irlbablk}, \texttt{irblsvds}) \citep{baglama2005augmented} \footnote{IRLBA stands for Implicitly Restarted Lanczos Bidiagonalization Algorithms.}, 
$(iii)$ the limited memory block Krylov subspace optimization 
for computing dominant SVDs (\texttt{LMSVD}) \citep{liu2013limited},
and $(iv)$ the \texttt{PROPACK} software package \citep{larsen2004propack}.
We consider random realizations of matrices in $\R^{m \times n}$ (w.l.o.g., assume $m = n$), for varying values of $m$. 
For SVD computations, we look for the best rank-$r$ approximation, for varying values of $r$.
In the case of matrix-matrix multiplication, we record the time required for the computation of 2 matrix-matrix multiplications of matrices $\R^{m \times m}$ and $\R^{m \times r}$, which is equivalent with the computational complexity required in our scheme, in order to avoid SVD calculations.
All experiments are performed in a Matlab environment.

Figure \ref{fig:svd_mm_comp} (left panel) shows execution time results for the algorithms under comparison, as a function of the dimension $m$. 
Rank $r$ is fixed to $r = 100$. 
While both SVD and matrix multiplication procedures are known to have $O(m^2 r)$ complexity, it is obvious that the latter on dense matrices is at least two-orders of magnitude faster than the former. 
In Table \ref{tbl:svd_mm_comp}, we also report the approximation guarantees of some faster SVD subroutines, as compared to \texttt{svds}: while \texttt{irblablk} seems to be faster, it returns a very rough approximation of the singular values, when $r$ is relatively large. 
Similar findings are depicted in Figures \ref{fig:svd_mm_comp} (middle and right panel).

\begin{figure*}
\centering
\includegraphics[width=0.325\textwidth]{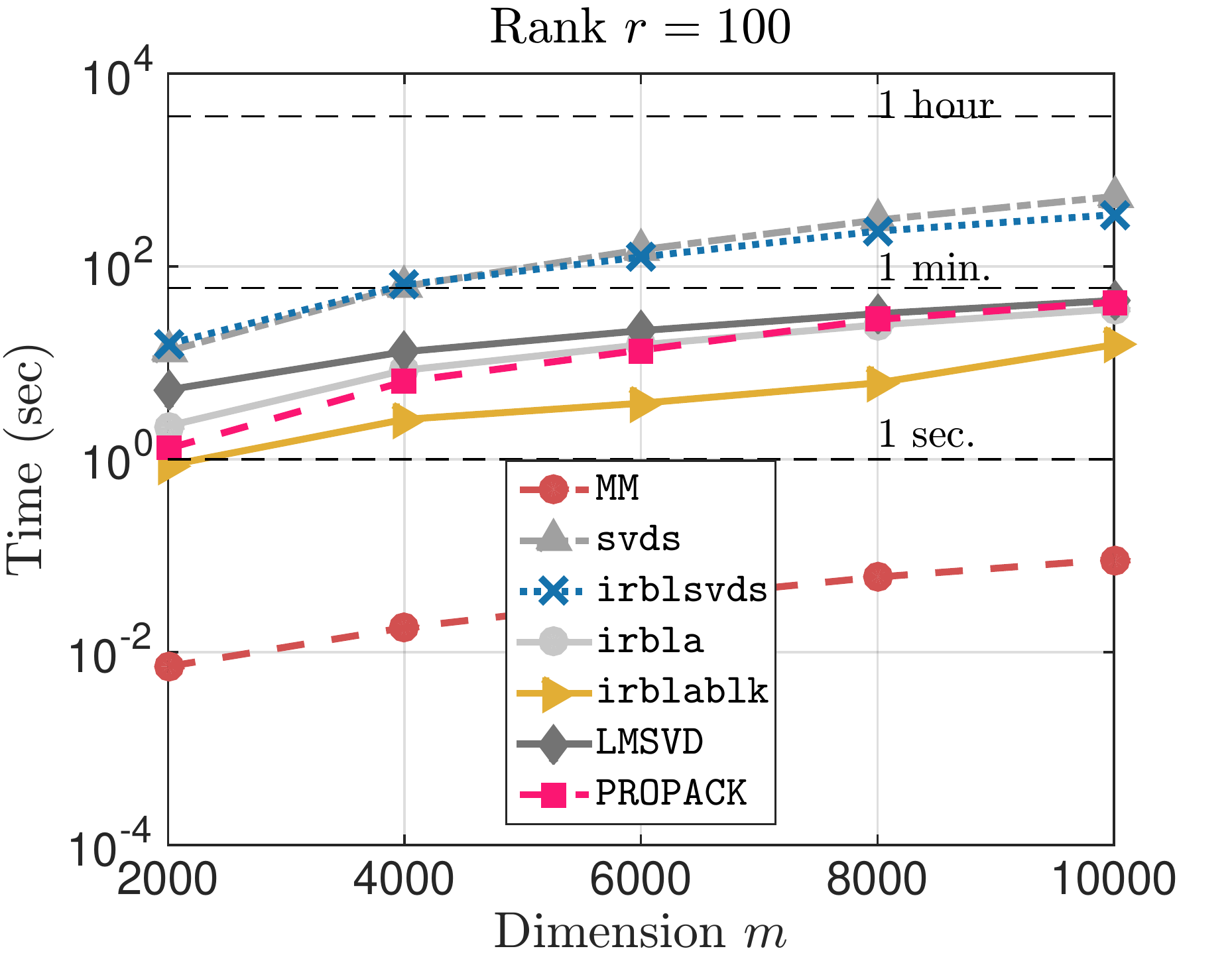}
\includegraphics[width=0.31\textwidth]{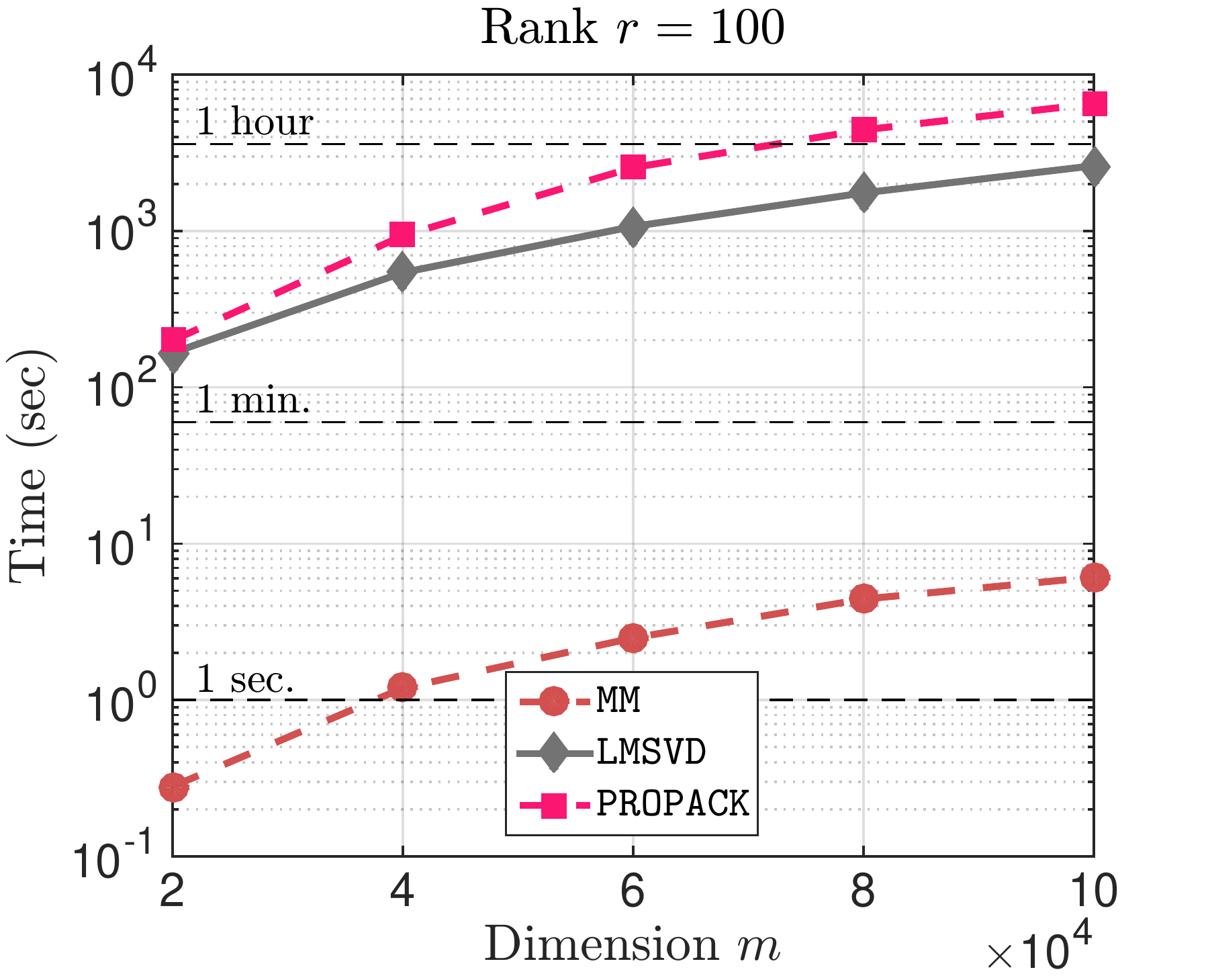}
\includegraphics[width=0.32\textwidth]{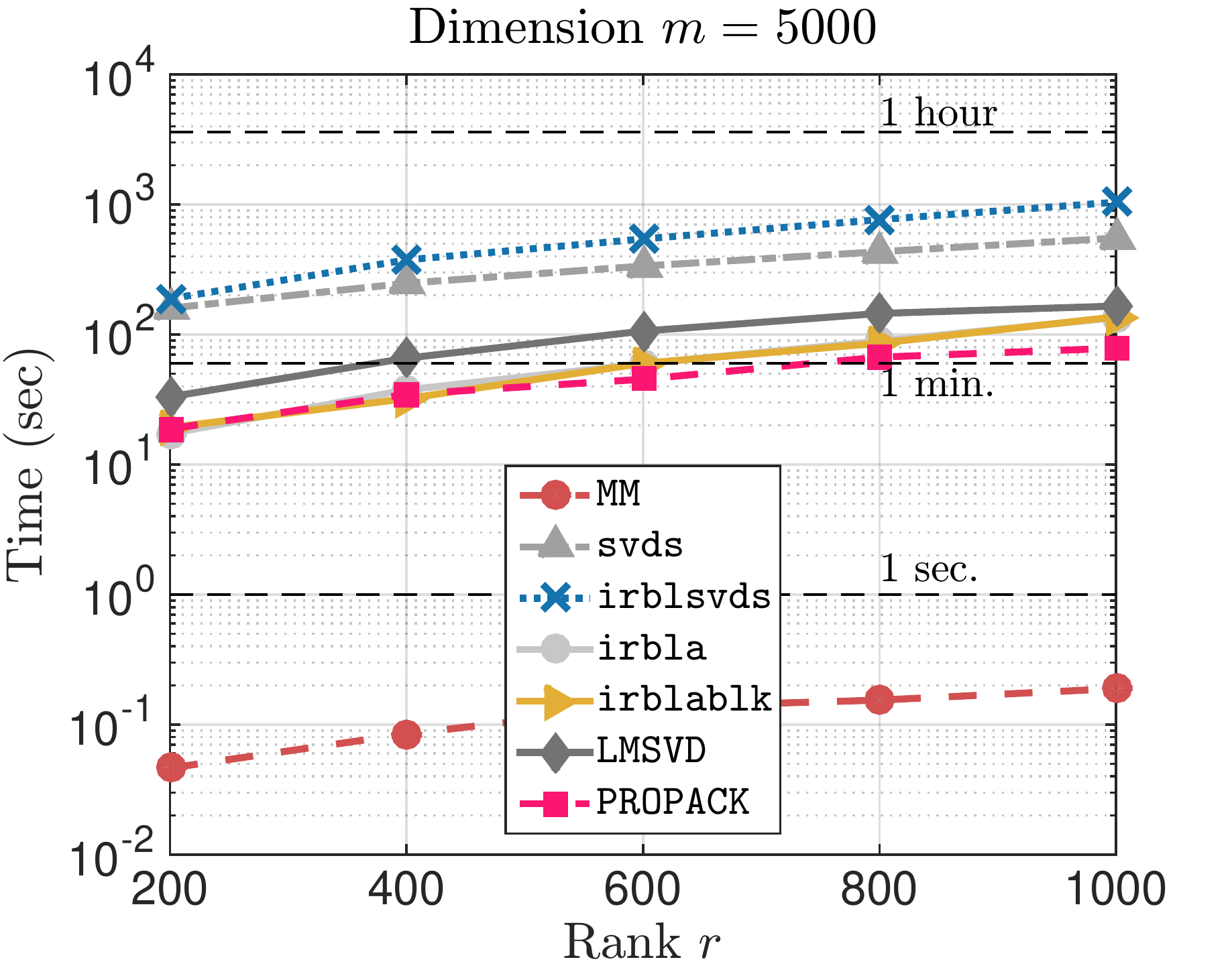}
\caption{Comparison of SVD procedures versus Matrix Matrix (MM) multiplication. \textit{Left panel:}  Varying dimension $m$ and constant rank $r = 100$. \textit{Middle panel:} Similar to left panel where $m$ scales larger and we focus on a subset of SVD algorithms that can scale up. \textit{Right panel:} Varying rank values and constant dimension $m = 5\cdot 10^3$. }\label{fig:svd_mm_comp}
\end{figure*}

\begin{table}[!h]
\centering
\ra{1.3}
\begin{scriptsize}
\rowcolors{2}{white}{black!05!white}
\begin{tabular}{c c c c c c c c c c c} \toprule
Algorithms & & \multicolumn{9}{c}{Error $\tfrac{\|\widehat{\Sigma} - \Sigma^\star\|_F}{\|\Sigma^\star\|_F}$, where $\Sigma^\star$ is diagonal matrix with top $r$ singular values from \texttt{svds}. } \\ 
\midrule
$\widehat{\Sigma}$  & \phantom{a} & \multicolumn{1}{c}{$m=2\cdot 10^3$} & \phantom {a} & \multicolumn{1}{c}{$m=4\cdot 10^3$} & & \multicolumn{1}{c}{$m=6\cdot 10^3$} & \phantom {a} & \multicolumn{1}{c}{$m=8\cdot 10^3$} & & \multicolumn{1}{c}{$m=10^4$} \\
\cmidrule{3-3} \cmidrule{5-5} \cmidrule{7-7} \cmidrule{9-9} \cmidrule{11-11} 
\texttt{irblsvds} & &  3.63e-15 & &  4.33e-09 & &  8.11e-11 & &  4.79e-12 & &  5.82e-10 \\ 
\texttt{irbla} & &  6.00e-15 & &  9.01e-07 & &  1.05e-04 & &  2.99e-04 & &  7.29e-04 \\ 
\texttt{irblablk}  & &  1.48e+03 & &  1.67e+03 & &  1.24e+03 & &  1.45e+03 & &  7.91e+11 \\ 
\texttt{LMSVD} & &  2.14e-14 & &  4.49e-12 & &  3.94e-11 & &  1.33e-10 & &  7.30e-10  \\ 
\texttt{PROPACK} & &  4.10e-12 & &  2.46e-10 & &  1.63e-12 & &  7.90e-12 & &  3.55e-11 \\ 
\bottomrule
\end{tabular}
\end{scriptsize}
\caption{Approximation errors of singular values, in the form $\tfrac{\|\widehat{\Sigma} - \Sigma^\star\|_F}{\|\Sigma^\star\|_F}$. Here, $\widehat{\Sigma}$ denote the diagonal matrix, returned by SVD subroutines, containing $r$ top singular values; we use \texttt{svds} to compute the reference matrix $\Sigma^\star$, that contains top-$r$ singular values of the input matrix. Observe that some algorithms deviate singificantly from the ``ground-truth": this is due to either early stopping (only a subset of singular values could be computed) or due to accumulating approximation error.} \label{tbl:svd_mm_comp}
\end{table}

\subsection{The Role of Regularizer $g$}{\label{sec:regularizer}}
In this section, we provide a simple example that illustrates the role of the regularizer $g$. 
As discussed in Section \ref{sec:bfgd_strcvx}, $g$ forces our algorithm to converge to a well-conditioned factorization of $\Xs$. This regularizer not only enables us to control and guarantee convergence of \algo, but also provides a better convergence rate, as we know next.

Figure \ref{fig:regularizer} (left panel) shows the convergence behavior of \algo, when $f$ and $f+g$ is used, with an ill-conditioned initial point $(\Uinit,\Vinit)$. 
It is obvious from the convergence plot that adding the regularizer results into faster convergence to an optimum. This difference in convergence rate is due to dependency on the condition numbers of $\Us$ and $\Vs$ that the algorithm converges to. 
As shown in Figure \ref{fig:regularizer} (right panel), the algorithm converges to a well-conditioned factorization of $\Xs$, while the condition number is not forced to decrease when there is no regularizer.
\begin{figure*}
\centering
\includegraphics[width=0.39\textwidth]{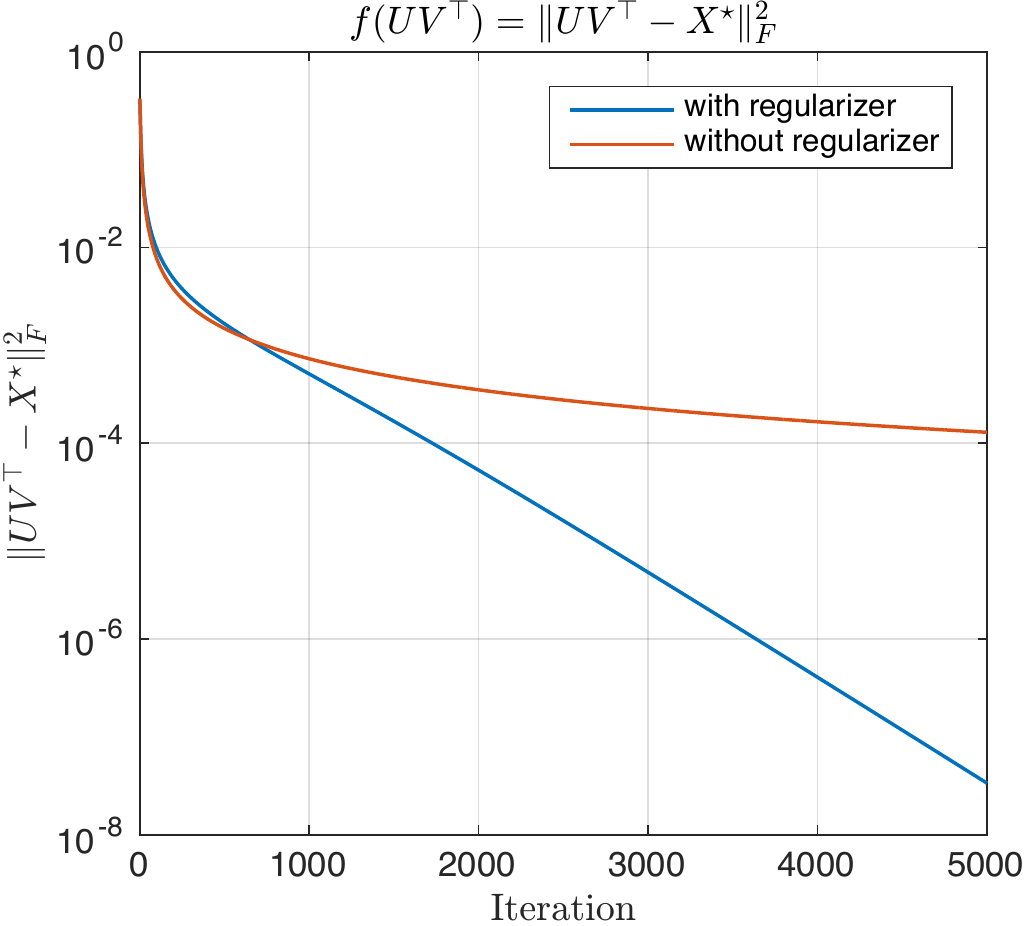}
\includegraphics[width=0.4\textwidth]{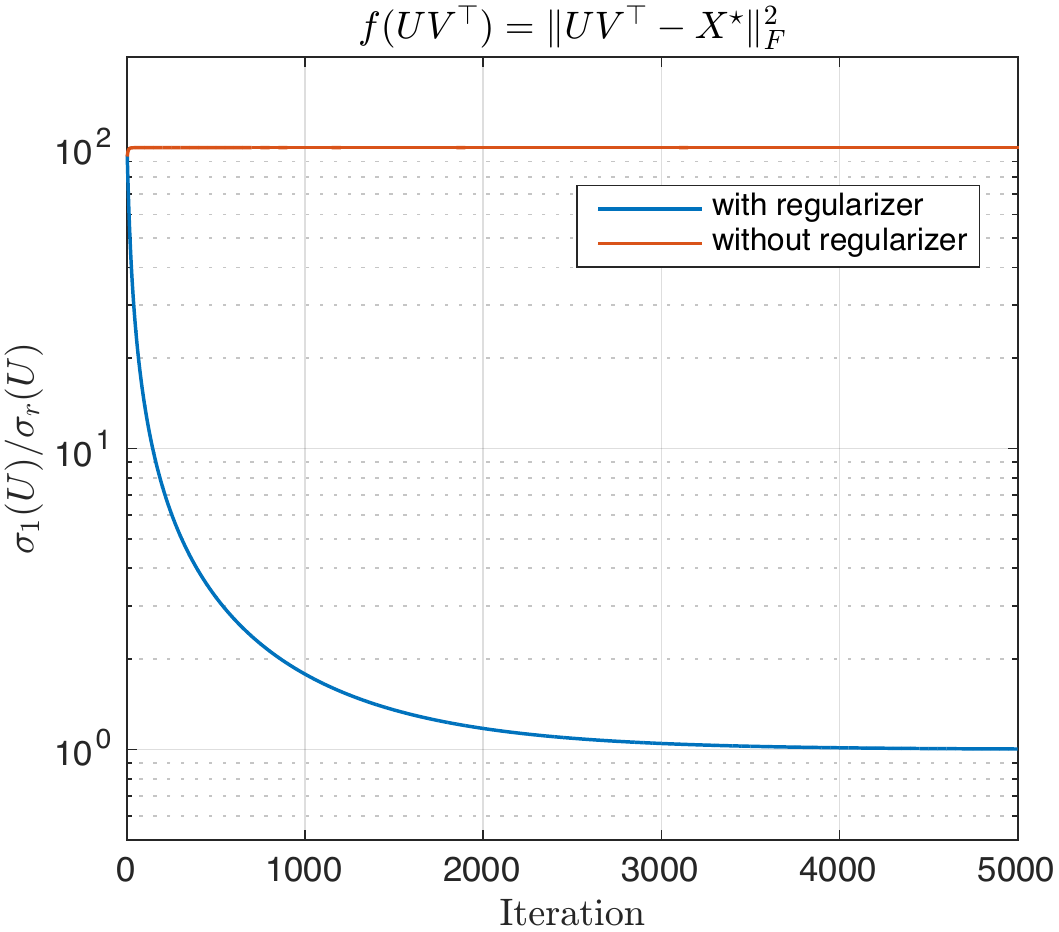}
\caption{ $f(UV^\top) = \normf{UV^\top - \Xs}^2$, where $\Xs = \Us {\Vs}^\top \in \R^{100 \times 100}$, and $\Us, \Vs \in \R^{100 \times 10}$ are orthonormal. The initial point is an ill-conditioned (in particular $\sigma_1(\Uinit)/\sigma_r(\Uinit) = \sigma_1(\Vinit)/\sigma_r(\Vinit) = 10^2$) pair of $X_0$ near $\Xs$. \textit{Left panel:} Convergence behaviors of \algo with a regularizer $g = \tfrac{1}{4} \normf{U^\top U - V^\top V}^2$ and without any regularizer. \textit{Right panel:} The ratios $\sigma_1(U)/\sigma_r(U)$ over iterations.} \label{fig:regularizer}
\end{figure*}

\subsection{Affine Rank Minimization Using Noiselet Linear Maps}{\label{sec:ARM}}

In this task, we consider the problem of \emph{affine rank minimization}. 
In particular, we observe unknown $\Xopt$ through a limited set of observations $y \in \R^p$, that satisfy:
\begin{align}
y = \mathcal{A}\left(\Xopt\right),
\end{align} 
where $\Xopt \in \R^{m \times n}, ~p \ll m \cdot n$, and $\mathcal{A}: \R^{m \times n} \rightarrow \R^{p}$ is a given linear map.
The task is to recover $\Xopt$, using $\mathcal{A}$ and $y$. 
Here, we use permuted and sub-sampled noiselets for the linear operator $\mathcal{A} $, due to their efficient implementation \citep{waters2011sparcs}; similar results can be obtained for $\mathcal{A}$ being a subsampled Fourier linear operator or, even, a random Gaussian linear operator.
For the purposes of this experiment, the ground truth $\Xopt$ is synthetically generated as the multiplication of two tall matrices, $\U^\star \in \R^{m \times r} $ and $\V^\star \in \R^{n \times r}$, 
such that $\Xopt = \U^\star \V^{\star \top}$ and $\|\Xopt\|_F = 1 $. 
Both $\U^\star $ and $\V^\star $ contain random, independent and identically distributed (i.i.d.) Gaussian entries, with zero mean and unit variance. 

\textit{List of algorithms.}
We compare the following state-of-the-art algorithms: 
$(i)$ the Singular Value Projection (\texttt{SVP}) algorithm \citep{jain2010guaranteed}, a non-convex, projected gradient descent algorithm for \eqref{MS:eq_00}, with {\it constant} step size selection (we study the case where $\mu = \sfrac{1}{3} $, as it is the one that showed the best performance in our experiments), 
%$(ii)$ the \textsc{Matrix ALPS II} variant in \citep{kyrillidis2014matrix}, an accelerated, first-order, non-convex algorithm, with adaptive step size and optimized sub-procedures for the criterion in \eqref{MS:eq_00}, 
$(ii)$ the \textsc{SparseApproxSDP} extension to non-square cases for \eqref{MS:eq_02} in \citep{jaggi2010simple}, based on \citep{hazan2008sparse}, where a putative solution is refined via rank-1 updates from the gradient\footnote{\textsc{SparseApproxSDP} in \citep{hazan2008sparse} avoids computationally expensive operations per iteration, such as full SVDs. In theory, at the $r$-th iteration, these schemes guarantee to compute a $\tfrac{1}{r}$-approximate solutio, with rank at most $r$, \emph{i.e.}, achieves a sublinear rate.}, 
$(iii)$ the matrix completion algorithm in \citep{sun2015guaranteed}, which we call \texttt{GuaranteedMC}\footnote{We note that the original algorithm in \citep{sun2015guaranteed} is designed for the matrix \emph{completion} problem, not the matrix \emph{sensing} problem here.}, where the objective is \eqref{MS:trans_eq_00},
$(iv)$ the Procrustes Flow algorithm in \citep{tu2016low} for \eqref{MS:trans_eq_00}, and $(v)$ the \algo algorithm.\footnote{The algorithm in \citep{zhao2015nonconvex} assumes step size that depends on RIP constants, which are NP-hard to compute; since no heuristic is proposed, we do not include this algorithm in the comparison list.}

\textit{Implementation details.}
To properly compare the algorithms in the above list, we preset a set of parameters that are common. 
In all experiments, we fix the number of observations in $y$ to $p = C \cdot n \cdot r$, where $n \geq m$ in our cases, and for varying values of $C$.
All algorithms in comparison are implemented in a \textsc{Matlab} environment, where no \texttt{mex}-ified parts present, apart from those used in SVD calculations; see below.

In all algorithms, we fix the maximum number of iterations to $T = 4000$, unless otherwise stated. 
We use the same stopping criteria for the majority of algorithms as:
\begin{align}
\frac{\|\X_t - \X_{t-1}\|_F}{\|\X_t\|_F} \leq \rm{tol},
\end{align} 
where $\X_t, ~\X_{t-1}$ denote the current and the previous estimates in the $\X$ space and $\rm{tol} := 5\cdot 10^{-6}$. 
For SVD calculations, we use the $\texttt{lansvd}$ implementation in PROPACK package \citep{larsen2004propack}.
For fairness, we modified all the algorithms so that they {\it exploit the true rank} $r$; however, we observed that small deviations from the true rank result in relatively small degradation in terms of the reconstruction performance.\footnote{In case the rank of $\Xopt$ is unknown, one has to predict the dimension of the principal singular space. 
The authors in \citep{jain2010guaranteed}, based on ideas in \citep{keshavan2009gradient}, propose to compute singular values incrementally until a significant gap between singular values is found. For a more recent discussion on how to efficiently estimate the numerical rank of a matrix, refer to \citep{ubaru2016fast}}

In the implementation of \algo, we set $g$ to be $\tfrac{1}{16} \cdot \|\U^\top \U - \V^\top \V\|_F^2$, as suggested in \citep{tu2016low}, for ease of comparison. 
Moreover, for our implementation of Procrustes Flow, we set the constant step size as $\mu := \tfrac{2}{187} \cdot \lbrace \tfrac{1}{\|\U_0\|_F^2}, \tfrac{1}{\|V_0\|_F^2\|} \rbrace$, as suggested in \citep{tu2016low}. 
We use the implementation of \citep{sun2015guaranteed}, with random initialization (unless otherwise stated) and regularization type $\texttt{soft}$, as suggested by their implementation.
In \citep{jaggi2010simple}, we require an upper bound on the nuclear norm of $\Xopt$; in our experiments we assume we know $\|\Xopt\|_*$, which requires a full SVD calculation. 
Moreover, for our experiments, we set the curvature constant for the \textsc{SparseApproxSDP} implementation to its true value $C_f = 1$.

For initialization, we consider the following settings: 
$(i)$ random initialization, where $\Xinit = \U_0 \V_0^\top$ for some randomly selected $\U_0$ and $\V_0$ such that $\|\Xinit\|_F = 1$, 
and $(ii)$ specific initialization, as suggested in each of the papers above.
Our specific initialization is based on the discussion in Section \ref{sec:init}, where $\Xinit = \proj{r}(-\frac{1}{L}\gradf(0))$.
Algorithms SVP, 
%\textsc{Matrix ALPS II}, 
\textsc{SparseApproxSDP} and the solver in \citep{sun2015guaranteed} work with random initialization.
For the initialization phase of \citep{tu2016low}, we consider two cases: $(i)$ the condition number $\kappa$ is known, where according to Theorem 3.3 in \citep{tu2016low}, we require $T_{\text{init}} := \lceil 3 \log (\sqrt{r} \cdot \kappa) + 5\rceil$ \textsc{SVP} iterations\footnote{Observe that setting $T_{\text{init}} = 1$ leads to spectral method initialization and the algorithm in \citep{zheng2015convergent} for non-square cases, given sufficient number of samples.}, and $(ii)$ the condition number $\kappa$ is unknown, where we use Lemma 3.4 in \citep{tu2016low}.

\textit{Results using random initialization. }
Figure \ref{fig:app_exp1} depicts the convergence performance of the above algorithms w.r.t. total execution time. 
Top row corresponds to the case $m = n = 1024$, bottom row to the case $m = 2048, ~n = 4096$. 
For all cases, we fix $r = 50$; from left to right, we decrease the number of available measurements, by decreasing the constant $C$. 
%\textsc{Matrix ALPS II} shows the best performance in terms of execution time: while still using SVD routines per iteration, \textsc{Matrix ALPS II} is specialized to solve matrix sensing problem instances and performs several subroutines per iteration (subspace exploration, debias steps, adaptive step size selection, among others).
%However, \textsc{Matrix ALPS II} applies only to this problem. 
%Compared to \textsc{Matrix ALPS II}, 
\algo shows the best performance, compared to the rest of the algorithms. 
It is notable that \algo performs better than \textsc{SVP}, by avoiding SVD calculations and employing a better step size selection.\footnote{If our step size is used in \texttt{SVP}, we get slightly better performance, but not in a universal manner. }
For this setting, \texttt{GuaranteedMC} converges to a local minimum while \textsc{SparseApproxSDP} and Procrustes Flow show close to sublinear convergence rate. 

\begin{figure*}[t!]
\centering
\includegraphics[width=0.32\textwidth]{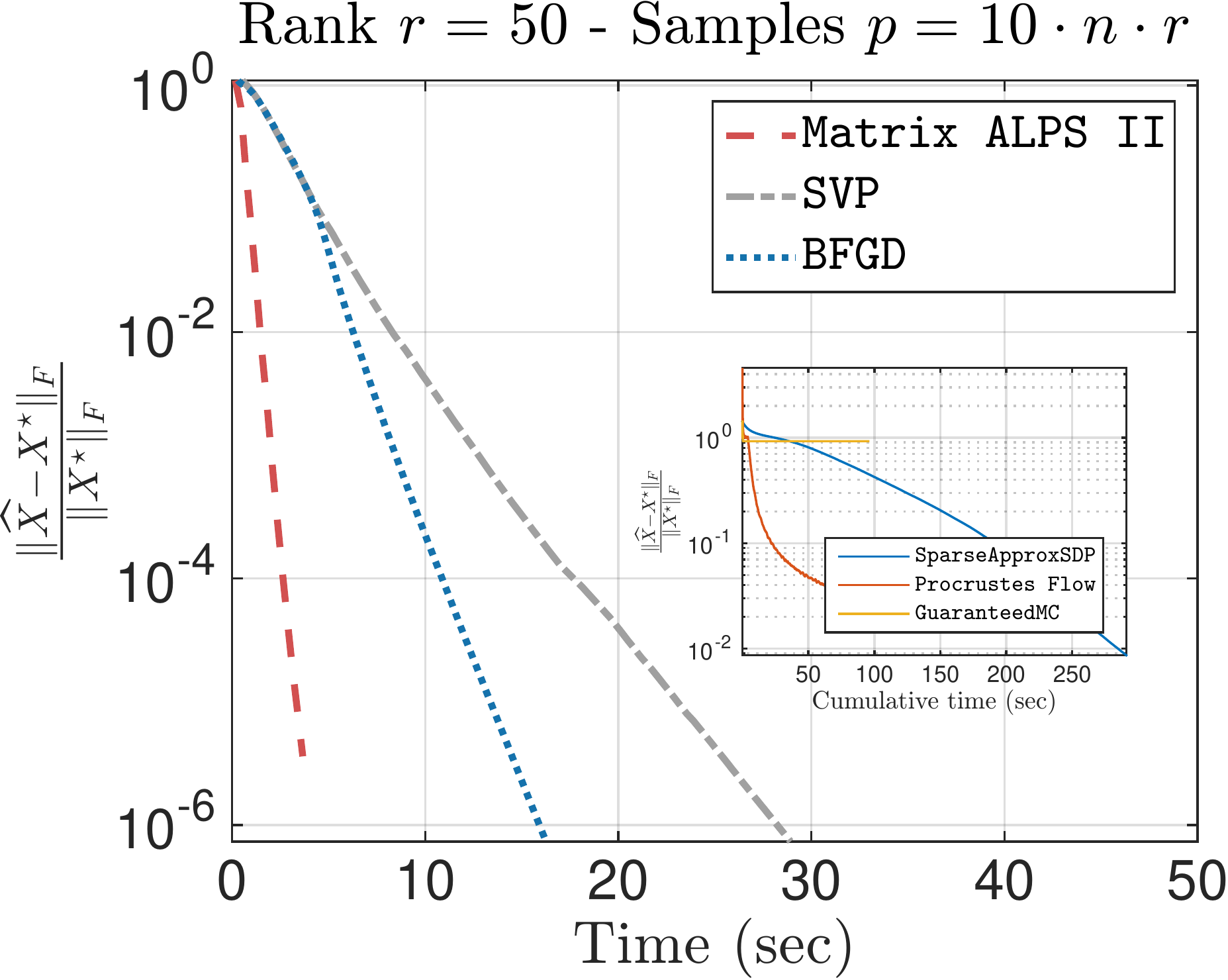} 
\includegraphics[width=0.33\textwidth]{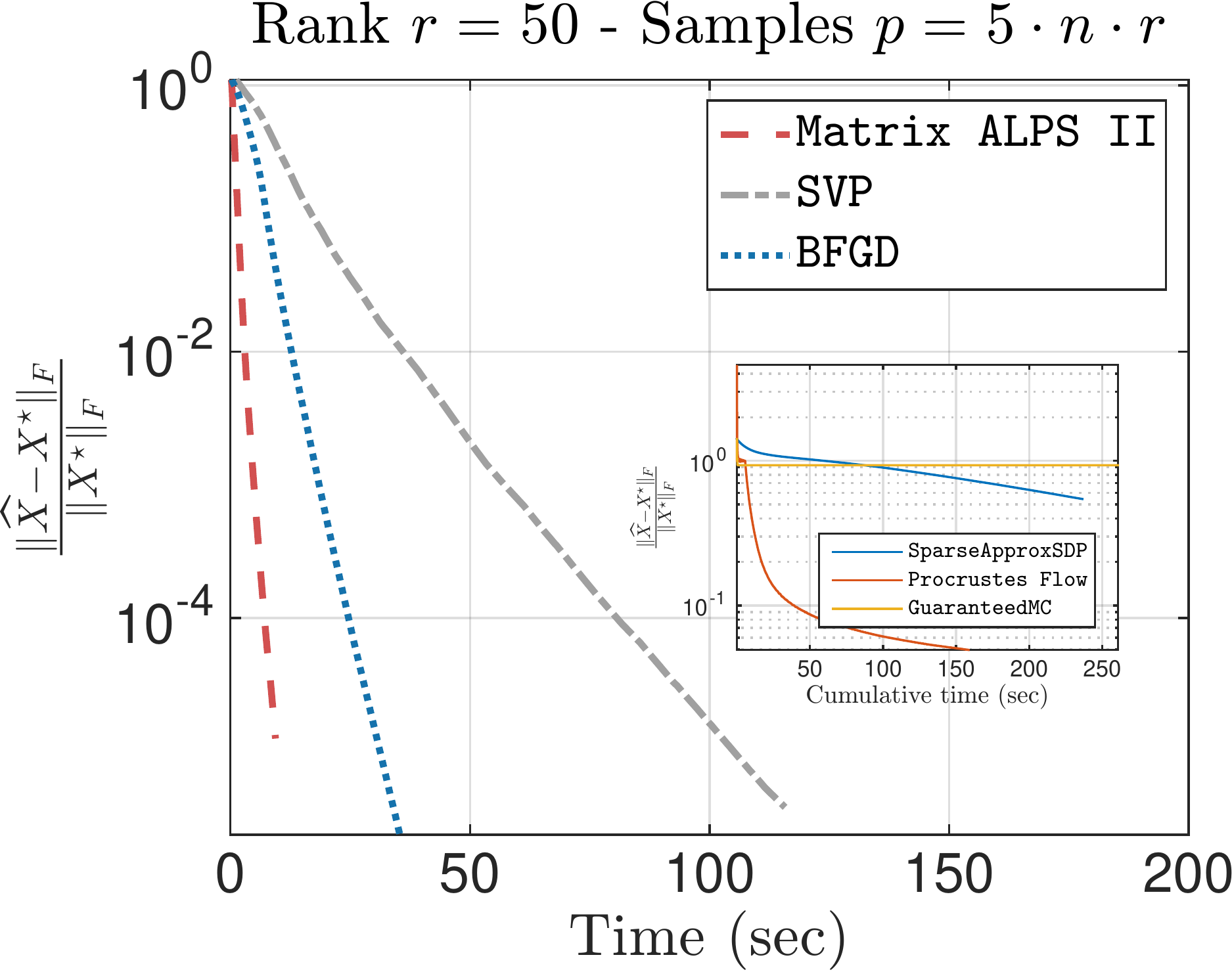} 
\includegraphics[width=0.33\textwidth]{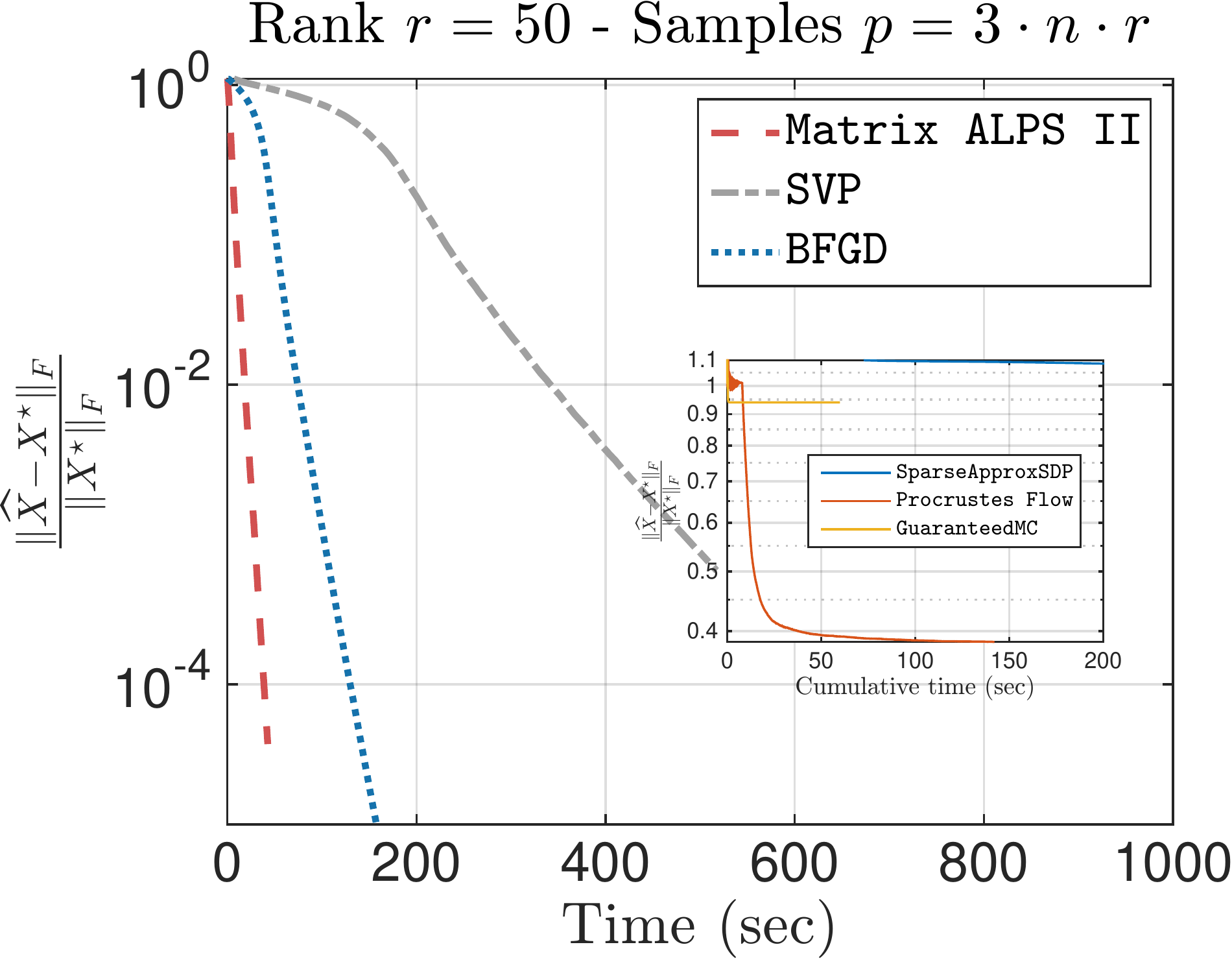} \\
\includegraphics[width=0.32\textwidth]{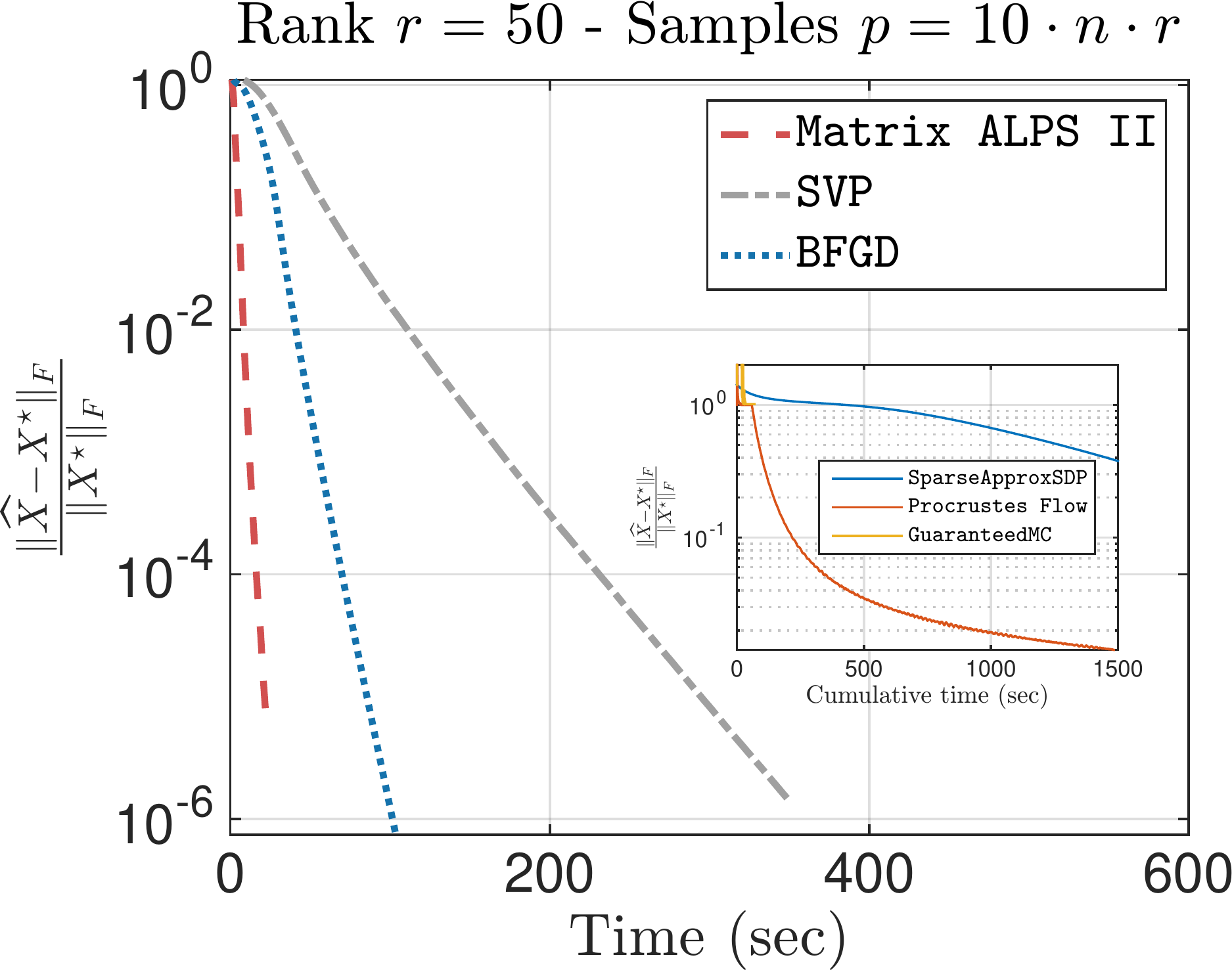} 
\includegraphics[width=0.32\textwidth]{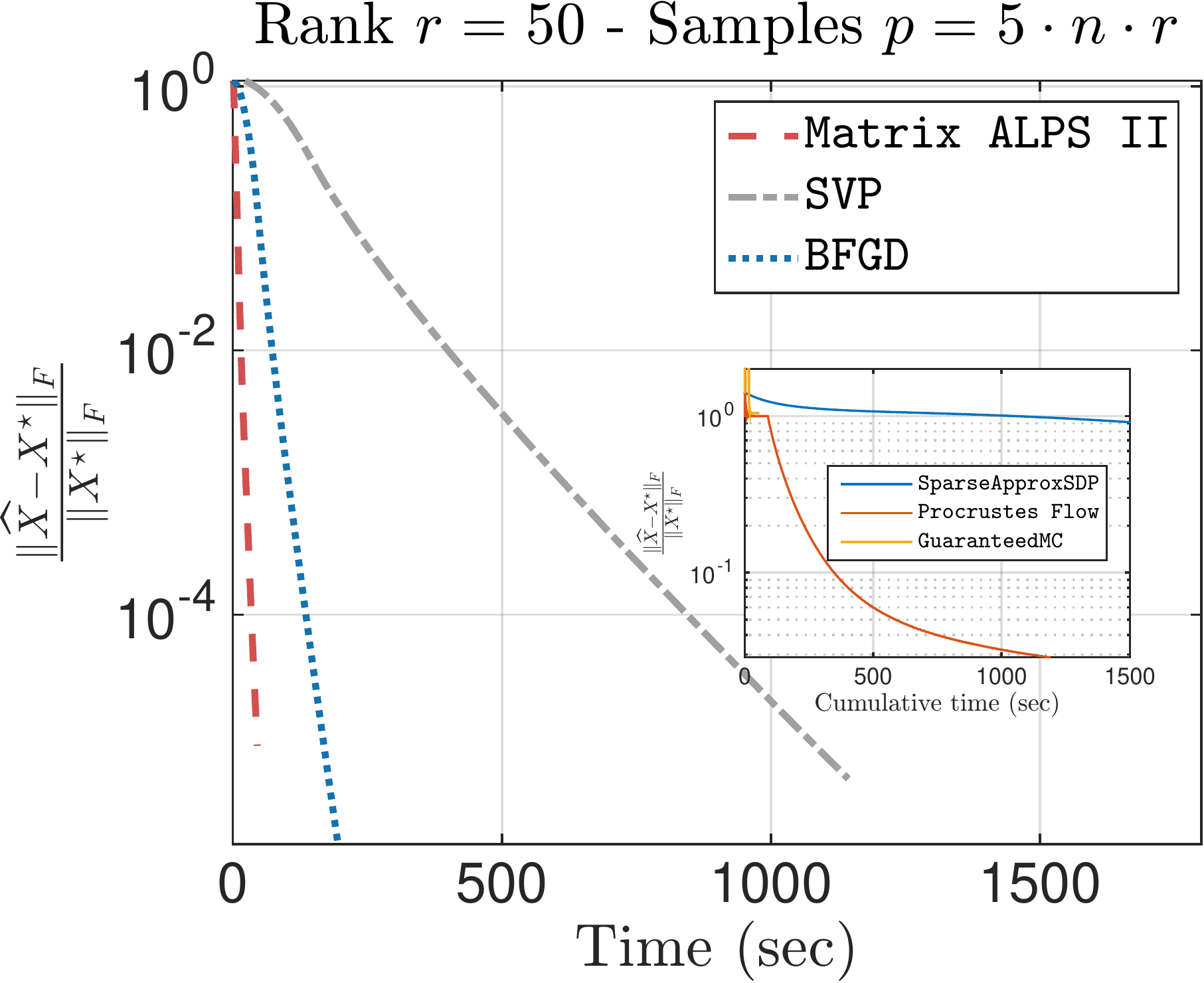} 
\includegraphics[width=0.32\textwidth]{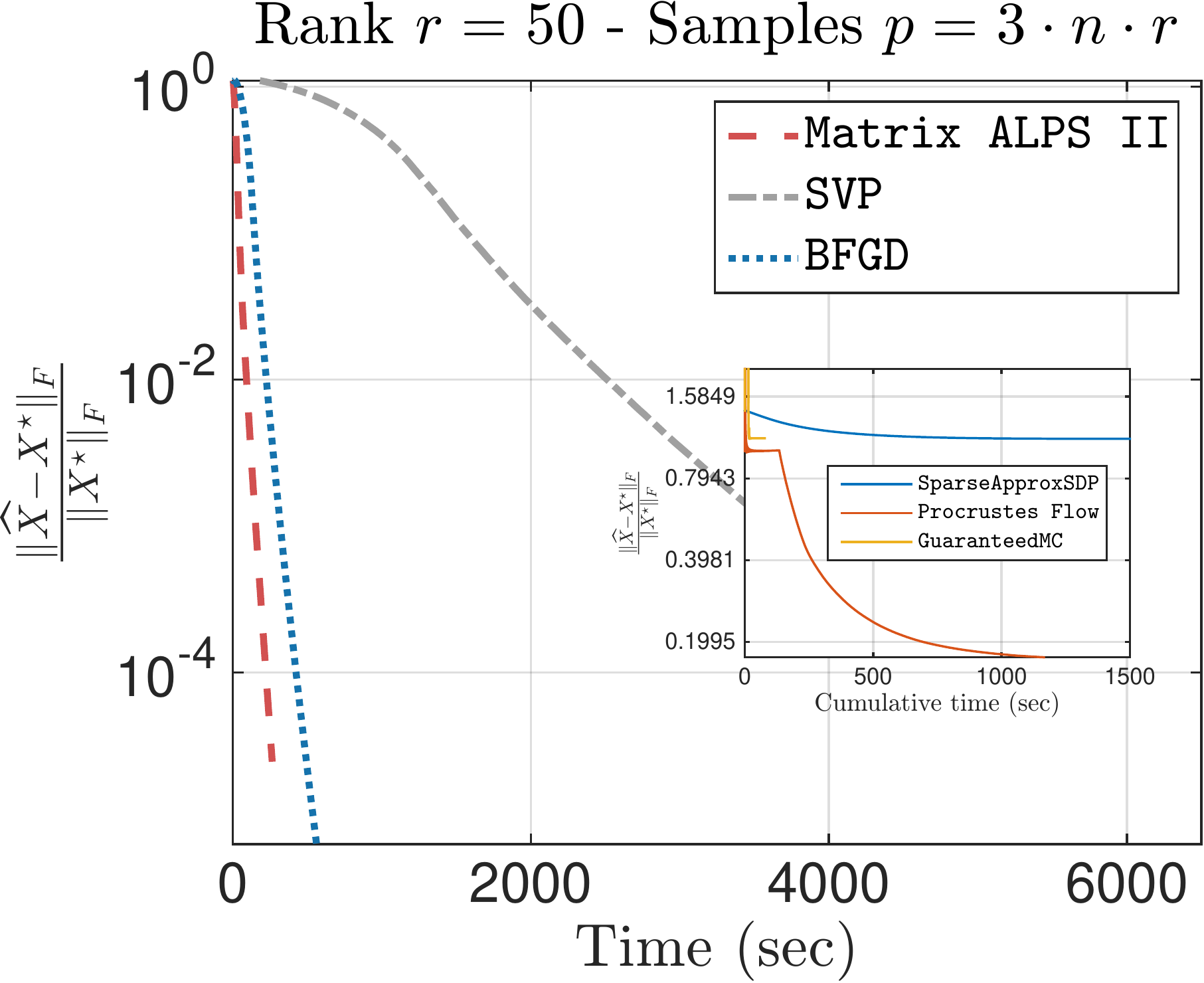} \\
\caption{Convergence performance of algorithms under comparison w.r.t. $\tfrac{\|\widehat{X} - \X^\star\|_F}{\|\X^\star\|_F}$ vs. the total execution time. Top row corresponds to dimensions $m = n = 1024$; bottom row corresponds to dimensions $m = 2048, ~n = 4096$. Details on problem configuration are given on plots' title. For all cases, we used $\mathcal{A}$ as noiselets and $r = 50$.
}
\label{fig:app_exp1}
\end{figure*}

To further show how the performance of each algorithms scales as dimension increases, we provide aggregated results in Tables \ref{tbl:Comp1}-\ref{tbl:Comp2}. 
Observe that \algo is one order of magnitude faster than the rest non-convex factorization algorithms. %, while being competitive with \textsc{Matrix ALPS II} algorithm. 
Table \ref{tbl:Comp3} shows the \emph{median time} per iteration, spent by each algorithm, for both problem instances and $C = 3$.
Observe that 
%\textsc{Matrix ALPS II} and 
\textsc{SVP} requires one order of magnitude more time to complete one iteration, mostly due to the SVD step. 
In stark contrast, all factorization-based approaches spend less time per iteration, as was expected by the discussion in Section \ref{sec:svd_mm}. % however, less progress is achieved by performing only matrix-matrix computations. 

\begin{table*}[!h]
\centering
\ra{1.3}
\begin{tiny}
\rowcolors{2}{white}{black!05!white}
\begin{tabular}{c c c c c c c c c c c c c} \toprule
& \phantom{a} & \multicolumn{3}{c}{$r = 50$, $C = 10$} & \phantom {a} & \multicolumn{3}{c}{$r = 50$, $C = 5$} & \phantom {a} & \multicolumn{3}{c}{$r = 50$, $C = 3$} \\
\cmidrule {3-5} \cmidrule{7-9} \cmidrule{11-13} 
Algorithm & \phantom{a} & $\tfrac{\|\widehat{\X} - \X^\star\|_F}{\|\X^\star\|_F}$ & \phantom{a}  & Total time 
			   & \phantom{a} & $\tfrac{\|\widehat{\X} - \X^\star\|_F}{\|\X^\star\|_F}$ & \phantom{a}  & Total time 
			   & \phantom{a} & $\tfrac{\|\widehat{\X} - \X^\star\|_F}{\|\X^\star\|_F}$ & \phantom{a}  & Total time  \\
\cmidrule{1-1} \cmidrule {3-3} \cmidrule{5-5} \cmidrule{7-7} \cmidrule{9-9} \cmidrule {11-11} \cmidrule{13-13}
%\texttt{Matrix ALPS II} & &  1.6096e-06 & &  3.9296 & & 1.0070e-05 & &  9.6347 & & 3.7877e-05 & &  42.9211  \\ 
\texttt{SVP} & &  6.8623e-07 & &  29.0563 & & 3.7511e-06 & &  115.9088 & & 5.7362e-04 & &  517.5673  \\ 
\texttt{Procrustes Flow} & &  8.6546e-03 & &  291.0442 & & 5.4418e-01 & &  236.4496 & & 1.0831e+00 & &  223.2486  \\ 
\texttt{SparseApproxSDP} & &  1.5616e-02 & &  223.1522 & & 4.9298e-02 & &  158.5459 & & 3.8444e-01 & &  141.5906  \\ 
\texttt{GuaranteedMC} & &  9.2570e-01 & &  95.5135 & & 9.3168e-01 & &  260.8471 & & 9.3997e-01 & &  59.4259  \\ 
\texttt{BFGD} & &  7.0830e-07 & &  16.2818 & & 2.3199e-06 & &  35.3988 & & 1.1575e-05 & &  157.6610  \\ 
\bottomrule
\end{tabular}
\end{tiny}
\caption{Summary of comparison results for reconstruction and efficiency. Here, $m = n = 1024$, resulting into $1,048,576$ variables to optimize,  and $\mathcal{A}$ is a noiselet-based subsampled linear map. The number of samples $p$ satisfies $p = C \cdot n \cdot r$ for various values of constant $C$. Time reported is in seconds.} \label{tbl:Comp1}
\end{table*}

\begin{table*}[!h]
\centering
\ra{1.3}
\begin{tiny}
\rowcolors{2}{white}{black!05!white}
\begin{tabular}{c c c c c c c c c c c c c} \toprule
& \phantom{a} & \multicolumn{3}{c}{$r = 50$, $C = 10$} & \phantom {a} & \multicolumn{3}{c}{$r = 50$, $C = 5$} & \phantom {a} & \multicolumn{3}{c}{$r = 50$, $C = 3$} \\
\cmidrule {3-5} \cmidrule{7-9} \cmidrule{11-13} 
Algorithm & \phantom{a} & $\tfrac{\|\widehat{\X} - \X^\star\|_F}{\|\X^\star\|_F}$ & \phantom{a}  & Total time 
			   & \phantom{a} & $\tfrac{\|\widehat{\X} - \X^\star\|_F}{\|\X^\star\|_F}$ & \phantom{a}  & Total time 
			   & \phantom{a} & $\tfrac{\|\widehat{\X} - \X^\star\|_F}{\|\X^\star\|_F}$ & \phantom{a}  & Total time  \\
\cmidrule{1-1} \cmidrule {3-3} \cmidrule{5-5} \cmidrule{7-7} \cmidrule{9-9} \cmidrule {11-11} \cmidrule{13-13}
%\texttt{Matrix ALPS II} & &  2.9961e-06 & &  23.7907 & & 7.3350e-06 & &  46.7165 & & 2.2546e-05 & &  266.1574 \\ 
\texttt{SVP} & &  1.4106e-06 & &  349.7021 & & 5.6928e-06 & &  1144.7489 & & 1.4110e-04 & &  4703.2012 \\ 
\texttt{Procrustes Flow} & &  2.1174e-01 & &  1909.4748 & & 8.7944e-01 & &  1653.7020 & & 1.1088e+00 & &  1692.7454 \\ 
\texttt{SparseApproxSDP} & &  1.4268e-02 & &  1484.2220 & & 2.8839e-02 & &  1187.5287 & & 1.7544e-01 & &  1165.4210 \\ 
\texttt{GuaranteedMC} & &  1.0114e+00 & &  69.2279 & & 1.0465e+00 & &  53.1636 & & 1.1147e+00 & &  78.2304 \\ 
\texttt{BFGD} & &  6.9736e-07 & &  103.8387 & & 1.7946e-06 & &  195.5111 & & 6.6787e-06 & &  561.8248 \\ 
\bottomrule
\end{tabular}
\end{tiny}
\caption{Summary of comparison results for reconstruction and efficiency. Here, $m = 2048, ~n = 4096$, resulting into $8,388,608$ variables to optimize, and $\mathcal{A}$ is a noiselet-based subsampled linear map. The number of samples $p$ satisfies $p = C \cdot n \cdot r$ for various values of constant $C$. Time reported is in seconds.} \label{tbl:Comp2}
\end{table*}

\begin{table*}[!h]
\centering
\ra{1.3}
\begin{scriptsize}
\rowcolors{2}{white}{black!05!white}
\begin{tabular}{c c c c c} \toprule
& \phantom{a} & \multicolumn{1}{c}{$m = n = 1024$, $C = 3$} & \phantom {a} & \multicolumn{1}{c}{$m = 2048, ~n = 4096$, $C = 3$} \\
%\cmidrule {3-3} \cmidrule{5-5}
Algorithm & \phantom{a} & Median time per iter. & \phantom{a}  & Median time per iter. \\	   
\cmidrule{1-1} \cmidrule {3-3} \cmidrule{5-5}
%\texttt{Matrix ALPS II} & &  2.360e-01 & &  2.461e+00 \\ 
\texttt{SVP} & &  1.604e-01 & &  1.040e+00 \\ 
\texttt{Procrustes Flow} & &  5.871e-02 & &  4.525e-01 \\ 
\texttt{SparseApproxSDP} & &  3.407e-02 & &  3.001e-01 \\ 
\texttt{GuaranteedMC} & &  7.142e-02 & &  4.059e-01 \\ 
\texttt{BFGD} & &  5.337e-02 & &  3.981e-01 \\
\bottomrule
\end{tabular}
\end{scriptsize}
\caption{Median time per iteration. Time reported is in seconds.} \label{tbl:Comp3}
\end{table*}

\textit{Results using specific initialization.}\footnote{\citep{bhojanapalli2016global} recently proved that random initialization is sufficient to lead to the optimum $\X^\star$ for matrix sensing problems, while operating on the factors for the case where $\X^\star \succeq 0$, \emph{i.e.}, $X^\star$ is square. We conjecture similar results can be proved for the non-square case, but we still consider specific initializations for completeness.} 
In this case, we study the effect of initialization in the convergence performance of each algorithm.
To do so, we focus only on the factorization-based algorithms: Procrustes Flow, \texttt{GuaranteedMC}, and \algo.
We consider two problem cases: $(i)$ all these schemes use \emph{our initialization procedure}, and $(ii)$ each algorithm uses its own suggested initialization procedure. 
The results are depicted in Tables \ref{tbl:Comp4}-\ref{tbl:Comp5}, respectively. 

Using our initialization procedure for all algorithms, we observe that both Procrustes Flow and \texttt{GuaranteedMC} schemes can compute an approximation $\widehat{X}$ such that $\tfrac{\|\widehat{\X} - \Xopt\|_F}{\|\Xopt\|_F} > 10^{-1}$. 
In contrast, our approach achieves a solution $\widehat{X}$ that is close to the stopping criterion, \emph{i.e.}, $\tfrac{\|\widehat{\X} - \Xopt\|_F}{\|\Xopt\|_F} \approx 10^{-6}$.
\begin{table*}[!h]
\centering
\ra{1.3}
\begin{scriptsize}
\rowcolors{2}{white}{black!05!white}
\begin{tabular}{c c c c c c c c c} \toprule
& \phantom{a} & \multicolumn{3}{c}{$m = n = 1024$, $C = 10$, $r = 50$} & \phantom {a} & \multicolumn{3}{c}{$m = n = 1024$, $C = 10$, $r = 5$} \\
%\cmidrule {3-3} \cmidrule{5-5}
Algorithm & \phantom{a} & $\tfrac{\|\widehat{\X} - \X^\star\|_F}{\|\X^\star\|_F}$ &  & Total time & \phantom{a} & $\tfrac{\|\widehat{\X} - \X^\star\|_F}{\|\X^\star\|_F}$ &  & Total time\\	   
\cmidrule{1-1} \cmidrule {3-3} \cmidrule{5-5} \cmidrule{7-7} \cmidrule{9-9}
\texttt{Procrustes Flow} & &  2.2703e+01 & &  281.2095 & & 4.0432e+01 & &  192.0993 \\ 
\texttt{GuaranteedMC} & &  9.2570e-01 & &  96.8512 & & 4.7646e-01 & &  2.4792  \\ 
\texttt{BFGD} & &  3.7055e-06 & &  52.5205 & & 8.1246e-06 & &  65.4926 \\ 
\bottomrule
\end{tabular}
\end{scriptsize}
\caption{Summary of results of factorization algorithms using our proposed initialization. } \label{tbl:Comp4}
\end{table*}

Using different initialization schemes per algorithm, the results are depicted in Table \ref{tbl:Comp5}. 
We remind that \texttt{GuaranteedMC} is designed for matrix completion tasks, where the linear operator is a selection mask of the entries.
Observe that Procrustes Flow's performance improves significantly by using their proposed initialization: the idea is to perform \textsc{SVP} iterations to get to a good initial point; then switch to non-convex factored gradient descent for low per-iteration complexity.
However, this initialization is computationally expensive: Procrustes Flow might end up performing several \textsc{SVP} iterations. 
This can be observed \emph{e.g.}, in the case $m = n = 1024, ~r = 5$ and comparing the results in Tables \ref{tbl:Comp4}-\ref{tbl:Comp5}: for this case, Procrustes Flow performs $T = 4000$ iterations when our initialization is used and spends $\sim 200$ seconds, while in Table \ref{tbl:Comp5} it performs $T \ll 4000$ iterations, at least $20\%$ of them using \textsc{SVP}, and consumes $\sim 2000$ seconds.

\begin{table*}[!h]
\centering
\ra{1.3}
\begin{scriptsize}
\rowcolors{2}{white}{black!05!white}
\begin{tabular}{c c c c c c c c c} \toprule
& \phantom{a} & \multicolumn{3}{c}{$m = n = 1024$, $C = 10$, $r = 50$} & \phantom {a} & \multicolumn{3}{c}{$m = n = 1024$, $C = 10$, $r = 5$} \\
%\cmidrule {3-3} \cmidrule{5-5}
Algorithm & \phantom{a} & $\tfrac{\|\widehat{\X} - \X^\star\|_F}{\|\X^\star\|_F}$ &  & Total time & \phantom{a} & $\tfrac{\|\widehat{\X} - \X^\star\|_F}{\|\X^\star\|_F}$ &  & Total time\\	   
\cmidrule{1-1} \cmidrule {3-3} \cmidrule{5-5} \cmidrule{7-7} \cmidrule{9-9}
\texttt{Procrustes Flow} & &  3.2997e-05 & &  390.6830 & & 8.5741e-04 & &  2017.7942 \\ 
\texttt{GuaranteedMC} & &  9.2570e-01 & &  114.9332 & & 1.0114e+00 & &  68.1775 \\ 
\texttt{BFGD} & &  3.6977e-06 & &  64.2690 & & 3.1471e-06 & &  74.2345 \\ 
\midrule
& \phantom{a} & \multicolumn{3}{c}{$m = 2048,n=4096$, $C = 10$, $r = 50$} & \phantom {a} & \multicolumn{3}{c}{$m = 2048,n=4096$, $C = 10$, $r = 5$} \\
%\cmidrule {3-3} \cmidrule{5-5}
Algorithm & \phantom{a} & $\tfrac{\|\widehat{\X} - \X^\star\|_F}{\|\X^\star\|_F}$ &  & Total time & \phantom{a} & $\tfrac{\|\widehat{\X} - \X^\star\|_F}{\|\X^\star\|_F}$ &  & Total time\\	   
\cmidrule{1-1} \cmidrule {3-3} \cmidrule{5-5} \cmidrule{7-7} \cmidrule{9-9}
\texttt{Procrustes Flow} & &  4.9896e-02 & &  265.2787 & & 4.2263e-02 & &  1497.6867 \\ 
\texttt{GuaranteedMC} & &  4.7646e-01 & &  4.0752 & & 1.0302e+00 & &  35.0559 \\ 
\texttt{BFGD} & &  8.1381e-06 & &  83.3411 & & 5.8428e-06 & &  379.1430 \\ 
\bottomrule
\end{tabular}
\end{scriptsize}
\caption{Summary of results of factorization algorithms using each algorithm's proposed initialization. } \label{tbl:Comp5}
\end{table*}

%As a concluding remark, we note that similar results have been observed in noisy settings and, thus, are omitted.

\subsection{Image Denoising as Matrix Completion Problem}{\label{sec:image_exp}}

In this example, we consider the matrix completion setting for an image denoising task: 
In particular, we observe a limited number of pixels from the original image and perform a low rank approximation based only on the set of measurements---similar experiments can be found in \citep{kyrillidis2014matrix, wen2012solving}. 
We use real data images: While the true underlying image might not be low-rank, we apply our solvers to obtain low-rank approximations.

Figures \ref{fig:app_exp2}-\ref{fig:app_exp4} depict the reconstruction results for three image cases. 
In all cases, we compute the best $100$-rank approximation of each image (see \emph{e.g.}, the top middle image in Figure \ref{fig:app_exp2}, where the full set of pixels is observed) and we observe only the $35\%$ of the total number of pixels, randomly selected---a realization is depicted in the top right plot in Figure \ref{fig:app_exp2}. 
Given a fixed common tolerance level and the same stopping criterion as before, the top rows of Figures \ref{fig:app_exp2}-\ref{fig:app_exp4} show the recovery performance achieved by a range of algorithms under consideration---the peak signal to noise ration (PSNR), depicted in dB, corresponds to median values after 10 Monte-Carlo realizations. 
%In all cases, we note that \textsc{Matrix ALPS II} has overall slightly better performance as compared to the rest of the algorithms, as a more specialized algorithms for matrix completion problems. 
Our algorithm shows competitive performance compared to simple gradient descent schemes as \texttt{SVP} and Procrustes Flow, while being a fast and scalable solver. 
Table \ref{tbl:Comp6} contains timing results from 10 Monte Carlo random realizations for all image cases.

\begin{figure*}[t!]
\centering
\includegraphics[width=0.32\textwidth]{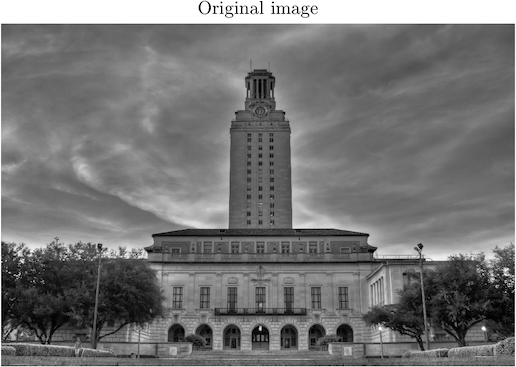} 
\includegraphics[width=0.32\textwidth]{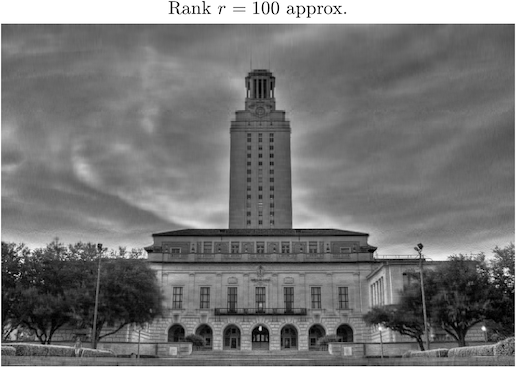} 
\includegraphics[width=0.32\textwidth]{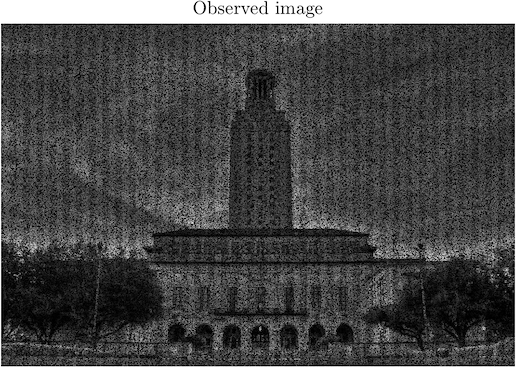} \\
\includegraphics[width=0.32\textwidth]{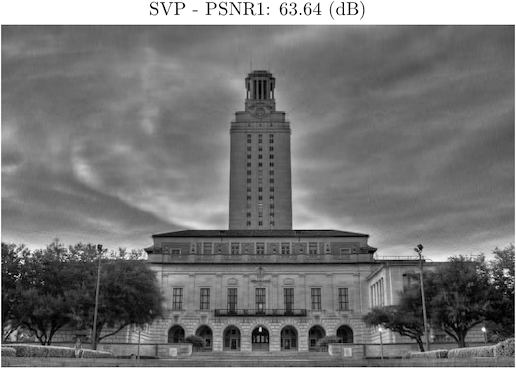} 
\includegraphics[width=0.32\textwidth]{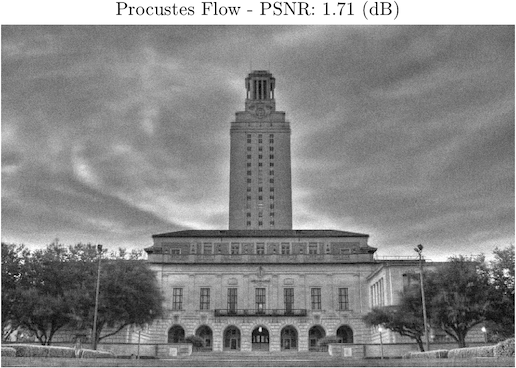} 
\includegraphics[width=0.32\textwidth]{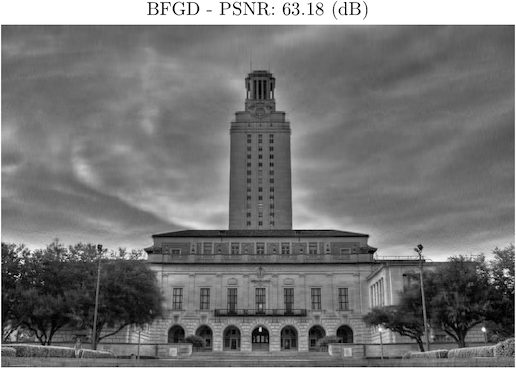}
\caption{Reconstruction performance in image denoising settings. The image size is $2845 \times 4266$ ($12,136,770$ pixels) and the approximation rank is preset to $r = 100$. We observe 35\% of the pixels of the true image. We depict the median reconstruction error with respect to the true image in dB over 10 Monte Carlo
realizations.}
\label{fig:app_exp2}
\end{figure*}

\begin{figure*}[t!]
\centering
\includegraphics[width=0.32\textwidth]{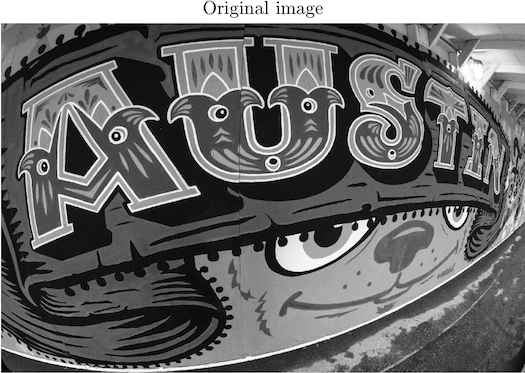} 
\includegraphics[width=0.32\textwidth]{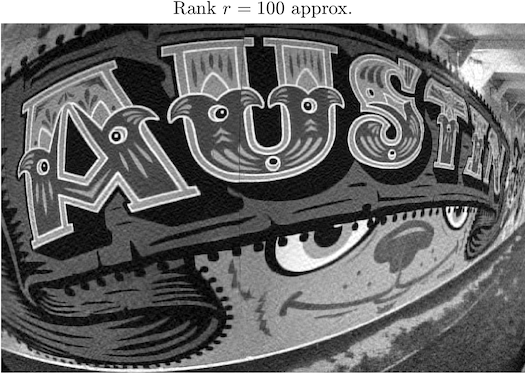} 
\includegraphics[width=0.32\textwidth]{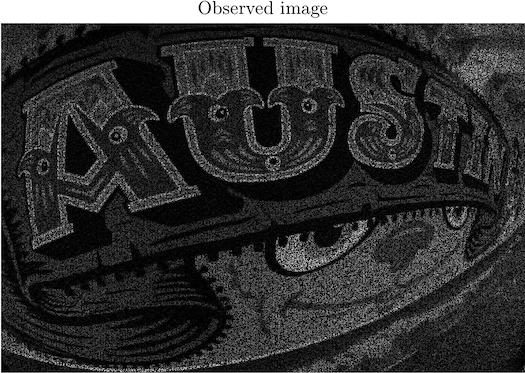} \\
\includegraphics[width=0.32\textwidth]{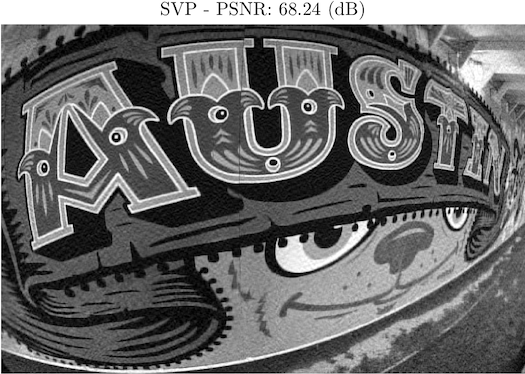} 
\includegraphics[width=0.32\textwidth]{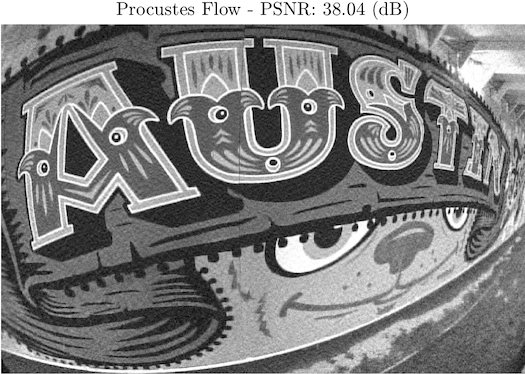} 
\includegraphics[width=0.32\textwidth]{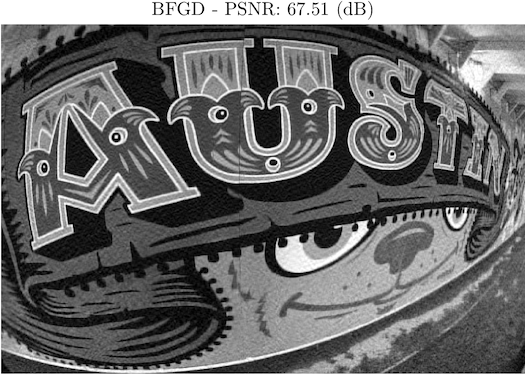}
\caption{Reconstruction performance in image denoising settings. The image size is $3309 \times 4963$ ($16,422,567$ pixels) and the approximation rank is preset to $r = 100$. We observe 30\% of the pixels of the true image. We depict the median reconstruction error with respect to the true image in dB over 10 Monte Carlo
realizations.}
\label{fig:app_exp3}
\end{figure*}

\begin{figure*}[t!]
\centering
\includegraphics[width=0.32\textwidth]{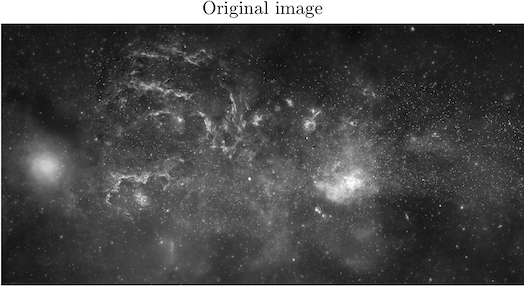} 
\includegraphics[width=0.32\textwidth]{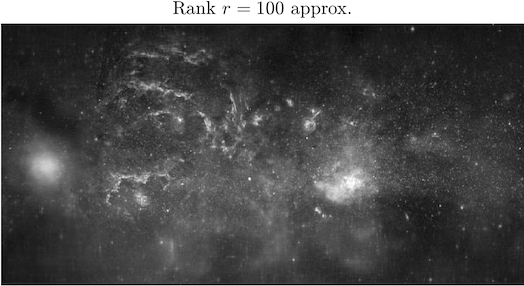} 
\includegraphics[width=0.32\textwidth]{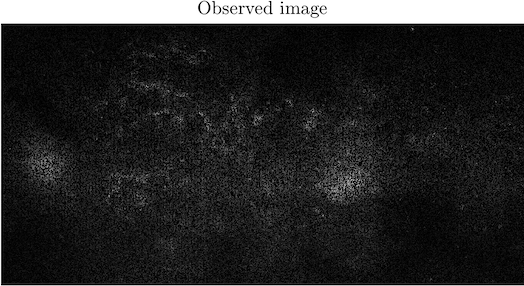} \\
\includegraphics[width=0.32\textwidth]{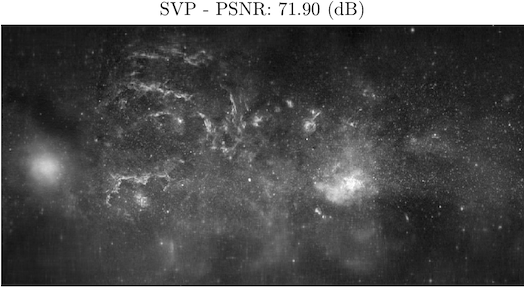} 
\includegraphics[width=0.32\textwidth]{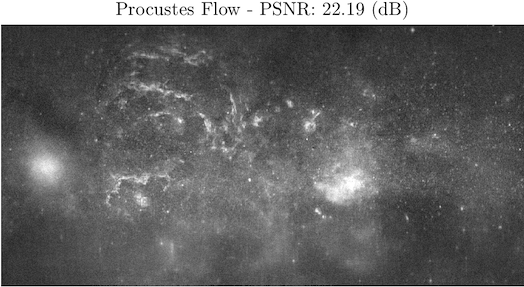} 
\includegraphics[width=0.32\textwidth]{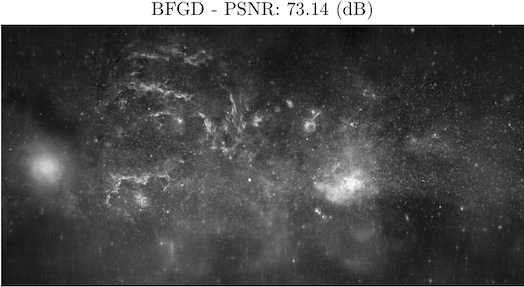}
\caption{Reconstruction performance in image denoising settings. The image size is $4862 \times 9725$ ($47,282,950$ pixels) and the approximation rank is preset to $r = 100$. We observe 30\% of the pixels of the true image. We depict the median reconstruction error with respect to the true image in dB over 10 Monte Carlo
realizations.}
\label{fig:app_exp4}
\end{figure*}

\begin{table*}[!h]
\centering
\ra{1.3}
\begin{scriptsize}
\rowcolors{2}{white}{black!05!white}
\begin{tabular}{c c c c c c c} \toprule
& \phantom{a} & \multicolumn{5}{c}{Time (sec.)} \\
\cmidrule {3-7}
Algorithm & & UT Campus &  & Graffiti & & Milky way \\	   
\cmidrule{1-1} \cmidrule {3-3} \cmidrule{5-5} \cmidrule{7-7} 
%\texttt{Matrix ALPS II} & & 2550.2 & & 2495.9 & & 7332.5 \\
\texttt{SVP} & & 5224.1 & & 4154.9 & & 7921.4 \\
\texttt{Procrustes Flow} & &  5383.4 & & 6501.4 & &  12806.3 \\ 
\texttt{BFGD} & &  4062.4 & &  3155.9 & & 9119.6  \\ 
\bottomrule
\end{tabular}
\end{scriptsize}
\caption{Summary of execution time results for the problem of image denoising. Timings correspond to median values on 10 Monte Carlo random instantiations. } \label{tbl:Comp6}
\end{table*}

\subsection{1-bit Matrix Completion}{\label{sec:1-bit}}

For this task, we repeat the experiments in \citep{davenport20141} and compare \algo with their proposed schemes.
The observational model we consider here satisfies the following principles: 
We assume $\Xopt \in \R^{m \times n}$ is an unknown low rank matrix, satisfying $\|\Xopt\|_\infty \leq \alpha, ~\alpha > 0$, from which we observe only a subset of indices $\Omega \subset [m] \times [n]$, according to the following rule:
\begin{align}
Y_{i, j} = \left\{
	\begin{array}{ll}
		+1  & \mbox{with probability } \sigma(\Xopt_{i, j}) \\
		-1 & \mbox{with probability } 1- \sigma(\Xopt_{i, j})
	\end{array}
\right. \quad \quad \text{for } (i, j) \in \Omega. \label{exp:bit}
\end{align} 
Similar to classic matrix completion results, we assume $\Omega$ is chosen uniformly at random, \emph{e.g.}, we assume $\Omega$ follows a binomial model, as in \citep{davenport20141}. 
Two natural choices for $\sigma$ function are: $(i)$ the logistic regression model, where $\sigma(x) = \tfrac{e^x}{1 + e^x}$, and $(ii)$ the probit regression model, where $\sigma(x) = 1 - \Phi(-x/\sigma)$ for $\Phi$ being the cumulative Gaussian distribution function. 
Both models correspond to different noise assumptions: in the first case, noise is modeled according to the standard logistic distribution, while in the second case, noise follows standard Gaussian assumptions.
Under this model, \citep{davenport20141} propose two convex relaxation algorithmic solutions to recover $\Xopt$: $(i)$ the convex maximum log-likelihood estimator under nuclear norm and infinity norm constraints:
\begin{equation}{\label{eq:exp_cvx_bit}}
\begin{aligned}
	& \underset{X \in \R^{m \times n}}{\text{minimize}}
	& & f(X),
	\\
	& \text{subject to}
	& & \|\X\|_* \leq \alpha \sqrt{r m n}, ~~ \|\X\|_\infty \leq \alpha,
\end{aligned} 
\end{equation} and $(ii)$ the the convex maximum log-likelihood estimator under only nuclear norm constraints. 
In both cases, $f(\X)$ satisfies the expression in \eqref{logPCA:eq_00}.
\citep{davenport20141} proposes a \emph{spectral projected-gradient descent} method for both these criteria; in the case where only nuclear norm constraints are present, SVD routines compute the convex projection onto norm balls, while in the case where both nuclear and infinity norm constraints are present, \citep{davenport20141} propose a alternating-direction method of multipliers (ADMM) solution, in order to compute the joint projection onto these sets.

\medskip
\noindent \textit{Synthetic experiments.}
We synthetically construct $\Xopt \in \R^{m \times n}$, where $m = n = 100$, such that $\Xopt = \U^\star \V^{\star \top}$, where $\U^\star \in \R^{m \times r}$, $\V^\star \in \R^{n \times r}$ for $r = 1$. 
The entries of $\U^\star$, $\V^\star$ are drawn i.i.d. from $\text{Uni}\left[-\tfrac{1}{2}, ~\tfrac{1}{2}\right]$.
Moreover, according to \citep{davenport20141}, we scale $\Xopt$ such that $\|\Xopt\|_\infty = 1$.
Then, we observe $Y \in \R^{m \times n}$ according to \eqref{exp:bit}, where $|\Omega| = \tfrac{1}{4}\cdot m n$. 
we consider the probit regression model with additive Gaussian noise, with variance $\sigma^2$. 

\begin{figure*}
\centering
\includegraphics[width=0.4\textwidth]{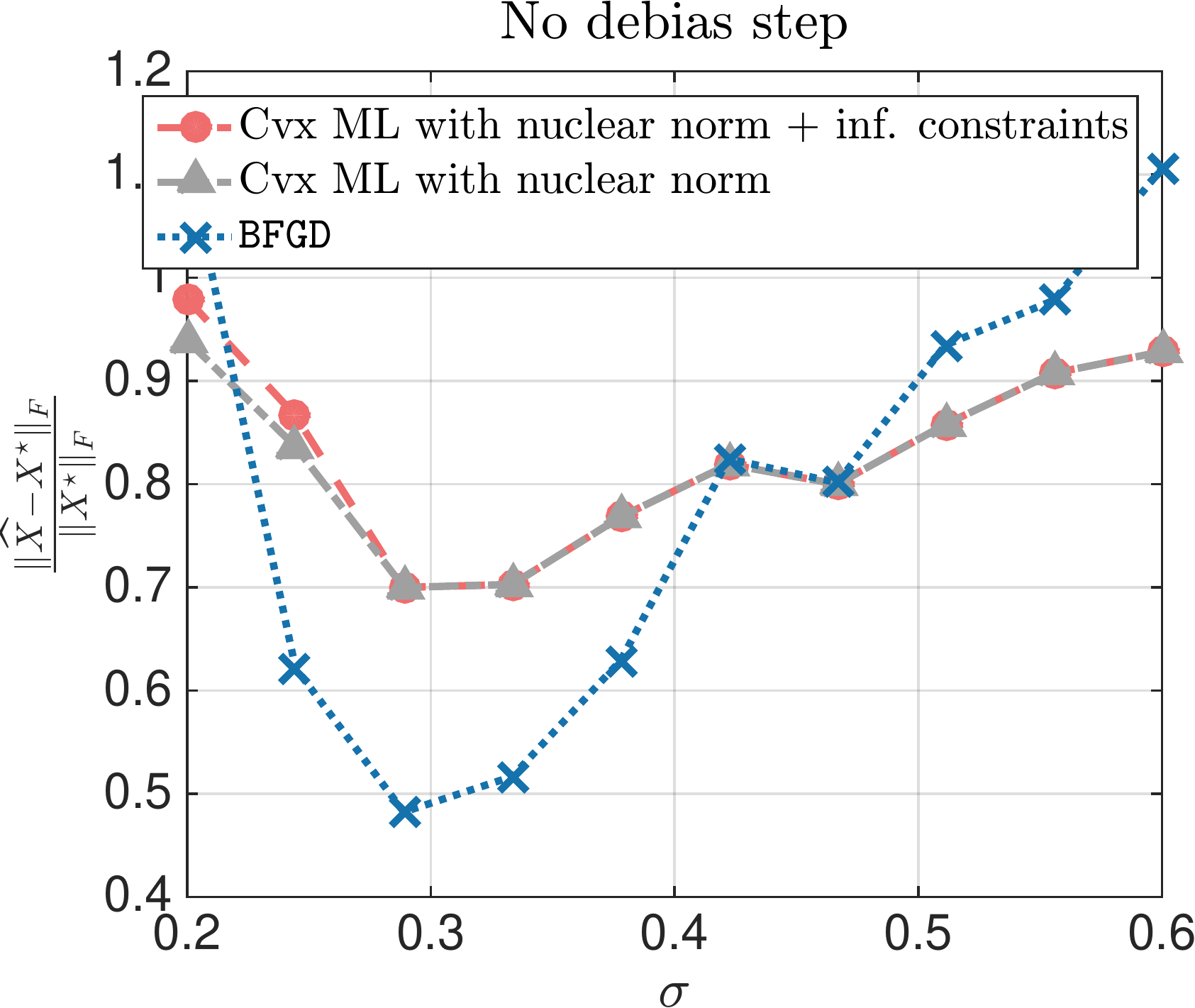}
\includegraphics[width=0.4\textwidth]{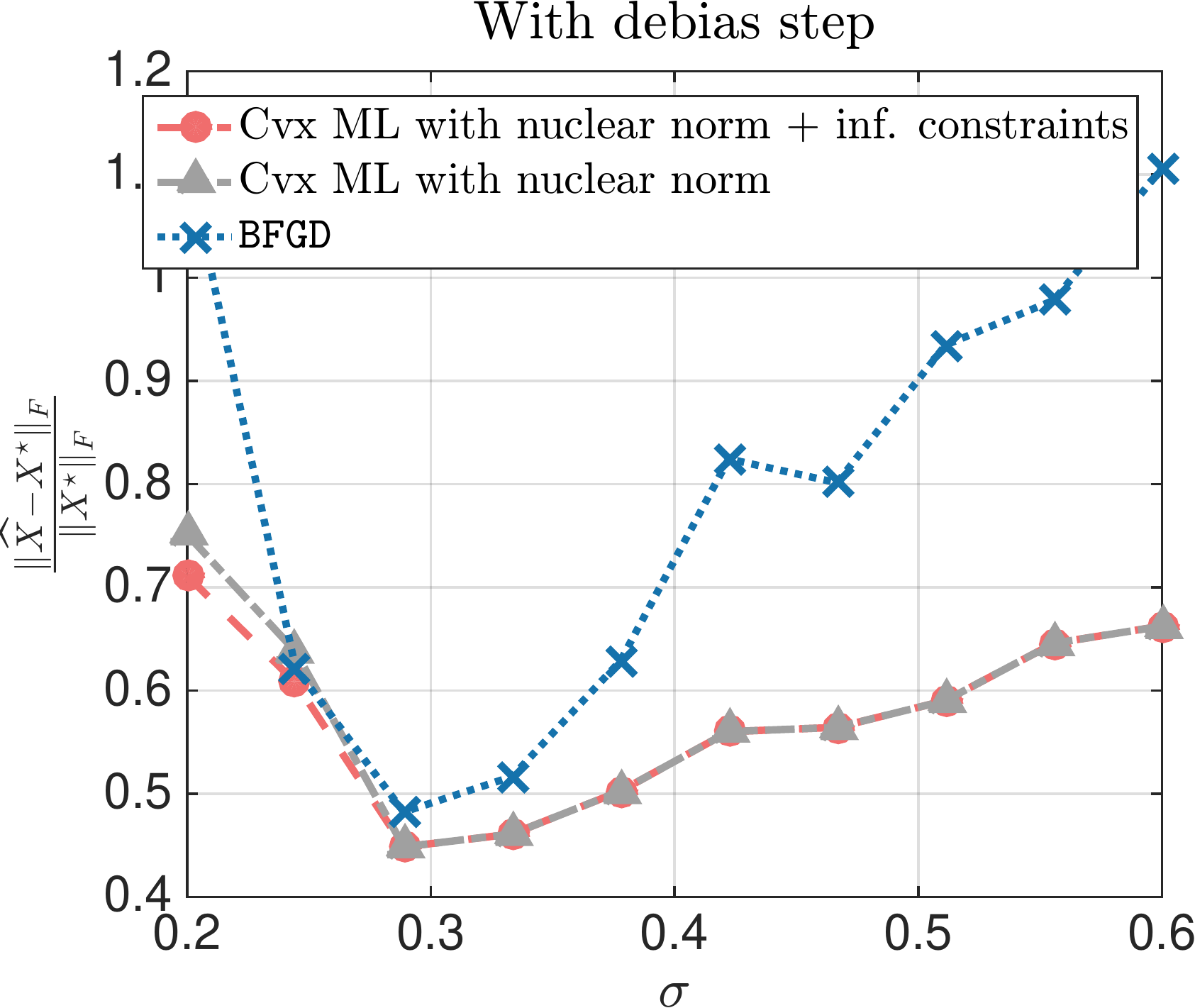}
\caption{Comparison of 1-bit matrix procedures. \textit{Left panel:} Output of \eqref{eq:exp_cvx_bit} is not projected onto rank-$r$ set.  \textit{Right panel:} Output of \eqref{eq:exp_cvx_bit} is projected onto rank-$r$ set.  }\label{fig:exp_bit_1}
\end{figure*}

Figure \ref{fig:exp_bit_1} depicts the recovery performance of \algo, as compared to variants of \eqref{eq:exp_cvx_bit}  in \citep{davenport20141}.
We consider their performance over different noise levels w.r.t. the normalized Frobenius norm distance $\tfrac{\|\widehat{X} - \Xopt\|_F}{\|\Xopt\|_F}$.
As noted in \citep{davenport20141}, the performance of all algorithms is poor when $\sigma$ is too small or too large, while in between, for moderate noise levels, we observe better performance for all approaches. 

By default, in all problem settings, we observe that the estimate of \eqref{eq:exp_cvx_bit} \emph{is not of low rank}: 
to compute the closest rank-$r$ approximation to that, we further perform a \emph{debias} step via truncated SVD. 
The effect of the debias step is better illustrated in Figure \ref{fig:exp_bit_1}, focusing on the differences between left and right plot: without such step, \algo has a better performance in terms of $\tfrac{\|\widehat{X} - \Xopt\|_F}{\|\Xopt\|_F}$, within the ``sweet" range of noise levels, compared to the convex analog in \eqref{eq:exp_cvx_bit}. 
Applying the debias step, both approaches have comparable performance, with that of \eqref{eq:exp_cvx_bit} being slightly better.

Perhaps somewhat surprisingly, the performance of \algo, in terms of \emph{estimating the correct sign pattern} of the entries, is better than that of \citep{davenport20141}, even with the debias step. 
Figure \ref{fig:exp_bit_2} (left panel) illustrates the observed performances for various noise levels.

\begin{figure}[!h]
\centering
		\includegraphics[width=0.52\textwidth]{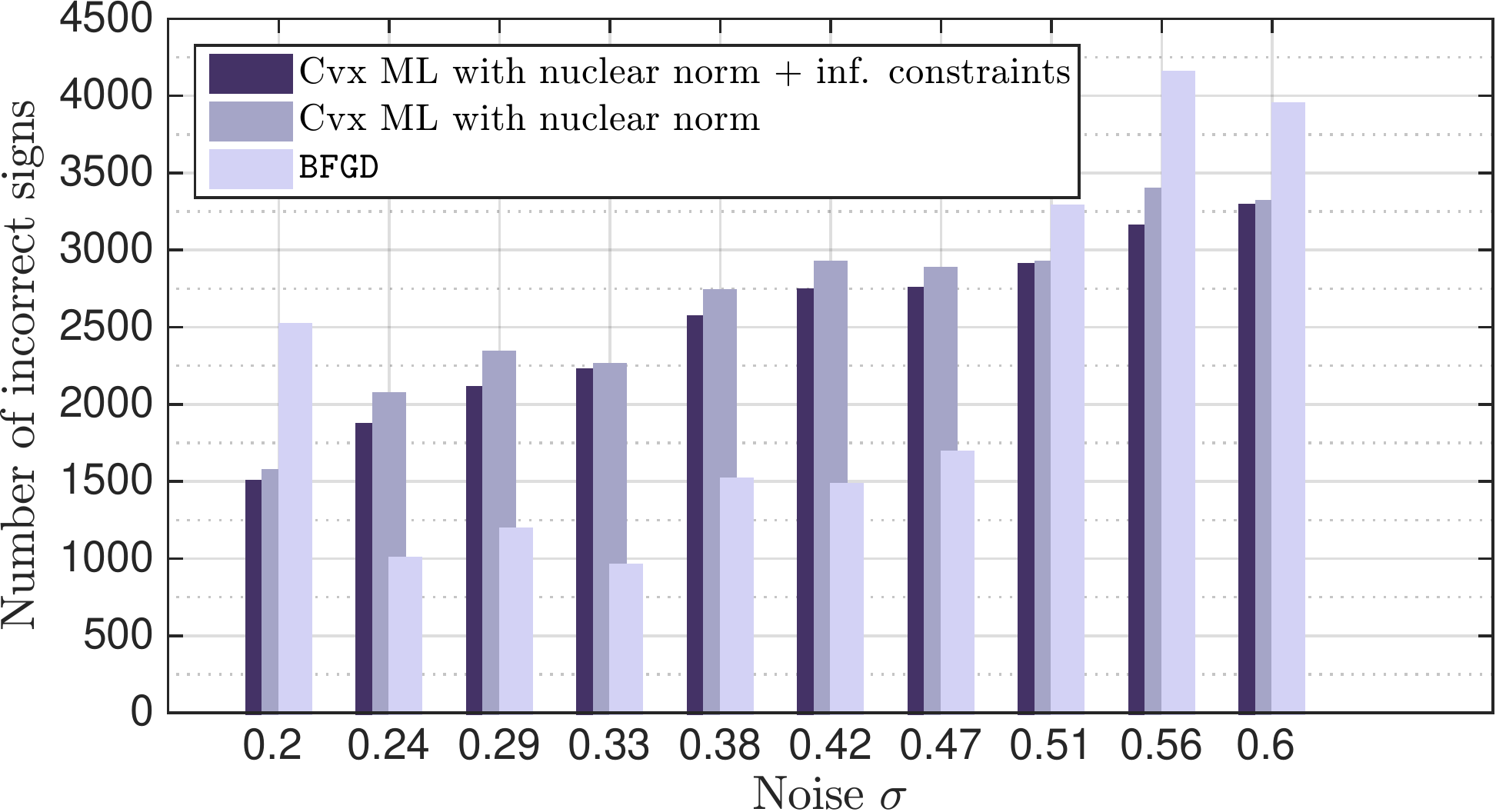}	
		\includegraphics[width=0.38\textwidth]{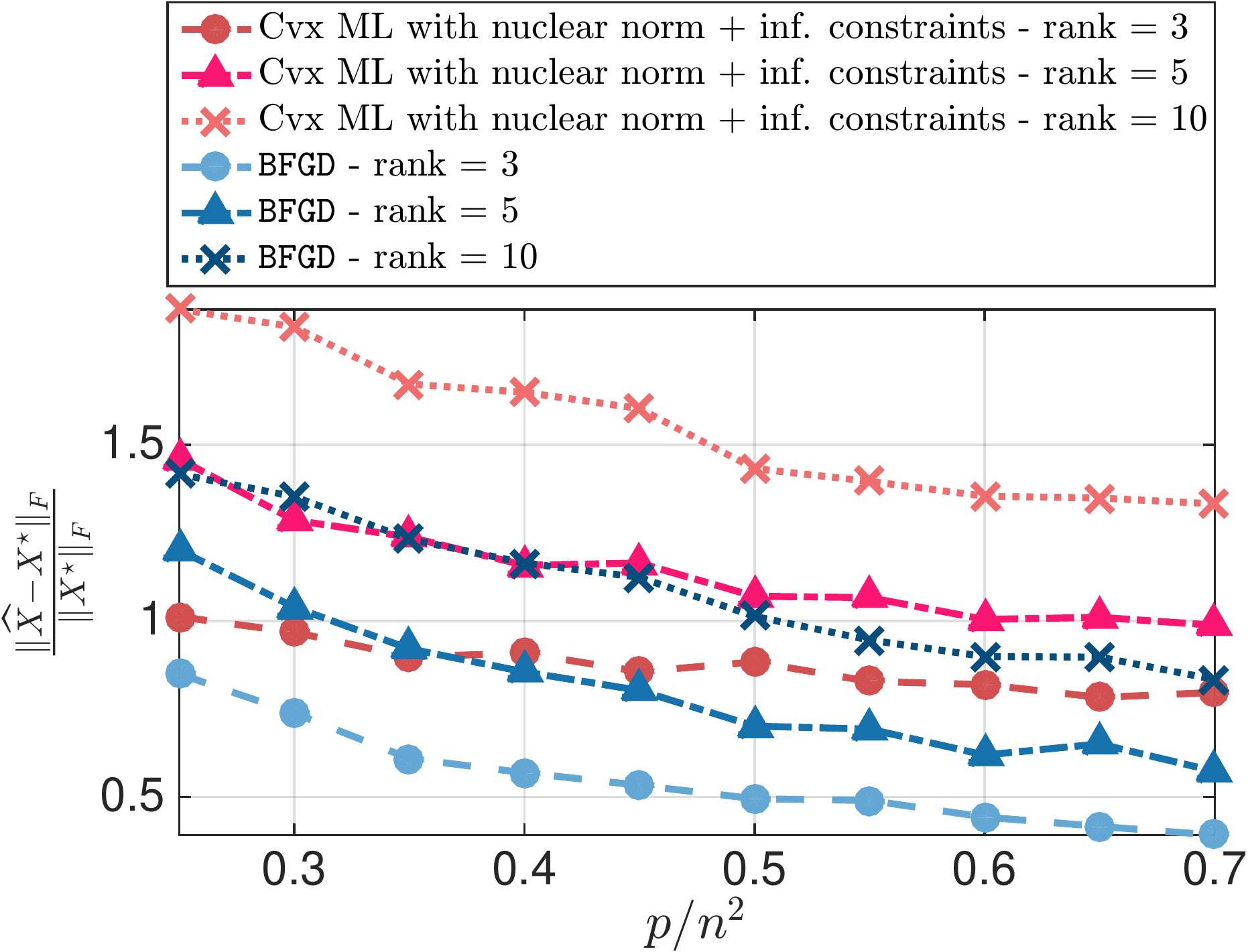}
\caption{\emph{Left panel:} Comparison of 1-bit matrix procedures w.r.t. sign pattern estimation. \emph{Right panel:} Recovery of $\Xopt$ from $p = C \cdot n^2$ measurements. $\Xopt$ is designed to be low rank: $r = 3, 5$ and $10$. x-axis represents $C$ for various values. }\label{fig:exp_bit_2}
\end{figure}

Finally, we study the performance of the algorithms under consideration as a function of the number of measurements, for fixed settings of dimensions $m = n = 200$ and noise level $\sigma = 0.244$. 
By the discussion above, such noise level leads to good performance from all schemes.
We considered matrices $\Xopt$ with rank $r \in \left\{3, ~5, ~10\right\}$ and generate $p = C \cdot n^2$, over a wide range of $0 < C < 1$. 
Figure \ref{fig:exp_bit_2} (right panel) shows the performance of \algo and the approach for \eqref{eq:exp_cvx_bit} in \citep{davenport20141}, in terms of the relative Frobenius norm of the error.
All approaches do poorly when there are only $p < 0.35 \cdot n^2$ measurements, since this is near the noiseless information-theoretic limit. 
For higher numbers of measurements, the non-convex approach in \algo returns more reasonable solutions and outperforms convex approaches, taking advantage of the prior knowledge on low-rankness of the solution.

\medskip
\noindent \textit{Recommendation system using the MovieLens data set.}
We compare 1-bit matrix completion solvers on the 100k MovieLens data set. 
To do so, we repeat the experiment in Section 4.3 of \citep{davenport20141}: we use the MovieLens 100k, which consists
of 100k movie ratings, from 1000 users on 1700 movies.
Each user entry denotes the movie rating, ranging from 1 to 5. 
To convert this data set into 1-bit measurements, we convert these ratings to binary observations by comparing each rating
to the average rating for the entire data set (which is approximately 3.5), according to \citep{davenport20141}. 
To evaluate the performance of the algorithms, we assume part of the observed ratings as unobserved (5k of them) and check if the estimate of $\Xopt$, $\widehat{\X}$, predicts the sign of these ratings. 
We perform ML estimation using logistic function $\sigma(x) = \tfrac{e^x}{1+ e^x}$ in $f$.

We compare the following algorithms: $(i)$ the spectral projected gradient descent (\texttt{SPG}) implementation of \eqref{eq:exp_cvx_bit} in \citep{davenport20141} for 1-bit matrix completion, $(ii)$ the standard matrix completion implementation \texttt{TFOCS} \citep{becker2011templates}, where we observe the \emph{unquantized} data set (actual values)\footnote{Using \texttt{TFOCS}, we set the regularizer $\mu = 10^{-3}$ as the parameter value that returned the best recovery results over a wide range of $\mu$ values.}, $(iii)$ \algo for various values of rank parameter $r$. 
The results are shown Table \ref{tbl:Comp7} over 10 Monte Carlo realizations (\emph{i.e.}, we randomly selected 5k ratings as test sets and solved the problem for different runs of the algorithms). 
The values in Table \ref{tbl:Comp7} denote the accuracy in predicting whether the unobserved ratings are above or below the average rating of 3.5.
\algo shows competitive performance, compared to convex approaches.
Moreover, setting the parameter $r$ is an ``easier" and more intuitive task:
our algorithm administers precise control on the rankness of the solution, which might lead to further interpretation of the results. 
Convex approaches lack of this property: the mapping between the regularization parameters and the number of rank-1 components in the extracted solution is highly non-linear.
At the same time, \algo shows much faster convergence to a good solution, which constitutes it a preferable algorithmic solution for large scale applications.

\begin{table*}[!h]
\centering
\ra{1.3}
\begin{footnotesize}
\rowcolors{2}{white}{black!05!white}
\begin{tabular}{c c c c c c c c c c c c c c c} \toprule
& & \multicolumn{9}{c}{Ratings (\%)} & & Overall (\%)& & Time (sec) \\ 
 \cmidrule{3-11} 
Algorithm & & 1 & & 2 & & 3 & & 4 & & 5 & & \phantom{a} & & \phantom{a} \\ 
 \midrule 
\texttt{SPG} ($\alpha\sqrt{r} = 0.32$) & &  73.7 & &  68.4 & &  52.5 & &  \textbf{74.9} & &  \textbf{91.0} & &  71.3 & &  79.5 \\ 
\texttt{SPG} ($\alpha\sqrt{r} = 4.64$) & &  77.2 & &  71.0 & &  58.5 & &  72.5 & &  86.9 & &  71.8 & &  213.4 \\ 
\texttt{SPG} ($\alpha\sqrt{r} = 10.00$) & &  76.2 & &  71.3 & &  58.3 & &  71.0 & &  85.7 & &  71.0 & &  491.8 \\ 
\texttt{TFOCS} & & 70.4 & &  69.4 & &  \textbf{59.2} & &  39.1 & &  59.4 & &  64.8 & &  42.3 \\ 
\texttt{BFGD} ($r = 3$) & &  \textbf{79.4} & &  74.5 & &  56.9 & &  72.5 & &  88.2 & &  \textbf{72.2} & &  \textbf{25.4} \\ 
\texttt{BFGD} ($r = 5$) & &  79.0 & &  72.4 & &  56.8 & &  71.6 & &  86.2 & &  71.2 & &  27.5 \\ 
\texttt{BFGD} ($r = 10$) & &  77.6 & &  \textbf{75.0} & &  57.5 & &  70.5 & &  84.1 & &  70.9 & &  30.3 \\ 
 \bottomrule 
 \end{tabular}
\end{footnotesize}
\caption{Summary of results for the problem of 1-bit matrix completio on MovieLens data set. Individual and overall ratings correspond to percentages of signs correctly estimated (+1 corresponds to original rating above 3.5, -1 corresponds to original rating below 3.5). Timings correspond to median values on 10 Monte Carlo random instantiations. } \label{tbl:Comp7}
\end{table*}

% Acknowledgements should go at the end, before appendices and references

%\acks{We would like to acknowledge support for this project
%from the National Science Foundation (NSF XXX)
%and ...}


\appendix
\section{Appendix}{\label{sec:proof_equiv_eta}}

\textit{Proof of \eqref{eq:weta1}$\Rightarrow$\eqref{eq:new_step_size}}
Let $R^\star_t$ be the $r \times r$ orthogonal matrix such that $\dist(\U_t,\V_t;\Xoptr) = \normf{\W_t - \Wo R^\star_t}$. By the triangle inequality, we have
\begin{align} \label{eqn:norm_bound1}
\| \W_t \|_2 &= \| \W_t - \Wo R^\star_t + \Wo R_t^\star\|_2  \stackrel{(i)}{\leq} \| \Wo R^\star_t \|_2 + \| \W_t - \Wo R^\star_t \|_2 \nonumber \\
				&\stackrel{(ii)}{\leq} \| \Wo \|_2 + \tfrac{\sqrt{2} \sigma_r\left(\Xo_r\right)^{\sfrac{1}{2}}}{10} \nonumber \\
				&\stackrel{(iii)}{\leq} \| \Wo \|_2 + \tfrac{\sigma_r\left(\Wo\right)}{10} \nonumber \\
				&\leq \tfrac{11}{10} \cdot \|\Wo\|_2
\end{align}
where $(i)$ is due to triangle inequality, 
$(ii)$ is due to Assumption A.1
$(iii)$ is due to the fact that $\sqrt{2} \cdot \sigma_r(\Xo_r)^{\sfrac{1}{2}} = \sigma_r(\Wo)$ and $\kappa \geq 1$. 
The above bound holds for every $t = 0,1,\ldots$.

On the other hand, we have:
\begin{align} \label{eqn:norm_bound2}
\| \W_0 \|_2 &= \| \W_0 - \Wo R^\star_t + \Wo R^\star\|_2  \stackrel{(i)}{\geq} \| \Wo R^\star_t \|_2 - \| \W_0 - \Wo R^\star_t \|_2 \nonumber \\
				 &\geq \| \Wo \|_2 - \tfrac{\sqrt{2} \sigma_r \left( \Xo_r \right)^{\sfrac{1}{2}}}{10} \nonumber \\
				 &\geq \| \Wo \|_2 - \tfrac{\sigma_1 \left( \Wo \right)}{10} \nonumber \\
				 &\geq \tfrac{9}{10} \cdot \|\Wo\|_2
\end{align}
Combining \eqref{eqn:norm_bound1} and \eqref{eqn:norm_bound2}, we obtain:
\begin{align*}
\|\W_t \|_2 \leq \tfrac{11}{10} \cdot \|\Wo\|_2 \leq \tfrac{11}{9} \|\W_0\|_2   \Longrightarrow   \tfrac{81}{121} \cdot \|\W_t\|_2^2 \leq \|\W_0\|_2^2
\end{align*} and, finally, 
\begin{align*}
\tfrac{1}{8 \cdot \max\{L, ~L_g\} \cdot \|\W_t\|_2^2} = \tfrac{1}{8 \cdot \tfrac{121}{81} \cdot \max\{L, ~L_g\} \cdot \tfrac{81}{121} \cdot \|\W_t\|_2^2} \geq \tfrac{1}{12 \cdot \max\{L, ~L_g\} \cdot \|\W_0\|_2^2}
\end{align*} \hfill $\blacksquare$

\textit{Proof of \eqref{eq:weta}$\Rightarrow$\eqref{eq:new_step_size_smooth}}
We have
\begin{align}
\| \nabla f(\U_t\V_t^\top) \|_2 \nonumber
&\le \| \nabla f(\factorinit) \|_2 + \| \nabla f(\U_t\V_t^\top) - \nabla f(\factorinit) \|_2 \nonumber \\
&\stackrel{(i)}{\le} \| \nabla f(\factorinit) \|_2 + L \| \U_t\V_t^\top - \factorinit \|_F \nonumber \\
&\stackrel{(ii)}{\le} \| \nabla f(\factorinit) \|_2 + L \| \U_t\V_t^\top - \factoropt \|_F + L \| \factorinit - \factoropt \|_F \label{eq:equiv_eta_2}
\end{align} where $(i)$ is due to the fact that $f$ is $L$-smooth and, $(ii)$ holds by adding and subtracting $\U^\star \V^{\star\top}$ and then applying triangle inequality. To bound the last two terms on the right hand side, we observe:
\begin{align*}
\| \U_t\V_t^\top - \factoropt \|_F &= \|\U_t\V_t^\top - \U^\star R \V_t^\top + \U^\star R \V_t^\top - \Uo R R^\top {\Vo}^\top \|_F \\
&\stackrel{(i)}{\le} \normt{U^\star R} \cdot \normf{V_t - \Vs R} + \normt{V_t} \cdot \normf{U_t - \Uo R} \\
&\le \left(\normt{U^\star} + \normt{V_t}\right) \cdot \dist(U_t, V_t; \Xoptr) \\
&\stackrel{(ii)}{\le} \frac{21}{10} \cdot \normt{W^\star} \cdot \frac{\sigma_r(W^\star)}{10} \\
&\le \frac{7}{10} \cdot \normt{W_0}^2
\end{align*} where $(i)$ is due to the triangle and Cauchy-Schwartz inequalities, $(ii)$ is by Assumption A1 and \eqref{eqn:norm_bound1}. Similarly, one can show that $\| \factorinit - \factoropt \|_F \leq \frac{7}{10} \cdot \normt{W_0}^2$. Thus, \eqref{eq:equiv_eta_2} becomes:
\begin{align}
\| \nabla f(\factorinit) \|_2 \ge \| \nabla f(\U_t\V_t^\top) \|_2 - \frac{3L}{2} \normt{W_0}^2 \label{eq:equiv_eta_3}
\end{align}
Applying \eqref{eqn:norm_bound1}, \eqref{eqn:norm_bound2}, and the above bound, we obtain the desired result.
 \hfill $\blacksquare$

\section{Appendix} \label{sec:proof_linear}
For clarity, we omit the subscript $t$, and use $(\U,~\V)$ to denote the current estimate and $(\U^+, \V^+)$ the next estimate. 
Further, we abuse the notation by denoting $\nabla g \triangleq \nabla g (\U^\top \U - \V^\top \V)$, where the gradient is taken over both $U$ and $V$. 
We denote the stacked matrices of $(\U,\V)$ and their variants as follows:
$$
\W = \Wfac,~ \W^+ = \begin{bmatrix} \U^+ \\ \V^+ \end{bmatrix} ,~ \Wo = \Wofac.
$$
Observe that $\W, \W^+, \Wo \in \R^{(m + n) \times r}$. 
Then, the main recursion of \algo in Algorithm \ref{alg:BFGD} can be succinctly written as
\begin{align*}
\Wp = \W - \weta \nabla_W (f + \tfrac{1}{2} g),
\end{align*}
where
\begin{align*}
\nabla_W (f + \tfrac{1}{2} g)
=
\begin{bmatrix}
\gradU + \tfrac{1}{2} \nabla_U g \\
\gradV + \tfrac{1}{2} \nabla_V g \\
\end{bmatrix}
=
\begin{bmatrix}
\gradf(\factor) V + \tfrac{1}{2} U \nabla g \\
\gradf(\factor)^\top U - \tfrac{1}{2} V \nabla g \\
\end{bmatrix}.
\end{align*}
In the above formulations, we use as regularizer of $g$ function $\lambda = \tfrac{1}{2}$.

Our discussion below is based on the Assumption A.1, where:
\begin{align} \label{eqn:condition}
\dist(\U,\V;\Xoptr) \le \frac{\sqrt{2} \cdot \sigma_r(\Xoptr)^{1/2}}{10 \sqrt{\kappa}} = \frac{\sigma_r(\Wo)}{10 \sqrt{\kappa}},
\end{align}
holds for the current iterate. 
The last equality is due to the fact that $\sigma_r(\Wo) = \sqrt{2} \cdot \sigma_r(\Xoptr)^{1/2}$, for $(\Uo, \Vo)$ with ``equal footing".
For the initial point $(\Uinit,\Vinit)$, \eqref{eqn:condition} holds by the assumption of the theorem. 
Since the right hand side is fixed, \eqref{eqn:condition} holds for every iterate, as long as $\dist(\U,\V;\Xoptr)$ decreases. 

To show this, let $R \in O_r$ be the minimizing orthogonal matrix such that $\dist(\U,\V;\Xoptr) = \normf{W - \Wo R}$; here, $O_r$ denotes the set of $r \times r$ orthogonal matrices such that $R^\top R = I$.
Then, the decrease in distance can be lower bounded by
\begin{align} \label{eqn:decrease}
\dist(\U,\V;\Xoptr)^2 &- \dist(\Up,\Vp;\Xoptr)^2 \nonumber \\ &= \normf{\W - \Wo R}^2 - \min_{Q \in O_r} \normf{\W^+ - \Wo Q}^2 \nonumber \\
																&\ge \normf{\W - \Wo R}^2 - \normf{\W^+ - \Wo R}^2 \nonumber \\
																&= 2 \weta \cdot \innerp{\nabla_W (f + \tfrac{1}{2}g)}{\W - \Wo R}  - \weta^2 \cdot \normf{\nabla_W (f + \tfrac{1}{2}g)}^2
% &= 2 \eta \cdot \tagterm{(A)}{\left\{ \innerp{\gradf(\factor) V}{\U - \Uh R} 
%                                    + \innerp{\gradf(\factor)^\top U}{\V - \Vh R} \right\}} \nonumber \\
% &\quad + \eta \cdot \tagterm{(B)}{\left\{ \innerp{\U \nabla g}{\U - \Uh R}
%                        - \innerp{\V \nabla g}{\V - \Vh R} \right\}} \nonumber \\                       
% &\quad - \eta^2 \cdot \tagterm{(B)}{\left\{\normf{\gradf(\factor)\V + \frac{1}{2} \U \nabla g}^2
%                                        + \normf{\gradf(\factor)^\top\U + \frac{1}{2} \V \nabla g}^2 \right\}}
\end{align}
where the last equality is obtaining by substituting $\W^+$, according to its definition above.
To bound the first term on the right hand side, we use the following lemma; the proof is provided in Section \ref{sec:proof_lindescent}.
\begin{lemma}[Descent lemma] \label{lem:linconv_descent}
Suppose \eqref{eqn:condition} holds for $\W$. Let $\mumin = \min \left\{\mu,~\mu_g\right\}$ and $L_{\max} = \max \left\{L, ~L_g\right\}$ for $(\mu, L)$ and $(\mu_g, L_g)$ the strong convexity and smoothness parameters pairs for $f$ and $g$, respectively.
Then, the following inequality holds:
\begin{align} \label{eqn:bound_AB_final}
\innerp{\nabla_{\W}(f + \tfrac{1}{2} g)}{\W - \Wo R} &\ge \tfrac{\mumin \cdot \sigma_r(\Wo)^2}{20} \normf{\W - \Wo R}^2 + \tfrac{1}{4L_{\max}} \normf{\gradf(\factor)}^2 \nonumber \\ &+ \tfrac{1}{16L_{\max}} \normf{\nabla g}^2 - \tfrac{L}{2} \normf{\Xopt - \Xoptr}^2
\end{align}
\end{lemma}
For the second term on the right hand side of \eqref{eqn:decrease}, we obtain the following upper bound:
\begin{align} \label{eqn:bound_C}
\normf{\nabla_W (f + \tfrac{1}{2}g)}^2 &= \left\| \begin{bmatrix} 
																		\gradf(\factor) V + \tfrac{1}{2} U \nabla g \\
																		\gradf(\factor)^\top U - \tfrac{1}{2} V \nabla g \\
																   \end{bmatrix} \right\|_F^2 \nonumber \\
													  &= \normf{\gradf(\factor) V + \tfrac{1}{2} \U \nabla g}^2  + \normf{\gradf(\factor)^\top U - \tfrac{1}{2} \V \nabla g}^2 \nonumber \\
													  &\stackrel{(a)}{\le} 2\normf{\gradf(\factor) V}^2 + \tfrac{1}{2} \normf{\U \nabla g}^2 + 2\normf{\gradf(\factor)^\top U}^2 + \tfrac{1}{2} \normf{\V \nabla g}^2 \nonumber \\
													  &= 2\normf{\gradf(\factor) V}^2 + 2\normf{\gradf(\factor)^\top U}^2 + \tfrac{1}{2} \normf{W \nabla g}^2 \nonumber \\
													  &\stackrel{(b)}{\le} 2\normf{\gradf(\factor)}^2 \cdot \left(\normt{U}^2 + \normt{V}^2\right) + \tfrac{1}{2} \normt{W}^2 \normf{\nabla g}^2 \nonumber \\
													  &\stackrel{(c)}{\le} \left(4\normf{\gradf(\factor)}^2 + \tfrac{1}{2} \normf{\nabla g}^2\right) \cdot \normt{W}^2,
\end{align}
where $(a)$ follows from the fact $\normf{A+B}^2 \le 2\normf{A}^2 + 2\normf{B}^2$,
$(b)$ is due to the fact $\normf{AB} \le \normf{A} \cdot \normt{B}$, 
and $(c)$ follows from the observation that $\normt{U},\normt{V} \le \normt{W}$. 

Plugging \eqref{eqn:bound_AB_final} and \eqref{eqn:bound_C} in \eqref{eqn:decrease}, we get
\begin{align*}
\dist(\U,\V;\Xoptr)^2 &- \dist(\Up,\Vp;\Xoptr)^2 \nonumber \\ &\ge 2 \weta \cdot \innerp{\nabla_W (f + \tfrac{1}{2}g)}{\W - \Wo R}
 - \weta^2 \cdot \normf{\nabla_W (f + \tfrac{1}{2}g)}^2 \\
&\ge \tfrac{\weta \cdot \mumin \cdot \sigma_r(\Wo)^2}{10} \dist(\U,\V;\Xoptr)^2 - \weta L \normf{\Xo - \Xoptr}^2 \\
&= \tfrac{\weta \cdot \mumin \cdot \sigma_r(\Xoptr)}{5} \dist(\U,\V;\Xoptr)^2 - \weta L \normf{\Xo - \Xoptr}^2
\end{align*} where we use the fact that $\sigma_r(\Wo) = \sqrt{2} \cdot \sigma_r(\Xoptr)^{1/2}$. 

The above lead to the following recursion:
\begin{align*}
\dist(\Up, \Vp;\Xoptr)^2 \leq \gamma_t \cdot \dist(\U,\V;\Xoptr)^2 + \weta L \normf{\Xo - \Xoptr}^2,
\end{align*}
where $\gamma_t = 1 - \tfrac{\weta \cdot \mumin \cdot \sigma_r(\Xoptr)}{5}$. 
By the definition of $\weta$ in \eqref{eq:weta1}, we further have:
\begin{align*}
\gamma_t &= 1 - \tfrac{\mumin \cdot \sigma_r(\Xoptr)}{40 \cdot L_{\max} \cdot \|\W\|_2^2} \\
			    &\stackrel{(i)}{\geq} 1 - \tfrac{\mumin \cdot \sigma_r(\Xoptr)}{40 \cdot L_{\max} \cdot \tfrac{100}{81} \cdot \|\Wo\|_2^2} \\
			    &\stackrel{(ii)}{\geq} 1 - \tfrac{\mumin}{17 \cdot L_{\max}} \cdot \tfrac{\sigma_r(\Xoptr)}{\sigma_1(\Xoptr)}
\end{align*}
where $(i)$ is by using \eqref{eqn:norm_bound2} that connects $\|\W\|_2$ with $\|\Wo\|_2$ as $\|\W\|_2 \geq \tfrac{9}{10} \|\Wo\|_2$, and 
$(ii)$ is due to the fact $\|\Wo\|_2 = \sqrt{2} \cdot \sigma_1(\Xoptr)^{1/2}$.  \hfill $\blacksquare$

\subsection*{B.1. Proof of Lemma \ref{lem:linconv_descent}}{\label{sec:proof_lindescent}}
Before we step into the proof, we require a few more notations for simpler presentation of our ideas. 
We use another set of stacked matrices
$
Y = \begin{bmatrix} \U \\ -\V \end{bmatrix} ,~
Y^\star = \begin{bmatrix} \Uo \\ -\Vo \end{bmatrix}.
$
The error of the current estimate from the closest optimal point is denoted by the following $\Delta_{\times}$ matrix structures:
$$
\Delta_U = U - \Uo R ,\quad
\Delta_V = V - \Vo R ,\quad
\Delta_W = W - \Wo R ,\quad
\Delta_Y = Y - \Yo R.
$$

For our proof, we can write
\begin{small}
\begin{align}
\innerp{\nabla_W (f + \tfrac{1}{2}g)}{\W - \Wo R} &= \tagterm{(A)}{\innerp{\gradf(\factor) V}{\U - \Uo R} + \innerp{\gradf(\factor)^\top U}{\V - \Vo R}} \nonumber \\ &\quad \quad \quad \quad \quad \quad \quad \quad \quad \quad \quad + \tfrac{1}{2} \cdot \left(\tagterm{(B)}{\innerp{\U \nabla g}{\U - \Uo R} - \innerp{\V \nabla g}{\V - \Vo R}}\right) \nonumber                  
\end{align}
\end{small}

For (A), we have
\begin{align}
(A)
&= \innerp{\gradf(\factor) V}{\U - \Uo R} + \innerp{\gradf(\factor)^\top U}{\V - \Vo R} \nonumber \\
&= \innerp{\gradf(\factor)}{\factor - \factoropt} + \innerp{\gradf(\factor)}{\Delta_U \Delta_V^\top} \label{eq:A_term}\\
&\ge \frac{\mu}{2} \tagterm{(A1)}{\normf{\factor - \factoropt}^2}
 + \frac{1}{2L} \tagterm{(A2)}{\normf{\gradf(\factor)}^2}
 - \frac{L}{2} \tagterm{(A3)}{\normf{\Xo - \Xoptr}^2}
 - \tagterm{(A4)}{\normt{\gradf(\factor)} \cdot \normf{\Delta_W}^2} \nonumber
\end{align}
where, for the second term in \eqref{eq:A_term}, we use the fact that $$\innerp{\gradf(\factor)}{\Delta_U \Delta_V^\top} \geq - \left|\innerp{\gradf(\factor)}{\Delta_U \Delta_V^\top}\right| = - \left|\innerp{\gradf(\factor) \Delta_V}{\Delta_U}\right|,$$ the Cauchy-Schwarz inequality and the fact that $\|\Delta_U\|_F, \|\Delta_V\|_F \leq \|\Delta_\W\|_F$; the first term in \eqref{eq:A_term} follows from:
\begin{align*}
&\innerp{\gradf(\factor)}{\factor - \factoropt}
\stackrel{(i)}{\ge} f(\factor) - f(\factoropt) + \frac{\mu}{2} \normf{\factor - \factoropt}^2 \\
&\stackrel{(ii)}{=} (f(\factor) - f(\Xopt)) - (f(\factoropt) - f(\Xopt)) + \frac{\mu}{2} \normf{\factor - \factoropt}^2 \\
&\stackrel{(iii)}{\ge} \frac{1}{2L} \normf{\gradf(\factor)}^2 - \frac{L}{2} \normf{\Xopt - \factoropt}^2 + \frac{\mu}{2} \normf{\factor - \factoropt}^2.
\end{align*}
where $(i)$ is due to the $\mu$-strong convexity of $f$, 
$(ii)$ is by adding and subtracting $f(\Xopt)$; observe that $f(\Xopt) = f(\factoropt)$ if and only if $\text{rank}(\Xopt) = r$, 
and $(iii)$ is due to the $L$-smoothness of $f$ and the fact that $\nabla f(\Xopt) = 0$ (for the middle term), and due to the inequality \cite[eq. (2.1.7)]{Nesterov} (for the first term):
\begin{align}
f(\X) + \langle \nabla f(\X), \Y - \X \rangle + \tfrac{1}{2L} \cdot \|\nabla f(\X) - \nabla f(\Y)\|_F^2 \leq f(\Y). \label{eq:Lipschitz}
\end{align}

For (B), we have
\begin{align}
(B)
	&= \innerp{Y \nabla g}{W - \Wo R} = \innerp{\nabla g}{Y^\top W - Y^\top \Wo R} \nonumber  \\
	&= \tfrac{1}{2} \innerp{\nabla g}{Y^\top W - R^\top {\Yo}^\top \Wo R}  + \tfrac{1}{2} \innerp{\nabla g}{Y^\top W - 2 Y^\top \Wo R + R^\top {\Yo}^\top \Wo R} \nonumber \\
	&\stackrel{(a)}{=} \tfrac{1}{2} \innerp{\nabla g}{Y^\top W} + \tfrac{1}{2} \innerp{\nabla g}{Y^\top W - Y^\top \Wo R - R^\top {\Yo}^\top W + R^\top {\Yo}^\top \Wo R} \nonumber \\
	&= \frac{1}{2} \innerp{\nabla g}{U^\top U - V^\top V} + \frac{1}{2} \innerp{\nabla g}{\Delta_Y^\top \Delta_W} \label{eq:B_term} \\
	&\stackrel{(b)}{\ge} \frac{\mu_g}{4} \tagterm{(B1)}{\normf{U^\top U - V^\top V}^2} + \frac{1}{4L_g} \tagterm{(B2)}{\normf{\nabla g}^2} - \frac{1}{2} \tagterm{(B3)}{\normt{\nabla g} \cdot \normf{\Delta_W} \cdot \normf{\Delta_Y}} \nonumber
\end{align}
where $(a)$ follows from the ``balance" assumption in $\mathcal{X}_r^\star$:
\begin{align*}
{\Yo}^\top \Wo &= {\Uo}^\top \Uo - {\Vo}^\top \Vo = 0,
\end{align*}
for the first term, and the fact that $\nabla g$ is symmetric, and therefore
\begin{align*}
\innerp{\nabla g}{Y^\top \Wo R} &= \innerp{\nabla g}{R^\top {\Wo}^\top Y} = \innerp{\nabla g}{R^\top {\Yo}^\top W},
\end{align*}
for the second term;
$(b)$ follows from the fact $$\innerp{\nabla g}{\Delta_Y^\top \Delta_W}  \geq - \left|\innerp{\nabla g}{\Delta_Y^\top \Delta_W} \right| = - \left|\innerp{\Delta_Y \nabla g}{\Delta_W} \right|,$$ and the Cauchy-Schwarz inequality on the second term in \eqref{eq:B_term}, and
\begin{align*}
\innerp{\nabla g}{U^\top U - V^\top V} &\stackrel{(i)}{\ge} g(U^\top U - V^\top V) - g(0) + \tfrac{\mu_g}{2} \normf{U^\top U - V^\top V}^2 \\
&\stackrel{(ii)}{\ge} \innerp{\nabla g(0)}{U^\top U - V^\top V} + \tfrac{1}{2L_g} \normf{\nabla g - \nabla g(0)}^2 + \tfrac{\mu_g}{2} \normf{U^\top U - V^\top V}^2 \\
&\stackrel{(iii)}{=} \tfrac{1}{2L_g} \normf{\nabla g}^2 + \tfrac{\mu_g}{2} \normf{U^\top U - V^\top V}^2
\end{align*}
where $(i)$ follow from the strong convexity, $(ii)$ is due to \eqref{eq:Lipschitz}, and $(iii)$ is by construction of $g$ where $\nabla g(0) = 0$. 
Furthermore, (B1) can be bounded below as follows:
\begin{small}
\begin{align*}
(B1)
&= \normf{U^\top U - V^\top V}^2 = \normf{U^\top U}^2 + \normf{V^\top V}^2 - 2 \innerp{U^\top U}{V^\top V} \\
&= \normf{UU^\top}^2 + \normf{VV^\top}^2 - 2 \innerp{UV^\top}{UV^\top} \\
&= \innerp{WW^\top}{YY^\top} \\
&= \innerp{WW^\top - \Wo{\Wo}^\top}{YY^\top - \Yo{\Yo}^\top}
 + \innerp{\Wo{\Wo}^\top}{YY^\top}
 + \innerp{WW^\top - \Wo{\Wo}^\top}{\Yo{\Yo}^\top} \\
&\stackrel{(i)}{=} \innerp{WW^\top - \Wo{\Wo}^\top}{YY^\top - \Yo{\Yo}^\top}
 + \innerp{\Wo{\Wo}^\top}{YY^\top}
 + \innerp{WW^\top}{\Yo{\Yo}^\top} \\
&\ge \innerp{WW^\top - \Wo {\Wo}^\top}{YY^\top - \Yo {\Yo}^\top} \\
&= \normf{UU^\top - \Uo{\Uo}^\top}^2 + \normf{VV^\top - \Vo{\Vo}^\top}^2 - 2 \normf{UV^\top - \Uo{\Vo}^\top}^2
\end{align*}
\end{small}
where $(i)$ is due to the fact that
$$
\innerp{\Wo{\Wo}^\top}{\Yo{\Yo}^\top} = \normf{{\Yo}^\top \Wo}^2 = \normf{{\Uo}^\top \Uo - {\Vo}^\top \Vo}^2 = 0
$$
and the first inequality holds by the fact that the inner product of two PSD matrices is non-negative. 

At this point, we have all the required components to compute the desired lower bound.
Combining (A1) and (B1), we get
\begin{align*}
4 (A1) + (B1)
&= \normf{UU^\top - \Uo{\Uo}^\top}^2 + \normf{VV^\top - \Vo{\Vo}^\top}^2 + 2 \normf{UV^\top - {\Uo}{\Vo}^\top}^2 \\
&= \normf{WW^\top - \Wo{\Wo}^\top}^2 \geq \frac{4\sigma_r(\Wo)^2}{5} \normf{\Delta_W}^2,
\end{align*} where, in order to obtain the last inequality, we borrow the following Lemma by \citep{tu2016low}:
\begin{lemma}
For any $\W, \Wo \in R^{(m+n)\times r}$, with $\Delta_W = \W - \Wo R$ for some orthogonal matrix $R \in \R^{r \times r}$, we have:
\begin{align*}
\normf{WW^\top - \Wo{\Wo}^\top}^2 \geq 2 \cdot \left(\sqrt{2} - 1\right) \cdot \sigma_r(\Wo)^2 \cdot \normf{\Delta_W}^2
\end{align*} 
\end{lemma}
For convenience, we further lower bound the right hand side of this lemma by: $2 \cdot \left(\sqrt{2} - 1\right) \cdot \sigma_r(\Wo)^2 \cdot \normf{\Delta_W}^2 \geq  \frac{4\sigma_r(\Wo)^2}{5} \normf{\Delta_W}^2.$

Given the definitions of $\mumin$ and $L_{\max}$, we have:
\begin{align} \label{eqn:bound_AB}
(A) + \tfrac{1}{2} (B) &\ge \tfrac{\mu}{2}(A1) + \tfrac{1}{2L} (A2) - \tfrac{L}{2} (A3) - (A4) + \tfrac{\mu_g}{8}(B1) + \tfrac{1}{8L_g} (B2) - \tfrac{1}{4} (B3) \nonumber \\ 
&\stackrel{(i)}{\geq} \tfrac{\mumin}{8} \left( 4(A1) + (B1) \right) + \tfrac{1}{2L_{\max}} (A2) + \tfrac{1}{8L_{\max}} (B2) - (A4) - \tfrac{1}{4} (B3) - \tfrac{L}{2} (A3)  \nonumber \\
%&\ge \tfrac{\mumin}{8} \left( 4(A1) + (B1) \right) + \frac{1}{2L_f} (A2) + \frac{1}{8L_g} (B2) - (A4) - \frac{1}{4} (B3) - \frac{L}{2} (A3)  \nonumber \\
&\ge \tfrac{\mumin \cdot \sigma_r(\Wo)^2}{10} \normf{\Delta_W}^2
   + \tfrac{1}{2L_{\max}} \normf{\gradf(\factor)}^2
   + \tfrac{1}{8L_{\max}} \normf{\nabla g}^2 \nonumber \\
&\quad \quad \quad \quad \quad \quad \quad \quad \quad \quad - \normt{\gradf(\factor)} \normf{\Delta_W}^2 - \tfrac{1}{4} \normf{\nabla g} \normf{\Delta_W} \normf{\Delta_Y} \nonumber \\
&\quad \quad \quad \quad \quad \quad  \quad \quad \quad \quad - \tfrac{L}{2} \normf{\Xo - \Xoptr}^2
\end{align} 
where in $(i)$ we used the definitions of $\mumin$ and $L_{\max}$.
Note that we have not used the condition \eqref{eqn:condition}. 
It follows from \eqref{eqn:condition} that
\begin{align}
\normt{\gradf(\factor)} \cdot \normf{\Delta_W}^2
&\le \tfrac{\sigma_r(\Wo)}{10\sqrt{\kappa}} \normt{\gradf(\factor)} \cdot \normf{\Delta_W} \nonumber \\
&\le \tfrac{\mumin \cdot \sigma_r(\Wo)^2}{25} \normf{\Delta_W}^2
   + \tfrac{1}{4L_{\max}} \normt{\gradf(\factor)}^2 \label{eqn:bound_error1}
\end{align}
and
\begin{align}
\tfrac{1}{4} \normt{\nabla g} \cdot \normf{\Delta_W} \cdot \normf{\Delta_Y}
&= \tfrac{1}{4} \normt{\nabla g} \cdot \normf{\Delta_W}^2 \nonumber \\
&\le \tfrac{\sigma_r(\Wo)}{40\sqrt{\kappa}} \normt{\nabla g} \cdot \normf{\Delta_W} \nonumber \\
&\le \tfrac{\mumin \cdot \sigma_r(\Wo)^2}{100} \normf{\Delta_W}^2
   + \tfrac{1}{16L_{\max}} \normt{\nabla g}^2 \label{eqn:bound_error2}
\end{align}
where we use the AM-GM inequality. 
Plugging \eqref{eqn:bound_error1} and \eqref{eqn:bound_error2} in \eqref{eqn:bound_AB}, it is easy to obtain:
\begin{align*}
(A) + \tfrac{1}{2} (B)
&\ge \tfrac{\mumin \cdot \sigma_r(\Wo)^2}{20} \normf{\Delta_W}^2
   + \tfrac{1}{4L_{\max}} \normf{\gradf(\factor)}^2 + \tfrac{1}{16L_{\max}} \normf{\nabla g}^2
   - \tfrac{L}{2} \normf{\Xo - \Xoptr}^2.
\end{align*}  \hfill $\blacksquare$

\section{Appendix}
The proof follows the same framework of the sublinear convergence proof in \citep{bhojanapalli2016dropping}. We use the following general lemma to prove the sublinear converegence.
\begin{lemma} \label{lem:sublinear}
Suppose that a sequence of iterates $\{\W_t\}_{t=0}^T$ satisfies the following conditions
\begin{align}
f(\W_t \W_t^\top) - f(\W_{t+1} \W_{t+1}^\top) &\ge \alpha \cdot \normf{\nabla_W f(\W_t \W_t^\top)}^2, \label{eqn:sublinear_hypo1}\\
f(\W_t \W_t^\top) - f(\Wo {\Wo}^\top) &\le \beta \cdot \normf{\nabla_W f(\W_t \W_t^\top)} \label{eqn:sublinear_hypo2}
\end{align}
for all $t = 0,\ldots,T-1$ and some values $\alpha, \beta > 0$ independent of the iterates. Then it is guaranteed that
\begin{align*}
f(\W_T \W_T^\top) - f(\Wo {\Wo}^\top) \le \frac{\beta^2}{\alpha \cdot T}
\end{align*}
\end{lemma}
\begin{proof}
Define $\delta_t = f(\W_t \W_t^\top) - f(\Wo {\Wo}^\top)$. If we get $\delta_{T_0} \le 0$ at some $T_0 < T$, the desired inequality holds because the first hypothesis guarantees $\{\delta_t\}_{t=0}^T$ to be non-increasing. Hence, we can only consider the time $t$ where $\delta_t, \delta_{t+1} \ge 0$. We have 
\begin{align*}
\delta_{t+1}
&\stackrel{(a)}{\le} \delta_t - \alpha \cdot \normf{\nabla_W f(\W_t \W_t^\top)}^2 \\
&\stackrel{(b)}{\le} \delta_t - \frac{\alpha}{\beta^2} \cdot \delta_t^2 \\
&\stackrel{(c)}{\le} \delta_t - \frac{\alpha}{\beta^2} \cdot \delta_t \cdot \delta_{t+1}
\end{align*}
where (a) follows from the first hypothesis, (b) follows from the second hypothesis, (c) follows from that $\delta_{t+1} \le \delta_t$ by the first hypothesis. Dividing by $\delta_t \cdot \delta_{t+1}$, we obtain
\begin{align*}
\frac{1}{\delta_{t+1}} - \frac{1}{\delta_t} \ge \frac{\alpha}{\beta^2}
\end{align*}
Then we obtain the desired result by telescoping the above inequality.
\end{proof}
Now it suffices to show \algo provides a sequence $\{\W_t\}_{t=0}^T$ satisfies the hypotheses of Lemma \ref{lem:sublinear}. 

\textit{Obtaining \eqref{eqn:sublinear_hypo1}} Although $f$ is non-convex over the factor space, it is reasonable to obtain a new estimate (with a carefully chosen steplength) which is no worse than the current one, because the algorithm takes a gradient step.
\begin{lemma} \label{lem:improvement}
Let $f$ be a $L$-smooth convex function. Moreover, consider the recursion in 
Let $\X = \W\W^\top$ and $\Xp = \Wp {\Wp}^\top$ be two consecutive estimates of \algo. Then
\begin{align} \label{eqn:improvement}
f(\W \W^\top) - f(\Wp {\Wp}^\top) \ge \frac{3\eta}{5} \cdot \normf{\nabla_W f(WW^\top)}^2
\end{align}
\end{lemma}
Since we can fix the steplength $\eta$ based on the initial solution so that it is independent of the following iterates, we have obtained the first hypothesis of Lemma \ref{lem:sublinear}.

\textit{Obtaining \eqref{eqn:sublinear_hypo2}}
Consider the following assumption.
$$
\text{(A)}: \quad \dist(U,V;\Xoptr) = \min_{R \in O(r)} \normf{W - \Wo R} \le \frac{\sigma_r(\Wo)}{10}
$$
Trivially (A) holds for $U_0$ and $V_0$. Now we provide key lemmas, and then the convergence proof will be presented.

\begin{lemma}[Suboptimality bound] \label{lem:suboptimality}
Assume that (A) holds for $\W$. Then we have
\begin{align*}
f(WW^\top) - f(\Wo {\Wo}^\top)
&\le \frac{7}{3}
\cdot \normf{\nabla_W f(WW^\top)}
\cdot \dist(U,V;\Xoptr)
\end{align*}
\end{lemma}

\begin{lemma}[Descent in distance] \label{lem:approaching}
Assume that (A) holds for $W$. If 
$$
f(\Wp {\Wp}^\top) \ge f(\Wo {\Wo}^\top),
$$
then
\begin{align*}
\dist(\Up,\Vp;\Xoptr) \le \dist(\U,\V;\Xoptr)
\end{align*}
\end{lemma}

Combining the above two lemmas, we obtain
\begin{align} \label{eqn:suboptimality}
f(WW^\top) - f(\Wo {\Wo}^\top)
&\le \frac{7 \cdot \dist(\Uinit,\Vinit;\Xoptr)}{3}
\cdot \normf{\nabla_W f(WW^\top)}
\end{align}
Plugging \eqref{eqn:improvement} and \eqref{eqn:suboptimality} in Lemma \ref{lem:sublinear}, we obtain the desired result.  \hfill $\blacksquare$

%%%%%%%%%%%%%%%%%%%%%%%%%%%%%%%%%%%%%%%%%%%%%%%%%%%

\subsection{Proof of Lemma \ref{lem:improvement}}
The $L$-smoothness gives
\begin{small}
\begin{align} \label{eqn:improve0}
&f(WW^\top) - f(\Wp{\Wp}^\top) \nonumber \\
&\ge \innerp{\nabla f(WW^\top)}{WW^\top - \Wp{\Wp}^\top}
- \frac{L}{2} \normf{WW^\top - \Wp{\Wp}^\top}^2 \nonumber \\
&= \innerp{\nabla f(WW^\top)}{(W-\Wp) W^\top + W (W-\Wp)^\top}
 - \innerp{\nabla f(WW^\top)}{(W-\Wp) (W-\Wp)^\top} \nonumber \\
&\quad - \frac{L}{2} \normf{WW^\top - \Wp{\Wp}^\top}^2
\end{align}
\end{small}
For the first term, we have
\begin{align} \label{eqn:improve1}
\innerp{\nabla f(WW^\top)}{(W-\Wp) W^\top + W (W-\Wp)^\top}
&= 2 \cdot \innerp{\nabla f(WW^\top) W}{W-\Wp} \nonumber \\
&= \eta \cdot \normf{\nabla_W f(WW^\top)}^2
\end{align}
Using the Cauchy-Schwarz inequality, the second term can be bounded as follows.
\begin{small}
\begin{align} \label{eqn:improve2}
\innerp{\nabla f(WW^\top)}{(\W-\Wp)(\W-\Wp)^\top}
&= \eta^2 \cdot \innerp{\nabla f(WW^\top)}{\nabla_W f(WW^\top) \cdot \nabla_W f(WW^\top)^\top} \nonumber \\
&= \eta^2 \cdot \innerp{\nabla f(WW^\top) \cdot \nabla_W f(WW^\top)}{\nabla_W f(WW^\top)} \nonumber \\
&\le \eta^2 \cdot \normf{\nabla f(WW^\top) \cdot \nabla_W f(WW^\top)} \cdot \normf{\nabla_W f(WW^\top)} \nonumber \\
&\le \eta^2 \cdot \normt{\nabla f(WW^\top)} \cdot \normf{\nabla_W f(WW^\top)}^2
\end{align}
\end{small}
To bound the third term of \eqref{eqn:improve0}, we have
\begin{align} \label{eqn:improve3}
\| WW^\top - \Wp{\Wp}^\top \|_F 
&\le \| WW^\top - W{\Wp}^\top \|_F + \| W{\Wp}^\top - \Wp{\Wp}^\top \|_F \nonumber \\
&\le (\normt{W} + \normt{\Wp}) \cdot \| W - \Wp \|_F \nonumber \\
&\le \eta \cdot \left( 2\normt{W} + \eta \cdot \normt{\nabla f(WW^\top)} \cdot \normt{W} \right) \cdot \| \nabla_W f(WW^\top) \|_F \nonumber \\
&\le \frac{7\eta}{3} \normt{W} \cdot \| \nabla_W f(WW^\top) \|_F
\end{align}
Plugging \eqref{eqn:improve1}, \eqref{eqn:improve2}, and \eqref{eqn:improve3} to \eqref{eqn:improve0}, we obtain
\begin{align*}
f(WW^\top) - f(\Wp{\Wp}^\top)
&\ge \eta \cdot \normf{\gradW}^2 \cdot \left( 1 - \eta \frac{17L \|W\|_2^2 + 3 \|\nabla f(WW^\top)\|_2}{3} \right) \\
&\ge \frac{3\eta}{5} \cdot \normf{\gradW}^2
\end{align*}
where the last inequality follows from the condition of the steplength $\eta$. This completes the proof.
 \hfill $\blacksquare$
 
\subsection{Proof of Lemma \ref{lem:suboptimality}}
We use the following lemma.
\begin{lemma}[Error bound] \label{lem:errorterm}
Assume that (A) holds for $\W$. Then
\begin{align*}
\innerp{\nabla f(WW^\top)}{\Delta_W \Delta_W^\top}
&\le \frac{1}{3} \cdot \normf{\gradW} \cdot \dist(U,V;\Xoptr)
\end{align*}
\end{lemma}
Now the lemma is proved as follows.
\begin{small}
\begin{align*}
f(WW^\top) - f(\Wo{\Wo}^\top)
&\stackrel{(a)}{\le} \langle \nabla f(WW^\top), WW^\top - {\Wo}{\Wo}^\top \rangle \\
&= \langle \nabla f(WW^\top), \Delta_W W^\top \rangle + \langle \nabla f(WW^\top), W \Delta_W^\top \rangle
- \langle \nabla f(WW^\top), \Delta_W \Delta_W^\top \rangle \\
&= 2 \langle \nabla f(WW^\top)W, \Delta_W \rangle - \langle \nabla f(WW^\top), \Delta_W \Delta_W^\top \rangle \\
&\stackrel{(b)}{\le} 2 \cdot \normf{\nabla_W f(WW^\top)} \cdot \normf{\Delta_W}
+ | \langle \nabla f(WW^\top), \Delta_W \Delta_W^\top \rangle | \\
&\stackrel{(c)}{\le} \frac{7}{3} \cdot \normf{\nabla_W f(WW^\top)} \cdot \normf{\Delta_W}
\end{align*}
\end{small}
(a) follows from the convexity of $f$, (b) follows from the Cauchy-Schwarz inequality, and (c) follows from Lemma \ref{lem:errorterm}.  \hfill $\blacksquare$

%%%%%%%%%%%%%%%%%%%%%%%%%%%%%%%%%%%%%%%%%%%%%%%%%%%%%%%%%%%%%%%%%%%%%%%%%%%%%%%%%%%%%%%%%%%%%%%%%%%%%%%%%%%%
\subsection{Proof of Lemma \ref{lem:errorterm}}
Define $Q_W$, $Q_{\Wo}$, and $Q_{\Delta_W}$ as the projection matrices of the column spaces of $W$, $\Wo$, and $\Delta_W = W - \Wo R$, respectively. We have
\begin{align}
\innerp{\gradf(WW^\top)}{\Delta_W \Delta_W^\top}
&= \innerp{\gradf(WW^\top) Q_{\Delta_W}}{\Delta_W \Delta_W^\top} \nonumber \\
&\stackrel{(a)}{\le} \normt{\gradf(WW^\top) Q_{\Delta_W}} \cdot \normf{\Delta_W}^2 \nonumber \\
&\stackrel{(b)}{\le} \left( \normt{\gradf(WW^\top) Q_{W}} + \normt{\gradf(WW^\top) Q_{\Wo}} \right) \cdot \normf{\Delta_W}^2 \nonumber \\ \label{eqn:step2_errorbound}
\end{align}
where (a) follows from the Cauchy-Schwarz inequality and the fact $\normf{AB} \le \normt{A} \cdot \normf{B}$, and (b) follows from that $W - \Wo$ lies on the column space spanned by $W$ and $\Wo$. To bound the terms in \eqref{eqn:step2_errorbound}, we obtain 
\begin{small}
\begin{align*}
\normt{\gradf(WW^\top) Q_W}
&= \normt{\gradf(WW^\top) W W^\dagger}
\le \frac{1}{\sigma_r(W)} \normt{\gradf(WW^\top) W} \\
&\le \frac{10}{9\sigma_r(\Wo)} \normt{\gradf(WW^\top) W}, \\
\normt{\gradf(WW^\top) Q_{\Wo}}
&= \normt{\gradf(WW^\top) \Wo {\Wo}^{\dagger}}
\le \frac{1}{\sigma_r(\Wo)} \normt{\gradf(WW^\top) \Wo}, \\
\| \gradf(WW^\top) \Wo \|_2 &\le \| \gradf(WW^\top) W \|_2 + \| \gradf(WW^\top) \Delta_W \|_2 \\
&\le \| \gradf(WW^\top) W \|_2 + \left( \| \gradf(WW^\top) Q_W \|_2 + \| \gradf(WW^\top) Q_{\Wo} \|_2 \right) \cdot \normt{\Delta_W} \\
&\le \frac{10}{9} \| \gradf(WW^\top) W \|_2 + \frac{1}{10} \cdot \| \gradf(WW^\top) \Wo \|_2, \\
\| \gradf(WW^\top) Q_{\Wo} \|_2
&\le \frac{1}{\sigma_r(\Wo)} \| \gradf(WW^\top) \Wo \|_2
 \le \frac{5}{4\sigma_r(\Wo)} \| \gradf(WW^\top) W \|_2,
\end{align*}
\end{small}
where $W^{\dagger}$ and ${\Wo}^{\dagger}$ are the Moore-Penrose pseudoinverses of $W$ and $\Wo$. Plugging the above into \eqref{eqn:step2_errorbound}, we get
\begin{align*}
\innerp{\gradf(WW^\top)}{\Delta_W \Delta_W^\top}
&\le \frac{95}{36\sigma_r(\Wo)} \cdot \normt{\gradf(WW^\top) W} \cdot \normf{\Delta_W}^2 \\
&\stackrel{(a)}{\le} \frac{1}{3} \cdot \normt{\gradf(WW^\top) W} \cdot \normf{\Delta_W}
\end{align*}
where (a) follows from (A).  \hfill $\blacksquare$

\subsection{Proof of Lemma \ref{lem:approaching}}
For this proof, we borrow a lemma from \citep{bhojanapalli2016dropping}. Although the assumption for the lemma is stronger than Assumption (A), but a slight modification of the proof leads to the following lemma from Assumption (A).
\begin{lemma}[Lemma C.2 of \cite{bhojanapalli2016dropping}] \label{lem:smooth_00}
Let Assumption (A) hold and $f(W^+ {W^+}^\top) \ge f(\Wo {\Wo}^\top)$. Then the following lower bound holds:
\begin{align*}
\innerp{\nabla f(WW^\top)}{\Delta_W \Delta_W^\top} \ge - \frac{\sqrt{2}}{\sqrt{2}-\frac{1}{10}} \cdot \frac{1}{10} \cdot \left| \innerp{\nabla f(WW^\top)}{WW^\top - \Wo {\Wo}^\top} \right|.
\end{align*}
\end{lemma}

\newcommand{\errorUtilde}{U - \tilde{U}}
\newcommand{\errorVtilde}{V - \tilde{V}}
We have
\begin{align} \label{eqn:approaching1}
&\dist(U,V;\Xoptr)^2 - \dist(\Up,\Vp;\Xoptr)^2 \nonumber \\
&\ge \normf{W-\Wo R}^2 - \normf{\Wp - \Wo R}^2 \nonumber \\
&= 2 \eta \innerp{\nabla_W f(WW^\top)}{\Delta_W} - \eta^2 \normf{\nabla_W f(WW^\top)}^2 \nonumber \\
&= 4 \eta \innerp{\nabla f(WW^\top) W}{\Delta_W} - \eta^2 \normf{\nabla_W f(WW^\top)}^2 \nonumber \\
&= 2 \eta \innerp{\nabla f(WW^\top)}{WW^\top - \Wo {\Wo}^\top} + 2 \eta \innerp{\nabla f(WW^\top)}{\Delta_W \Delta_W^\top} - \eta^2 \normf{\nabla_W f(WW^\top)}^2 \nonumber \\
&\stackrel{(a)}{\ge} \frac{17\eta}{10} \innerp{\nabla f(WW^\top)}{WW^\top - \Wo {\Wo}^\top} - \eta^2 \normf{\nabla_W f(WW^\top)}^2 \nonumber \\
&\stackrel{(b)}{\ge} \frac{51\eta^2}{50} \normf{\nabla_W f(WW^\top)}^2 - \eta^2 \normf{\nabla_W f(WW^\top)}^2 \nonumber \\
&\ge 0
\end{align}
where (a) follows from Lemma \ref{lem:smooth_00}, (b) follows from the convexity of $f$, the hypothesis of the lemma, and Lemma \ref{lem:improvement} as follows.
\begin{align*}
\innerp{\nabla f(WW^\top)}{WW^\top - \Wo {\Wo}^\top}
&\ge f(WW^\top) - f(\Wo {\Wo}^\top) \\
&\ge f(WW^\top) - f(\W^+ {\W^+}^\top) \\
&\ge \frac{3\eta}{5} \cdot \normf{\nabla_W f(WW^\top)}^2
\end{align*}
This completes the proof.   \hfill $\blacksquare$

\section{Appendix}
The triangle inequality gives that
\begin{align} \label{eqn:init_bound0}
\normf{\factorinit - \Xoptr}
&\le \normf{\factorinit - \Xinit} + \normf{\Xinit - \Xopt} + \normf{\Xopt - \Xoptr}
\end{align}
Let us first obtain an upper bound on the first term. We have
\begin{align*}
\normf{\Xinit - \factorinit}
&= \norm{\begin{bmatrix} \sigma_{r+1}(\Xinit) \\ \vdots \\ \sigma_{\min\{m,n\}}(\Xinit) \end{bmatrix}} \\
&\stackrel{(a)}{\le} \norm{\begin{bmatrix} \sigma_{r+1}(\Xopt) \\ \vdots \\ \sigma_{\min\{m,n\}}(\Xopt) \end{bmatrix}}
  + \norm{\begin{bmatrix} \sigma_{r+1}(\Xinit) - \sigma_{r+1}(\Xopt) \\ \vdots \\ \sigma_{\min\{m,n\}}(\Xinit) - \sigma_{\min\{m,n\}}(\Xopt) \end{bmatrix}} \\
&= \normf{\Xopt - \Xoptr} + \sqrt{\sum_{i=r+1}^{\min\{m,n\}} (\sigma_i(\Xinit) - \sigma_i(\Xopt))^2} \\
&\stackrel{(b)}{\le} \normf{\Xopt - \Xoptr} + \normf{\Xinit - \Xopt}
\end{align*}
where (a) follows from the triangle inequality, and (b) is due to Mirsky's theorem \citep{mirsky}. Plugging this bound into \eqref{eqn:init_bound0}, we get
\begin{align} \label{eqn:init_bound1}
\normf{\factorinit - \Xoptr}
&\le 2\normf{\Xinit - \Xopt} + 2\normf{\Xopt - \Xoptr}
\end{align}
Now we bound the first term of \eqref{eqn:init_bound1}. We have
\begin{align*}
\normf{\Xinit}
&= \frac{1}{L} \normf{\gradf(0)}
 = \frac{1}{L} \normf{\gradf(0) - \gradf(\Xopt)}
\stackrel{(a)}{\le} \normf{0 - \Xopt}
 = \normf{\Xopt}, \\
L \innerp{\Xinit}{\Xopt}
&= - \innerp{\gradf(0)}{\Xopt}
\stackrel{(b)}{\ge} f(0) - f(\Xopt) + \frac{\mu}{2} \normf{\Xopt}^2
\stackrel{(c)}{\ge} \mu \normf{\Xopt}^2
\end{align*}
where (a) follows from the $L$-smoothness, (b) and (c) follow from the $\mu$-strong convexity. Then it follows that
\begin{align*}
\normf{\Xinit - \Xopt}^2
&= \normf{\Xinit}^2 + \normf{\Xopt}^2 - 2\innerp{\Xinit}{\Xopt}
\le 2(1 - \frac{\mu}{L})\normf{\Xopt}^2
\end{align*}
Applying this inequality to \eqref{eqn:init_bound1}, we get the desired inequality.  \hfill $\blacksquare$

\bibliography{NonSquareFGD}

\begin{thebibliography}{108}
\providecommand{\natexlab}[1]{#1}
\providecommand{\url}[1]{\texttt{#1}}
\expandafter\ifx\csname urlstyle\endcsname\relax
  \providecommand{\doi}[1]{doi: #1}\else
  \providecommand{\doi}{doi: \begingroup \urlstyle{rm}\Url}\fi

\bibitem[Aaronson(2007)]{aaronson2007learnability}
S.~Aaronson.
\newblock The learnability of quantum states.
\newblock In \emph{Proceedings of the Royal Society of London A: Mathematical,
  Physical and Engineering Sciences}, volume 463, pages 3089--3114. The Royal
  Society, 2007.

\bibitem[Agarwal et~al.(2010)Agarwal, Negahban, and
  Wainwright]{agarwal2010fast}
A.~Agarwal, S.~Negahban, and M.~Wainwright.
\newblock Fast global convergence rates of gradient methods for
  high-dimensional statistical recovery.
\newblock In \emph{Advances in Neural Information Processing Systems}, pages
  37--45, 2010.

\bibitem[Agrawal et~al.(2013)Agrawal, Gupta, Prabhu, and
  Varma]{agrawal2013multi}
R.~Agrawal, A.~Gupta, Y.~Prabhu, and M.~Varma.
\newblock Multi-label learning with millions of labels: {R}ecommending
  advertiser bid phrases for web pages.
\newblock In \emph{Proceedings of the 22nd international conference on World
  Wide Web}, pages 13--24. International World Wide Web Conferences Steering
  Committee, 2013.

\bibitem[Anandkumar and Ge(2016)]{anandkumar2016efficient}
A.~Anandkumar and R.~Ge.
\newblock Efficient approaches for escaping higher order saddle points in
  non-convex optimization.
\newblock In \emph{29th Annual Conference on Learning Theory}, pages 81--102,
  2016.

\bibitem[Andrews and Patterson~III(1976)]{andrews1976singular}
H.~Andrews and C.~Patterson~III.
\newblock Singular value decomposition ({SVD}) image coding.
\newblock \emph{Communications, IEEE Transactions on}, 24\penalty0
  (4):\penalty0 425--432, 1976.

\bibitem[Aravkin et~al.(2014)Aravkin, Kumar, Mansour, Recht, and
  Herrmann]{aravkin2014fast}
A.~Aravkin, R.~Kumar, H.~Mansour, B.~Recht, and F.~Herrmann.
\newblock Fast methods for denoising matrix completion formulations, with
  applications to robust seismic data interpolation.
\newblock \emph{SIAM Journal on Scientific Computing}, 36\penalty0
  (5):\penalty0 S237--S266, 2014.

\bibitem[Baglama and Reichel(2005)]{baglama2005augmented}
J.~Baglama and L.~Reichel.
\newblock Augmented implicitly restarted {L}anczos bidiagonalization methods.
\newblock \emph{SIAM Journal on Scientific Computing}, 27\penalty0
  (1):\penalty0 19--42, 2005.

\bibitem[Balakrishnan et~al.(2014)Balakrishnan, Wainwright, and
  Yu]{balakrishnan2014statistical}
S.~Balakrishnan, M.~Wainwright, and B.~Yu.
\newblock Statistical guarantees for the {EM} algorithm: {F}rom population to
  sample-based analysis.
\newblock \emph{arXiv preprint arXiv:1408.2156}, 2014.

\bibitem[Balzano et~al.(2010)Balzano, Nowak, and Recht]{balzano2010online}
L.~Balzano, R.~Nowak, and B.~Recht.
\newblock Online identification and tracking of subspaces from highly
  incomplete information.
\newblock In \emph{Communication, Control, and Computing (Allerton), 2010 48th
  Annual Allerton Conference on}, pages 704--711. IEEE, 2010.

\bibitem[Baraniuk(2007)]{baraniuk2007compressive}
R.~Baraniuk.
\newblock Compressive sensing.
\newblock \emph{IEEE signal processing magazine}, 24\penalty0 (4), 2007.

\bibitem[Becker et~al.(2011{\natexlab{a}})Becker, Bobin, and
  Cand{\`e}s]{becker2011nesta}
S.~Becker, J.~Bobin, and E.~Cand{\`e}s.
\newblock {NESTA}: {A} fast and accurate first-order method for sparse
  recovery.
\newblock \emph{SIAM Journal on Imaging Sciences}, 4\penalty0 (1):\penalty0
  1--39, 2011{\natexlab{a}}.

\bibitem[Becker et~al.(2011{\natexlab{b}})Becker, Cand{\`e}s, and
  Grant]{becker2011templates}
S.~Becker, E.~Cand{\`e}s, and M.~Grant.
\newblock Templates for convex cone problems with applications to sparse signal
  recovery.
\newblock \emph{Mathematical Programming Computation}, 3\penalty0 (3):\penalty0
  165--218, 2011{\natexlab{b}}.

\bibitem[Becker et~al.(2013)Becker, Cevher, and
  Kyrillidis]{becker2013randomized}
S.~Becker, V.~Cevher, and A.~Kyrillidis.
\newblock Randomized low-memory singular value projection.
\newblock In \emph{10th International Conference on Sampling Theory and
  Applications (Sampta)}, 2013.

\bibitem[Bennett and Lanning(2007)]{bennett2007netflix}
J.~Bennett and S.~Lanning.
\newblock The {N}etflix prize.
\newblock In \emph{Proceedings of KDD cup and workshop}, volume 2007, page~35,
  2007.

\bibitem[Bhatia et~al.(2015)Bhatia, Jain, Kar, Varma, and
  Jain]{bhatia2015sparse}
K.~Bhatia, H.~Jain, P.~Kar, M.~Varma, and P.~Jain.
\newblock Sparse local embeddings for extreme multi-label classification.
\newblock In \emph{Advances in Neural Information Processing Systems}, pages
  730--738, 2015.

\bibitem[Bhojanapalli et~al.(2016{\natexlab{a}})Bhojanapalli, Kyrillidis, and
  Sanghavi]{bhojanapalli2016dropping}
S.~Bhojanapalli, A.~Kyrillidis, and S.~Sanghavi.
\newblock Dropping convexity for faster semi-definite optimization.
\newblock In \emph{29th Annual Conference on Learning Theory}, pages 530--582,
  2016{\natexlab{a}}.

\bibitem[Bhojanapalli et~al.(2016{\natexlab{b}})Bhojanapalli, Neyshabur, and
  Srebro]{bhojanapalli2016global}
S.~Bhojanapalli, B.~Neyshabur, and N.~Srebro.
\newblock Global optimality of local search for low rank matrix recovery.
\newblock \emph{To appear in NIPS-16, arXiv preprint arXiv:1605.07221},
  2016{\natexlab{b}}.

\bibitem[Biswas et~al.(2006)Biswas, Liang, Toh, Ye, and
  Wang]{biswas2006semidefinite}
P.~Biswas, T.-C. Liang, K.-C. Toh, Y.~Ye, and T.-C. Wang.
\newblock Semidefinite programming approaches for sensor network localization
  with noisy distance measurements.
\newblock \emph{IEEE transactions on automation science and engineering},
  3\penalty0 (4):\penalty0 360, 2006.

\bibitem[Boumal and Absil(2011)]{boumal2011rtrmc}
N.~Boumal and P.-A. Absil.
\newblock {RTRMC}: {A} {R}iemannian trust-region method for low-rank matrix
  completion.
\newblock In \emph{Advances in neural information processing systems}, pages
  406--414, 2011.

\bibitem[Burer and Monteiro(2003)]{burer2003nonlinear}
S.~Burer and R.~Monteiro.
\newblock A nonlinear programming algorithm for solving semidefinite programs
  via low-rank factorization.
\newblock \emph{Mathematical Programming}, 95\penalty0 (2):\penalty0 329--357,
  2003.

\bibitem[Burer and Monteiro(2005)]{burer2005local}
S.~Burer and R.~Monteiro.
\newblock Local minima and convergence in low-rank semidefinite programming.
\newblock \emph{Mathematical Programming}, 103\penalty0 (3):\penalty0 427--444,
  2005.

\bibitem[Cai et~al.(2010)Cai, Cand{\`e}s, and Shen]{cai2010singular}
J.~Cai, E.~Cand{\`e}s, and Z.~Shen.
\newblock A singular value thresholding algorithm for matrix completion.
\newblock \emph{SIAM Journal on Optimization}, 20\penalty0 (4):\penalty0
  1956--1982, 2010.

\bibitem[Candes and Plan(2011)]{candes2011tight}
E.~Candes and Y.~Plan.
\newblock Tight oracle inequalities for low-rank matrix recovery from a minimal
  number of noisy random measurements.
\newblock \emph{Information Theory, IEEE Transactions on}, 57\penalty0
  (4):\penalty0 2342--2359, 2011.

\bibitem[Cand{\`e}s and Recht(2009)]{candes2009exact}
E.~Cand{\`e}s and B.~Recht.
\newblock Exact matrix completion via convex optimization.
\newblock \emph{Foundations of Computational mathematics}, 9\penalty0
  (6):\penalty0 717--772, 2009.

\bibitem[Candes et~al.(2011)Candes, Li, Ma, and Wright]{candes2011robust}
E.~Candes, X.~Li, Y.~Ma, and J.~Wright.
\newblock Robust principal component analysis?
\newblock \emph{Journal of the ACM (JACM)}, 58\penalty0 (3):\penalty0 11, 2011.

\bibitem[Candes et~al.(2015{\natexlab{a}})Candes, Eldar, Strohmer, and
  Voroninski]{candes2015phase}
E.~Candes, Y.~Eldar, T.~Strohmer, and V.~Voroninski.
\newblock Phase retrieval via matrix completion.
\newblock \emph{SIAM Review}, 57\penalty0 (2):\penalty0 225--251,
  2015{\natexlab{a}}.

\bibitem[Candes et~al.(2015{\natexlab{b}})Candes, Li, and
  Soltanolkotabi]{candes2015phase2}
E.~Candes, X.~Li, and M.~Soltanolkotabi.
\newblock Phase retrieval via {W}irtinger flow: {T}heory and algorithms.
\newblock \emph{Information Theory, IEEE Transactions on}, 61\penalty0
  (4):\penalty0 1985--2007, 2015{\natexlab{b}}.

\bibitem[Carneiro et~al.(2007)Carneiro, Chan, Moreno, and
  Vasconcelos]{carneiro2007supervised}
G.~Carneiro, A.~Chan, P.~Moreno, and N.~Vasconcelos.
\newblock Supervised learning of semantic classes for image annotation and
  retrieval.
\newblock \emph{Pattern Analysis and Machine Intelligence, IEEE Transactions
  on}, 29\penalty0 (3):\penalty0 394--410, 2007.

\bibitem[Chandrasekaran et~al.(2009)Chandrasekaran, Sanghavi, Parrilo, and
  Willsky]{chandrasekaran2009sparse}
V.~Chandrasekaran, S.~Sanghavi, P.~Parrilo, and A.~Willsky.
\newblock Sparse and low-rank matrix decompositions.
\newblock In \emph{Communication, Control, and Computing, 2009. Allerton 2009.
  47th Annual Allerton Conference on}, pages 962--967. IEEE, 2009.

\bibitem[Chen and Sanghavi(2010)]{chen2010general}
Y.~Chen and S.~Sanghavi.
\newblock A general framework for high-dimensional estimation in the presence
  of incoherence.
\newblock In \emph{Communication, Control, and Computing (Allerton), 2010 48th
  Annual Allerton Conference on}, pages 1570--1576. IEEE, 2010.

\bibitem[Chen and Wainwright(2015)]{chen2015fast}
Y.~Chen and M.~Wainwright.
\newblock Fast low-rank estimation by projected gradient descent: {G}eneral
  statistical and algorithmic guarantees.
\newblock \emph{arXiv preprint arXiv:1509.03025}, 2015.

\bibitem[Chen et~al.(2014)Chen, Bhojanapalli, Sanghavi, and
  Ward]{chen2014coherent}
Y.~Chen, S.~Bhojanapalli, S.~Sanghavi, and R.~Ward.
\newblock Coherent matrix completion.
\newblock In \emph{Proceedings of The 31st International Conference on Machine
  Learning}, pages 674--682, 2014.

\bibitem[Chiang et~al.(2014)Chiang, Hsieh, Natarajan, Dhillon, and
  Tewari]{chiang2014prediction}
K.-Y. Chiang, C.-J. Hsieh, N.~Natarajan, I.~Dhillon, and A.~Tewari.
\newblock Prediction and clustering in signed networks: {A} local to global
  perspective.
\newblock \emph{The Journal of Machine Learning Research}, 15\penalty0
  (1):\penalty0 1177--1213, 2014.

\bibitem[Christoffersson(1970)]{christoffersson1970one}
A.~Christoffersson.
\newblock \emph{The one component model with incomplete data}.
\newblock Uppsala., 1970.

\bibitem[Cohen et~al.(2015)Cohen, Nelson, and Woodruff]{cohen2015optimal}
M.~Cohen, J.~Nelson, and D.~Woodruff.
\newblock Optimal approximate matrix product in terms of stable rank.
\newblock \emph{arXiv preprint arXiv:1507.02268}, 2015.

\bibitem[Collins et~al.(2001)Collins, Dasgupta, and
  Schapire]{collins2001generalization}
M.~Collins, S.~Dasgupta, and R.~Schapire.
\newblock A generalization of principal components analysis to the exponential
  family.
\newblock In \emph{Advances in neural information processing systems}, pages
  617--624, 2001.

\bibitem[Davenport et~al.(2014)Davenport, Plan, van~den Berg, and
  Wootters]{davenport20141}
M.~Davenport, Y.~Plan, E.~van~den Berg, and M.~Wootters.
\newblock 1-bit matrix completion.
\newblock \emph{Information and Inference}, 3\penalty0 (3):\penalty0 189--223,
  2014.

\bibitem[DeCoste(2006)]{decoste2006collaborative}
D.~DeCoste.
\newblock Collaborative prediction using ensembles of maximum margin matrix
  factorizations.
\newblock In \emph{Proceedings of the 23rd international conference on Machine
  learning}, pages 249--256. ACM, 2006.

\bibitem[Donoho(2006)]{donoho2006compressed}
D.~Donoho.
\newblock Compressed sensing.
\newblock \emph{Information Theory, IEEE Transactions on}, 52\penalty0
  (4):\penalty0 1289--1306, 2006.

\bibitem[Drineas and Kannan(2001)]{drineas2001fast}
P.~Drineas and R.~Kannan.
\newblock Fast {M}onte-{C}arlo algorithms for approximate matrix
  multiplication.
\newblock In \emph{focs}, page 452. IEEE, 2001.

\bibitem[Drineas et~al.(2006)Drineas, Kannan, and Mahoney]{drineas2006fast}
P.~Drineas, R.~Kannan, and M.~Mahoney.
\newblock Fast {M}onte {C}arlo algorithms for matrices {I}: {A}pproximating
  matrix multiplication.
\newblock \emph{SIAM Journal on Computing}, 36\penalty0 (1):\penalty0 132--157,
  2006.

\bibitem[Duchi et~al.(2011)Duchi, Hazan, and Singer]{duchi2011adaptive}
J.~Duchi, E.~Hazan, and Y.~Singer.
\newblock Adaptive subgradient methods for online learning and stochastic
  optimization.
\newblock \emph{The Journal of Machine Learning Research}, 12:\penalty0
  2121--2159, 2011.

\bibitem[Fazel(2002)]{fazel2002matrix}
M.~Fazel.
\newblock \emph{Matrix rank minimization with applications}.
\newblock PhD thesis, PhD thesis, Stanford University, 2002.

\bibitem[Fazel et~al.(2004)Fazel, Hindi, and Boyd]{fazel2004rank}
M.~Fazel, H.~Hindi, and S.~Boyd.
\newblock Rank minimization and applications in system theory.
\newblock In \emph{American Control Conference, 2004. Proceedings of the 2004},
  volume~4, pages 3273--3278. IEEE, 2004.

\bibitem[Fazel et~al.(2008)Fazel, Candes, Recht, and
  Parrilo]{fazel2008compressed}
M.~Fazel, E.~Candes, B.~Recht, and P.~Parrilo.
\newblock Compressed sensing and robust recovery of low rank matrices.
\newblock In \emph{Signals, Systems and Computers, 2008 42nd Asilomar
  Conference on}, pages 1043--1047. IEEE, 2008.

\bibitem[Flammia et~al.(2012)Flammia, Gross, Liu, and
  Eisert]{flammia2012quantum}
S.~Flammia, D.~Gross, Y.-K. Liu, and J.~Eisert.
\newblock Quantum tomography via compressed sensing: {E}rror bounds, sample
  complexity and efficient estimators.
\newblock \emph{New Journal of Physics}, 14\penalty0 (9):\penalty0 095022,
  2012.

\bibitem[Ge et~al.(2016)Ge, Lee, and Ma]{ge2016matrix}
R.~Ge, J.~Lee, and T.~Ma.
\newblock Matrix completion has no spurious local minimum.
\newblock \emph{To appear in NIPS-16, arXiv preprint arXiv:1605.07272}, 2016.

\bibitem[Gross et~al.(2010)Gross, Liu, Flammia, Becker, and
  Eisert]{gross2010quantum}
D.~Gross, Y.-K. Liu, S.~Flammia, S.~Becker, and J.~Eisert.
\newblock Quantum state tomography via compressed sensing.
\newblock \emph{Physical review letters}, 105\penalty0 (15):\penalty0 150401,
  2010.

\bibitem[Gupta and Singh(2015)]{gupta2015collectively}
N.~Gupta and S.~Singh.
\newblock Collectively embedding multi-relational data for predicting user
  preferences.
\newblock \emph{arXiv preprint arXiv:1504.06165}, 2015.

\bibitem[Haeffele et~al.(2014)Haeffele, Young, and
  Vidal]{haeffele2014structured}
B.~Haeffele, E.~Young, and R.~Vidal.
\newblock Structured low-rank matrix factorization: {O}ptimality, algorithm,
  and applications to image processing.
\newblock In \emph{Proceedings of the 31st International Conference on Machine
  Learning (ICML-14)}, pages 2007--2015, 2014.

\bibitem[Halko et~al.(2011)Halko, Martinsson, and Tropp]{halko2011finding}
N.~Halko, P.-G. Martinsson, and J.~Tropp.
\newblock Finding structure with randomness: {P}robabilistic algorithms for
  constructing approximate matrix decompositions.
\newblock \emph{SIAM review}, 53\penalty0 (2):\penalty0 217--288, 2011.

\bibitem[Hardt and Wootters(2014)]{hardt2014fast}
M.~Hardt and M.~Wootters.
\newblock Fast matrix completion without the condition number.
\newblock In \emph{Proceedings of The 27th Conference on Learning Theory},
  pages 638--678, 2014.

\bibitem[Hazan(2008)]{hazan2008sparse}
E.~Hazan.
\newblock Sparse approximate solutions to semidefinite programs.
\newblock In \emph{LATIN 2008: Theoretical Informatics}, pages 306--316.
  Springer, 2008.

\bibitem[Jaggi and Sulovsk(2010)]{jaggi2010simple}
M.~Jaggi and M.~Sulovsk.
\newblock A simple algorithm for nuclear norm regularized problems.
\newblock In \emph{Proceedings of the 27th International Conference on Machine
  Learning (ICML-10)}, pages 471--478, 2010.

\bibitem[Jain et~al.(2010)Jain, Meka, and Dhillon]{jain2010guaranteed}
P.~Jain, R.~Meka, and I.~Dhillon.
\newblock Guaranteed rank minimization via singular value projection.
\newblock In \emph{Advances in Neural Information Processing Systems}, pages
  937--945, 2010.

\bibitem[Jain et~al.(2013)Jain, Netrapalli, and Sanghavi]{jain2013low}
P.~Jain, P.~Netrapalli, and S.~Sanghavi.
\newblock Low-rank matrix completion using alternating minimization.
\newblock In \emph{Proceedings of the 45th annual ACM symposium on Symposium on
  theory of computing}, pages 665--674. ACM, 2013.

\bibitem[Jain et~al.(2015)Jain, Jin, Kakade, and Netrapalli]{jain2015computing}
P.~Jain, C.~Jin, S.~Kakade, and P.~Netrapalli.
\newblock Computing matrix squareroot via non convex local search.
\newblock \emph{arXiv preprint arXiv:1507.05854}, 2015.

\bibitem[Javanmard and Montanari(2013)]{javanmard2013localization}
A.~Javanmard and A.~Montanari.
\newblock Localization from incomplete noisy distance measurements.
\newblock \emph{Foundations of Computational Mathematics}, 13\penalty0
  (3):\penalty0 297--345, 2013.

\bibitem[Jin et~al.(2016)Jin, Kakade, and Netrapalli]{jin2016provable}
C.~Jin, S.~Kakade, and P.~Netrapalli.
\newblock Provable efficient online matrix completion via non-convex stochastic
  gradient descent.
\newblock \emph{To appear in NIPS-16, arXiv preprint arXiv:1605.08370}, 2016.

\bibitem[Johnson(2014)]{johnson2014logistic}
C.~Johnson.
\newblock Logistic matrix factorization for implicit feedback data.
\newblock \emph{Advances in Neural Information Processing Systems}, 27, 2014.

\bibitem[Kalev et~al.(2015)Kalev, Kosut, and Deutsch]{kalev2015quantum}
A.~Kalev, R.~Kosut, and I.~Deutsch.
\newblock Quantum tomography protocols with positivity are compressed sensing
  protocols.
\newblock \emph{Nature partner journals (npj) Quantum Information}, 1:\penalty0
  15018, 2015.

\bibitem[Keshavan and Oh(2009)]{keshavan2009gradient}
Raghunandan~H Keshavan and Sewoong Oh.
\newblock A gradient descent algorithm on the grassman manifold for matrix
  completion.
\newblock \emph{arXiv preprint arXiv:0910.5260}, 2009.

\bibitem[Koren et~al.(2009)Koren, Bell, and Volinsky]{koren2009matrix}
Y.~Koren, R.~Bell, and C.~Volinsky.
\newblock Matrix factorization techniques for recommender systems.
\newblock \emph{Computer}, \penalty0 (8):\penalty0 30--37, 2009.

\bibitem[Krahmer and Ward(2011)]{krahmer2011new}
F.~Krahmer and R.~Ward.
\newblock New and improved {J}ohnson-{L}indenstrauss embeddings via the
  restricted isometry property.
\newblock \emph{SIAM Journal on Mathematical Analysis}, 43\penalty0
  (3):\penalty0 1269--1281, 2011.

\bibitem[Kyrillidis and Cevher(2014)]{kyrillidis2014matrix}
A.~Kyrillidis and V.~Cevher.
\newblock Matrix recipes for hard thresholding methods.
\newblock \emph{Journal of mathematical imaging and vision}, 48\penalty0
  (2):\penalty0 235--265, 2014.

\bibitem[Kyrillidis et~al.(2014)Kyrillidis, Vlachos, and
  Zouzias]{kyrillidis2014approximate}
A.~Kyrillidis, M.~Vlachos, and A.~Zouzias.
\newblock Approximate matrix multiplication with application to linear
  embeddings.
\newblock In \emph{Information Theory (ISIT), 2014 IEEE International Symposium
  on}, pages 2182--2186. Ieee, 2014.

\bibitem[Landgraf and Lee(2015)]{landgraf2015dimensionality}
A.~Landgraf and Y.~Lee.
\newblock Dimensionality reduction for binary data through the projection of
  natural parameters.
\newblock \emph{arXiv preprint arXiv:1510.06112}, 2015.

\bibitem[Larsen(2004)]{larsen2004propack}
R.~Larsen.
\newblock {PROPACK}-{S}oftware for large and sparse {SVD} calculations.
\newblock \emph{Available online. URL http://sun. stanford. edu/rmunk/PROPACK},
  pages 2008--2009, 2004.

\bibitem[Laue(2012)]{laue2012hybrid}
S.~Laue.
\newblock A hybrid algorithm for convex semidefinite optimization.
\newblock In \emph{Proceedings of the 29th International Conference on Machine
  Learning (ICML-12)}, pages 177--184, 2012.

\bibitem[Lee and Bresler(2010)]{lee2010admira}
K.~Lee and Y.~Bresler.
\newblock {ADMiRA}: {A}tomic decomposition for minimum rank approximation.
\newblock \emph{Information Theory, IEEE Transactions on}, 56\penalty0
  (9):\penalty0 4402--4416, 2010.

\bibitem[Lehoucq et~al.(1998)Lehoucq, Sorensen, and Yang]{lehoucq1998arpack}
R.~Lehoucq, D.~Sorensen, and C.~Yang.
\newblock \emph{{ARPACK} users' guide: solution of large-scale eigenvalue
  problems with implicitly restarted {A}rnoldi methods}, volume~6.
\newblock Siam, 1998.

\bibitem[Lin et~al.(2010)Lin, Chen, and Ma]{lin2010augmented}
Z.~Lin, M.~Chen, and Y.~Ma.
\newblock The augmented {L}agrange multiplier method for exact recovery of
  corrupted low-rank matrices.
\newblock \emph{arXiv preprint arXiv:1009.5055}, 2010.

\bibitem[Liu et~al.(2013)Liu, Wen, and Zhang]{liu2013limited}
X.~Liu, Z.~Wen, and Y.~Zhang.
\newblock Limited memory block {K}rylov subspace optimization for computing
  dominant singular value decompositions.
\newblock \emph{SIAM Journal on Scientific Computing}, 35\penalty0
  (3):\penalty0 A1641--A1668, 2013.

\bibitem[Liu et~al.(2016)Liu, Wu, Miao, Zhao, and Li]{liu2016neighborhood}
Y.~Liu, M.~Wu, C.~Miao, P.~Zhao, and X.-L. Li.
\newblock Neighborhood regularized logistic matrix factorization for
  drug-target interaction prediction.
\newblock \emph{PLoS Computational Biology}, 12\penalty0 (2):\penalty0
  e1004760, 2016.

\bibitem[Liu(2011)]{liu2011universal}
Y.-K. Liu.
\newblock Universal low-rank matrix recovery from {P}auli measurements.
\newblock In \emph{Advances in Neural Information Processing Systems}, pages
  1638--1646, 2011.

\bibitem[Liu and Vandenberghe(2009)]{liu2009interior}
Z.~Liu and L.~Vandenberghe.
\newblock Interior-point method for nuclear norm approximation with application
  to system identification.
\newblock \emph{SIAM Journal on Matrix Analysis and Applications}, 31\penalty0
  (3):\penalty0 1235--1256, 2009.

\bibitem[Lu et~al.(2005)Lu, Keles, Wright, and Wahba]{lu2005framework}
F.~Lu, S.~Keles, S.~Wright, and G.~Wahba.
\newblock Framework for kernel regularization with application to protein
  clustering.
\newblock \emph{Proceedings of the National Academy of Sciences of the United
  States of America}, 102\penalty0 (35):\penalty0 12332--12337, 2005.

\bibitem[Makadia et~al.(2008)Makadia, Pavlovic, and Kumar]{makadia2008new}
A.~Makadia, V.~Pavlovic, and S.~Kumar.
\newblock A new baseline for image annotation.
\newblock In \emph{Computer Vision--ECCV 2008}, pages 316--329. Springer, 2008.

\bibitem[Mirsky(1960)]{mirsky}
L.~Mirsky.
\newblock Symmetric gage functions and unitarily invariant norms.
\newblock \emph{Quarterly Journal of Mathematics}, 11:\penalty0 50--59, 1960.

\bibitem[Negahban and Wainwright(2012)]{negahban2012restricted}
S.~Negahban and M.~Wainwright.
\newblock Restricted strong convexity and weighted matrix completion: {O}ptimal
  bounds with noise.
\newblock \emph{The Journal of Machine Learning Research}, 13\penalty0
  (1):\penalty0 1665--1697, 2012.

\bibitem[Nesterov(2004)]{Nesterov}
Y.~Nesterov.
\newblock \emph{Introductory Lectures on Convex Optimization: A Basic Course}.
\newblock Springer, 2004.

\bibitem[Park et~al.(2016{\natexlab{a}})Park, Kyrillidis, Bhojanapalli,
  Caramanis, and Sanghavi]{park2016provable}
D.~Park, A.~Kyrillidis, S.~Bhojanapalli, C.~Caramanis, and S.~Sanghavi.
\newblock Provable {B}urer-{M}onteiro factorization for a class of
  norm-constrained matrix problems.
\newblock \emph{arXiv preprint arXiv:1606.01316}, 2016{\natexlab{a}}.

\bibitem[Park et~al.(2016{\natexlab{b}})Park, Kyrillidis, Caramanis, and
  Sanghavi]{park2016non}
D.~Park, A.~Kyrillidis, C.~Caramanis, and S.~Sanghavi.
\newblock Non-square matrix sensing without spurious local minima via the
  {B}urer-{M}onteiro approach.
\newblock \emph{arXiv preprint arXiv:1609.03240}, 2016{\natexlab{b}}.

\bibitem[Pearson(1901)]{pearson1901principal}
K.~Pearson.
\newblock Principal components analysis.
\newblock \emph{The London, Edinburgh, and Dublin Philosophical Magazine and
  Journal of Science}, 6\penalty0 (2):\penalty0 559, 1901.

\bibitem[Recht et~al.(2010)Recht, Fazel, and Parrilo]{recht2010guaranteed}
B.~Recht, M.~Fazel, and P.~Parrilo.
\newblock Guaranteed minimum-rank solutions of linear matrix equations via
  nuclear norm minimization.
\newblock \emph{SIAM review}, 52\penalty0 (3):\penalty0 471--501, 2010.

\bibitem[Rennie and Srebro(2005)]{rennie2005fast}
J.~Rennie and N.~Srebro.
\newblock Fast maximum margin matrix factorization for collaborative
  prediction.
\newblock In \emph{Proceedings of the 22nd international conference on Machine
  learning}, pages 713--719. ACM, 2005.

\bibitem[Ruhe(1974)]{ruhe1974numerical}
A.~Ruhe.
\newblock \emph{Numerical computation of principal components when several
  observations are missing}.
\newblock Univ., 1974.

\bibitem[Schein et~al.(2003)Schein, Saul, and Ungar]{schein2003logisticPCA}
A.~Schein, L.~Saul, and L.~Ungar.
\newblock A generalized linear model for principal component analysis of binary
  data.
\newblock In \emph{AISTATS}, 2003.

\bibitem[Srebro et~al.(2004)Srebro, Rennie, and Jaakkola]{srebro2004maximum}
N.~Srebro, J.~Rennie, and T.~Jaakkola.
\newblock Maximum-margin matrix factorization.
\newblock In \emph{Advances in neural information processing systems}, pages
  1329--1336, 2004.

\bibitem[Sun and Luo(2015)]{sun2015guaranteed}
R.~Sun and Z.-Q. Luo.
\newblock Guaranteed matrix completion via nonconvex factorization.
\newblock In \emph{{IEEE} 56th Annual Symposium on Foundations of Computer
  Science, {FOCS} 2015}, pages 270--289, 2015.

\bibitem[Tanner and Wei(2013)]{tanner2013normalized}
J.~Tanner and K.~Wei.
\newblock Normalized iterative hard thresholding for matrix completion.
\newblock \emph{SIAM Journal on Scientific Computing}, 35\penalty0
  (5):\penalty0 S104--S125, 2013.

\bibitem[Tipping and Bishop(1999)]{tipping1999probabilistic}
M.~Tipping and C.~Bishop.
\newblock Probabilistic principal component analysis.
\newblock \emph{Journal of the Royal Statistical Society: Series B (Statistical
  Methodology)}, 61\penalty0 (3):\penalty0 611--622, 1999.

\bibitem[Tu et~al.(2016)Tu, Boczar, Soltanolkotabi, and Recht]{tu2016low}
S.~Tu, R.~Boczar, M.~Soltanolkotabi, and B.~Recht.
\newblock Low-rank solutions of linear matrix equations via {P}rocrustes flow.
\newblock \emph{Proceedings of The 33rd International Conference on Machine
  Learning}, pages 964--973, 2016.

\bibitem[Ubaru and Saad(2016)]{ubaru2016fast}
S.~Ubaru and Y.~Saad.
\newblock Fast methods for estimating the numerical rank of large matrices.
\newblock In \emph{Proceedings of The 33rd International Conference on Machine
  Learning}, pages 468--477, 2016.

\bibitem[Verstrepen(2015)]{verstrepen2015collaborative}
K.~Verstrepen.
\newblock \emph{Collaborative Filtering with Binary, Positive-only Data}.
\newblock PhD thesis, University of Antwerpen, 2015.

\bibitem[Waldspurger et~al.(2015)Waldspurger, d’Aspremont, and
  Mallat]{waldspurger2015phase}
I.~Waldspurger, A.~d’Aspremont, and S.~Mallat.
\newblock Phase recovery, {M}ax{C}ut and complex semidefinite programming.
\newblock \emph{Mathematical Programming}, 149\penalty0 (1-2):\penalty0 47--81,
  2015.

\bibitem[Wang et~al.(2009)Wang, Yan, Zhang, and Zhang]{wang2009multi}
C.~Wang, S.~Yan, L.~Zhang, and H.-J. Zhang.
\newblock Multi-label sparse coding for automatic image annotation.
\newblock In \emph{Computer Vision and Pattern Recognition, 2009. CVPR 2009.
  IEEE Conference on}, pages 1643--1650. IEEE, 2009.

\bibitem[Waters et~al.(2011)Waters, Sankaranarayanan, and
  Baraniuk]{waters2011sparcs}
A.~Waters, A.~Sankaranarayanan, and R.~Baraniuk.
\newblock Spa{RCS}: {R}ecovering low-rank and sparse matrices from compressive
  measurements.
\newblock In \emph{Advances in neural information processing systems}, pages
  1089--1097, 2011.

\bibitem[Weinberger et~al.(2007)Weinberger, Sha, Zhu, and
  Saul]{weinberger2007graph}
K.~Weinberger, F.~Sha, Q.~Zhu, and L.~Saul.
\newblock Graph {L}aplacian regularization for large-scale semidefinite
  programming.
\newblock In \emph{Advances in Neural Information Processing Systems}, pages
  1489--1496, 2007.

\bibitem[Wen et~al.(2012)Wen, Yin, and Zhang]{wen2012solving}
Z.~Wen, W.~Yin, and Y.~Zhang.
\newblock Solving a low-rank factorization model for matrix completion by a
  nonlinear successive over-relaxation algorithm.
\newblock \emph{Mathematical Programming Computation}, 4\penalty0 (4):\penalty0
  333--361, 2012.

\bibitem[Weston et~al.(2011)Weston, Bengio, and Usunier]{weston2011wsabie}
J.~Weston, S.~Bengio, and N.~Usunier.
\newblock {WSABIE}: {S}caling up to large vocabulary image annotation.
\newblock In \emph{IJCAI}, volume~11, pages 2764--2770, 2011.

\bibitem[Wold and Lyttkens(1969)]{wold1969nonlinear}
H.~Wold and E.~Lyttkens.
\newblock Nonlinear iterative partial least squares ({NIPALS}) estimation
  procedures.
\newblock \emph{Bulletin of the International Statistical Institute},
  43\penalty0 (1), 1969.

\bibitem[Wootters(2016)]{mary2016personal}
M.~Wootters.
\newblock Personal communication.
\newblock 2016.

\bibitem[Yi et~al.(2016)Yi, Park, Chen, and Caramanis]{yi2016rpca}
X.~Yi, D.~Park, Y.~Chen, and C.~Caramanis.
\newblock Fast algorithms for robust {PCA} via gradient descent.
\newblock \emph{To appear in NIPS-16, arXiv preprint arXiv:1605.07784}, 2016.

\bibitem[Yurtsever et~al.(2015)Yurtsever, Tran-Dinh, and
  Cevher]{yurtsever2015universal}
A.~Yurtsever, Q.~Tran-Dinh, and V.~Cevher.
\newblock A universal primal-dual convex optimization framework.
\newblock In \emph{Advances in Neural Information Processing Systems 28}, pages
  3132--3140. 2015.

\bibitem[Zhao et~al.(2015)Zhao, Wang, and Liu]{zhao2015nonconvex}
T.~Zhao, Z.~Wang, and H.~Liu.
\newblock A nonconvex optimization framework for low rank matrix estimation.
\newblock In \emph{Advances in Neural Information Processing Systems 28}, pages
  559--567. 2015.

\bibitem[Zheng and Lafferty(2015)]{zheng2015convergent}
Q.~Zheng and J.~Lafferty.
\newblock A convergent gradient descent algorithm for rank minimization and
  semidefinite programming from random linear measurements.
\newblock In \emph{Advances in Neural Information Processing Systems}, pages
  109--117, 2015.

\bibitem[Zheng and Lafferty(2016)]{zheng2016convergent}
Q.~Zheng and J.~Lafferty.
\newblock Convergence analysis for rectangular matrix completion using
  burer-monteiro factorization and gradient descent.
\newblock \emph{arXiv preprint arXiv:1605.07051}, 2016.

\end{thebibliography}
\bibliographystyle{plainnat}

\end{document}